\documentclass[reqno, 11pt]{article}

\usepackage{lineno}

\usepackage{amsmath, amssymb, amsthm, amsfonts, txfonts, pxfonts, amsbsy}
\usepackage[english]{babel}
\usepackage{a4wide}
\usepackage{bbm}
\usepackage{enumerate}
\usepackage[all]{xypic}
\usepackage{mathtools}
\usepackage{hyperref}
\usepackage{mathrsfs}
\numberwithin{equation}{section}

\usepackage{color}

\theoremstyle{plain}
\newtheorem{thm}{Theorem}[section]
\newtheorem{lem}[thm]{Lemma}
\newtheorem{prop}[thm]{Proposition}

\newtheorem{cor}[thm]{Corollary}

\theoremstyle{definition}
\newtheorem{defi}[thm]{Definition}

\theoremstyle{remark}
\newtheorem{ex}[thm]{Example}
\newtheorem{Q}{Question/Comment}
\newtheorem{rem}[thm]{Remark}

\newcommand{\sk}{\vskip .4cm}
\newcommand{\nn}{\nonumber}
\newcommand{\ot}{\otimes}
\newcommand{\be}{\begin{equation}}
\newcommand{\ee}{\end{equation}}
\newcommand{\ra}{\rightarrow}
\newcommand{\tra}{\triangleright}
\newcommand{\trl}{\triangleleft}
\newcommand{\id}{\mathrm{id}}
\newcommand{\U}{\mathcal{U}}
\newcommand{\A}{\mathcal{A}}
\renewcommand{\Q}{\mathfrak{G}}

\newcommand{\IR}{\mathbb{R}}
\newcommand{\can}{\chi}
\newcommand{\da}{\rho}
\newcommand{\bbK}{\mathbb{K}}
\renewcommand{\1}{1}

\newcommand{\F}{\mathcal{F}}
\newcommand{\f}{\texttt{f}}
\newcommand{\tg}{\texttt{g}}

\newcommand{\ofu}[1]{\f^{-#1}}
\newcommand{\ofd}[1]{\f_{-#1}}

\renewcommand{\cot}{\gamma}
\newcommand{\co}[2]{\cot\left({#1}\ot{#2}\right)}
\newcommand{\coin}[2]{\bar\cot\left({#1}\ot{#2}\right)}
\newcommand{\mt}{\cdot_\cot}
\newcommand{\mtco}{\bullet_\cot}
\newcommand{\sg}{\sigma}
\newcommand{\sig}[2]{\sigma\left( {#1}\ot{#2}\right)}
\newcommand{\taug}[2]{\tau\left( {#1}\ot{#2}\right)}

\newcommand{\cog}[2]{\widetilde{\cot}\left({#1}\ot{#2}\right)}
\newcommand{\coing}[2]{\overline{\widetilde{\cot}}\left({#1}\ot{#2}\right)}

\newcommand{\hg}{H_\cot}
\newcommand{\pg}{A_\cot}

\newcommand{\pgls}{{_\sg A}}
\newcommand{\bgls}{{_\sg B}}

\newcommand{\mtcols}{~{_\sg\bullet}}
\newcommand{\col}{\varphi^\ell}

\newcommand{\zero}[1]{{#1}_{\scriptscriptstyle{(0)}}}
\newcommand{\one}[1]{{#1}_{\scriptscriptstyle{(1)}}}
\newcommand{\two}[1]{{#1}_{\scriptscriptstyle{(2)}}}

\newcommand{\three}[1]{{#1}_{\scriptscriptstyle{(3)}}}
\newcommand{\four}[1]{{#1}_{\scriptscriptstyle{(4)}}}
\newcommand{\five}[1]{{#1}_{\scriptscriptstyle{(5)}}}

\newcommand{\mzero}[1]{{#1}_{\scriptscriptstyle{(0)}}}
\newcommand{\mone}[1]{{#1}_{\scriptscriptstyle{(-1)}}}
\newcommand{\mtwo}[1]{{#1}_{\scriptscriptstyle{(-2)}}}
\newcommand{\mthree}[1]{{#1}_{\scriptscriptstyle{(-3)}}}
\newcommand{\mfour}[1]{{#1}_{\scriptscriptstyle{(-4)}}}

\newcommand{\C}{\mathbb C}

\newcommand{\LL}{{L}}

\newcommand{\fS}{{\mathfrak{S}}{}} 

\newcommand{\fL}{{\mathfrak{l}}{}}

\newcommand{\Ad}{\mathrm{Ad}}

\newcommand{\KMH}{{}^K{\!\cal M}{}^{H}}

\newcommand{\ft}{\mathcal{A}}
\newcommand{\sft}{\mathscr{A}}

\newcommand{\ocn}{\mathcal{O}(\alpha, \beta, x, c_{\scriptscriptstyle{N}}^{\pm 1})}
\newcommand{\ocs}{\mathcal{O}(\alpha, \beta, x, c_{\scriptscriptstyle{S}}^{\pm 1})}
\newcommand{\ocns}{\mathcal{O}(\alpha, \beta, x, c_{\scriptscriptstyle{N}}^{\pm 1}, c_{\scriptscriptstyle{S}}^{\pm 1})}

\title{\bf Noncommutative principal bundles \\
through twist deformation\\[10pt]}
\date{May 2016}
\author{
{\bf Paolo Aschieri}$\strut^{1}$, {\bf Pierre Bieliavsky}$\strut^{2}$, {\bf Chiara Pagani}$\strut^{3}$, {\bf Alexander Schenkel}$\strut^{4}$
\\[5pt]
$\strut^{1}$ {\small Dipartimento di Scienze e Innovazione Tecnologica}\\
 {\small and INFN Torino, }\\
{\small Universit{\`a} del Piemonte Orientale, 
Viale T.~Michel~11,~15121~Alessandria,~Italy.} \\ {\small {(e-mail: \texttt{aschieri@to.infn.it}})}
\\[5pt]
$\strut^{2}$
{\small  Institut de recherche en math\'ematique et physique, Universit\'e de Louvain,}\\
{\small Chemin du Cyclotron 2 bte L7.01.02, 1348 Louvain-la-Neuve, Belgium.} \\ {\small {(e-mail: \texttt{pierre.bieliavsky@uclouvain.be}})}
\\[5pt]
$\strut^{3}$
{\small  Mathematisches Institut,
Georg-August-Universit\"at G\"ottingen,}\\
{\small Bunsenstra\ss e 3-5,
37073 G\"ottingen,
Germany.} \\ {\small {(e-mail: \texttt{chiara.pagani@mathematik.uni-goettingen.de}})}
\\[5pt]
$\strut^{4}$
{\small Fakult\"at f\"ur Mathematik, Universit\"at Regensburg, 93040 Regensburg, Germany.}\\ 
{\small {(e-mail: \texttt{aschenkel83@gmail.com}})}
}

\begin{document}

\maketitle

\begin{abstract}
We construct noncommutative principal bundles deforming principal bundles with a Drinfeld twist (2-cocycle). If the twist is associated with the structure group then we have a deformation of the fibers. If the twist is associated with the automorphism group of the principal bundle, then we obtain noncommutative deformations of the base space as well. Combining the two twist deformations we obtain noncommutative principal bundles with both noncommutative fibers and base space. More in general, the natural isomorphisms proving the equivalence of a closed monoidal category of modules and its twist related one are used to obtain new Hopf-Galois extensions as twists of Hopf-Galois extensions. A sheaf approach is also considered, and examples presented.
\end{abstract}



\newpage

\tableofcontents


\section{Introduction}
Given the relevance of Lie groups and principal bundles, the noncommutative analogues of these structures have been
studied since the early days of noncommutative (NC) geometry, first
examples being quantum groups and their coset spaces.
The algebraic structure underlying NC principal bundles is that of
Hopf-Galois extension, the structure group $G$ being replaced by a Hopf algebra
$H$ (e.g.\ that of functions on $G$, or more in general a
neither commutative nor cocommutative Hopf algebra). 
Presently, in the literature, there are many examples of NC principal
bundles, most of them can be understood as deformation quantization of
classical principal bundles, see e.g.\ \cite{BrMaj, gs}.   Our NC geometry study of principal bundles
is in this deformation quantization context, and specifically Drinfeld twist (or 2-cocycle)
deformations \cite{Dri83, doi}.    We provide a general theory where both
the base space and the fibers are deformed, this allows to recover
previously studied examples as particular cases, including a wide
class of NC principal bundles
on quantum coset spaces, as well as the NC instanton bundle on
the $\theta$-sphere $S^4_\theta$ \cite{cl, gw, gw07}.

Drinfeld twist deformation is indeed a powerful method.
It applies to any  algebra $A$ that carries an action of a group (more in
general a coaction of a Hopf algebra $K$). Given a twist on the group one first deforms the
group in a quantum group and then canonically induces via its action a
deformation of the algebra. Similarly, modules over algebras are
twisted  into modules over deformed algebras, in particular into NC vector bundles.
This program has been successfully extended in \cite{AS} to the differential geometry 
of NC vector bundles. It has led  to a theory of arbitrary (i.e., not
necessary equivariant) connections  on bimodules and on their tensor
products that generalizes the notion of bimodule connection introduced
in \cite{Mou, DV}. The construction is categorical, and in particular
commutative connections can be canonically quantized to NC
connections. As sharpened in \cite{BSS, BSS2} the categorical setting is that of closed
monoidal categories.

In the next level of complexity after algebras and (bi)modules we find
$(A,H)$-relative Hopf modules, i.e.\ (bi)modules with respect to an algebra $A$ and
comodules with respect to an Hopf algebra $H$ (in particular we will
be concerned with the example $A\otimes H$, that in the commutative
case corresponds to the algebra of functions on the total space of a principal bundle tensored
that on the structure group). 
They are the first objects of our interest because  principality of
NC bundles is bijectivity of a map (the so-called canonical map)
between relative Hopf modules.

In this paper we thus first deform the category of $(A,H)$-relative Hopf modules by
considering a twist associated with $H$ itself 
(this is the special
degenerate case where $K=H$). Next we consider the case where there is a different Hopf algebra $K$ that coacts on $A$ and on $(A,H)$-relative Hopf modules,
and study how to twist deform this category using twists on $K$, and
then using twists both on $K$ and on $H$.
Studying this category we
are canonically led to twist deform classical principal bundles into  NC
principal bundles and more generally to prove that NC principal
bundles are twisted  into new NC principal bundles. 
The key point is to relate the canonical map between the twisted
modules to the initial canonical map, so to deduce bijectivity of the
first from bijectivity of the second.
This is achieved via a set of isomorphisms that are explicitly
constructed and have a categorical interpretation as components of
natural isomorphisms. 

Considering a  twist on the ``structure group''
$H$ leads to a deformation of the fibers of the principal bundle; this result was also obtained in  \cite{MS05} with a different proof
that disregards the natural categorical setting we are advocating. 
Considering a twist 
on ``the external symmetry Hopf algebra'' $K$ (classically associated
with a subgroup of the automorphism group of the bundle) leads to a
deformation of the base space. 
Combining  twists on $H$ and on $K$ 
we obtain deformations of both the fibers and the base space.

The categorical context of relative Hopf modules and of their twist
deformations we set up  is furthermore used in order to prove that principal $H$-comodule algebras
(i.e.\ Hopf-Galois extensions that admit the construction of associated vector bundles) are deformed into principal $H$-comodule algebras (Corollary
\ref{cor2:diagr-can2}), here principality is captured by a linear map that is
not in general $K$-equivariant and that has to be properly deformed.
This deformation is explicitly given and shown to be related to the natural isomorphism proving the
equivalence of the categories of Hopf algebra modules and of twisted Hopf algebra modules
as closed monoidal categories.
This same categorical context is relevant for planned further
investigations in the  geometry of NC principal bundles, in particular in the notion of gauge group and of
principal connection. Indeed both gauge transformations and
connections, as is the case for connections on NC vector bundles \cite{AS},
will not in general be $K$-equivariant maps.

Complementing the global description of principal $G$-bundles as
$G$-manifolds with extra properties, there is the important local description
based on trivial principal bundles  and on transition functions. 
We therefore also present a local theory of twists
deformations of NC principal
bundles, based on a sheaf theoretic approach that complements the
initial global approach. The explicit example of the 
$\theta$-sphere $S^4_\theta$ is detailed.

Finally we observe that  the present study is mainly algebraic so that the examples treated are
either in the context of formal deformation quantization, using Fr{\'e}chet Hopf-Galois
extensions on the ring $\mathbb{C}[[\hbar]]$ (cf.\ the main Example \ref{FDQ}), or obtained
via abelian Drinfeld twists associated with tori actions on algebraic
varieties. However these latter NC algebras can be completed to
$C^*$-algebras by the general deformation construction of
Rieffel \cite{R}; furthermore, also deformations of smooth manifolds
based on nonabelian Drinfeld twists can be constructed nonformally \cite{BG}.
It is then promising to combine these nonformal deformation
techniques with the algebraic and categorical ones here
developed in order to consider nonformal
deformations of principal bundles. This is even more so because, contrary to the well established theory of NC vector bundles
(consider for example finite projective $C^*$-modules over $C^*$-algebras),
a general characterization of NC principal bundles beyond the algebraic level and
in terms of NC topology is still missing.
In particular we are interested in the wide class of nonformal NC principal bundles that could be obtained via  twists based on an external symmetry Hopf
algebra $K$. The present paper is also motivated by this program and
can be seen as the first step toward its accomplishment.
\sk

The outline of the paper is the following: in \S \ref{sec:HG} we recall the basic
definitions and results about Hopf-Galois extensions, while in \S \ref{sec:twists}
we review some results  from the theory of  
deformations  of Hopf algebras and comodule algebras by 2-cocycles and extend them to 
the category of relative Hopf-modules, relevant to our study. 
The main results of the present paper are contained in
\S \ref{sec:twHG}: in three successive subsections we study
deformations of $H$-Hopf-Galois extensions by 2-cocycles on the
structure group $H$ (\S \ref{sec:def_sg}), on an external Hopf algebra
$K$ of symmetries (\S \ref{sec:def_es}), and the combination of these deformations (\S \ref{sec:combidef}).
In \S \ref{appsect4} we apply  the theory developed to
deformations of  quantum homogeneous spaces (\S \ref{sec:qhom}) and
 to encompass sheaves of Hopf-Galois extensions (\S
\ref{sec:sheaves}), providing also examples.
Appendix \ref{app:twists} reviews the close relationship
between the theory of 2-cocycles and that of Drinfeld twists,
and Appendix \ref{appC} clarifies the relationship between one of our 
deformation maps (the $\Q$-map) and the natural transformation
which establishes that twisting may be regarded as an equivalence of closed monoidal categories.
Appendix \ref{appB}  presents a complementary study of the twisted sheaf
describing the instanton bundle on $S^4_\theta$. 

\sk
\paragraph{Acknowledgments:}
We would like to thank Tomasz Brzezi\'nski, Lucio Cirio, Francesco D'Andrea,
Rita Fioresi, Giovanni Landi, Zoran \v{S}koda and Stefan Waldmann for useful
comments and discussions.
The research of P.B. and (in the initial stages of the project) C.P.  was supported by the Belgian Scientific Policy under IAP grant DYGEST.
A.S.\  was supported by a research fellowship of the Deutsche
Forschungsgemeinschaft (DFG, Germany).  The authors are members of 
the COST Action MP1405 QSPACE, supported by COST (European
Cooperation in Science and Technology). P.A.\ is affiliated to INdAM,
GNFM (Istituto Nazionale di Alta Matematica, Gruppo Nazionale di
Fisica Matematica).

\section{Background}
\paragraph{Notation:}
We work in the category of $\bbK$-modules, for $\bbK$ a fixed commutative ring 
with unit $\1_\bbK$. 
We denote the tensor product over $\bbK$ just  by $\otimes$. Morphisms of $\bbK$-modules  are simply called $\bbK$-linear maps.
In the following, all algebras  are over $\bbK$ and assumed to be
unital and associative. 
The product in an algebra $A$ is denoted by $m_A: A \ot A \ra A$,  $a \ot b \mapsto ab$ and 
 the unit map by $\eta_A: \bbK \ra A$, with $\1_A:= \eta_A(\1_\bbK)$ the unit element.   
Analogously all coalgebras  are assumed to be over $\bbK$,  counital   and coassociative. 
We denote the  coproduct and counit of a coalgebra $C$ by $\Delta_C: C \ra C \ot C$ and  $\varepsilon_C: C \ra \bbK$ respectively. 
We use the standard Sweedler notation for the coproduct:  $\Delta_C(c)= \one{c} \ot \two{c}$ (sum understood), for all $ c \in C$, 
and for  iterations of it $\Delta_C^n=(\id \ot  \Delta_C) \circ\Delta_C^{n-1}: c \mapsto \one{c}\ot \two{c} \ot 
\cdots \ot c_{\scriptscriptstyle{(n+1)\;}}$, $n >1$.
We denote by $*$ the convolution product in the  dual $\bbK$-module
$C':=\mathrm{Hom}(C,\bbK)$, $(f * f') (c):=f(\one{c})f'(\two{c})$, for all $c \in C$, $f,f' \in C'$.
Finally, for a Hopf algebra $H$, we denote by $S_H: H \ra H$ its antipode. 
For all maps mentioned above we will omit the subscripts referring to
the co/algebras involved when no risk of confusion can occur.  
Many of the examples presented will concern co/algebras
 equipped with an antilinear involution ($*$-structure); we will assume all maps therein to be compatible with the $*$-structure.
To indicate an  object $V$ in a  category $\mathcal{C}$ we frequently simply write $V \in \mathcal{C}$.
Finally, all monoidal categories appearing in this paper will have a trivial associator, hence we can unambiguously write $V_1\otimes V_2\otimes \cdots \otimes V_n$ for the tensor product of $n$ objects.

\subsection{Hopf-Galois extensions} \label{sec:HG}
We briefly collect the algebraic preliminaries on Hopf-Galois extensions
required for our work. 
Let $H$ be a bialgebra (or just a coalgebra).
A right $H$-\textit{comodule} is a $\bbK$-module $V$ with a $\bbK$-linear map
$\delta^V:V\to V\otimes H$ (called a right $H$-coaction) such that 
\be \label{eqn:Hcomodule}
(\id\otimes \Delta)\circ \delta^V = (\delta^V\otimes \id)\circ \delta^V~,\quad 
(\id\otimes \varepsilon) \circ \delta^V =\id~. 
\ee
The coaction on an element
$v\in V$ is written in Sweedler notation as $\delta^V(v) = \zero{v}\otimes \one{v}$ (sum understood). The right $H$-comodule properties
(\ref{eqn:Hcomodule}) then read as, for all $v\in V$,
\begin{eqnarray}
\zero{v} \otimes \one{(\one{v})}\otimes \two{(\one{v})} &=& \zero{(\zero{v})}\otimes \one{(\zero{v})} \otimes \one{v}=: \zero{v} \ot\one{v} \ot \two{v} ~, \\
\zero{v} \,\varepsilon (\one{v}) &=& v~.\nn
\end{eqnarray}
We denote by ${\cal M}^H$ the category of right $H$-comodules, with the obvious
definition of right $H$-comodule morphisms:  a morphism
 between  $V, W \in  {\cal M}^H$ is a $\bbK$-linear map $\psi:V\to W$ which satisfies
$\delta^W\circ \psi=(\psi\otimes\id)\circ \delta^V$ ($H$-equivariance condition). 
If $H$ is a bialgebra, then ${\cal M}^H$ is a monoidal category:  given
$V,W\in {\cal M}^H$, then the tensor product
$V\otimes W$ is an object in $ {\cal M}^H$ with  the right $H$-coaction  
\begin{eqnarray}\label{deltaVW}
 \delta^{V\otimes W} :V\otimes W & \longrightarrow & V\otimes W\otimes H~,\\
 v\otimes w & \longmapsto & \zero{v}\otimes \zero{w} \otimes 
 \one{v}\one{w} ~.\nn
\end{eqnarray}
The unit object in ${\cal M}^H$ is $\bbK$ together with the coaction
given by the unit map  of $H$, i.e.,
 $\delta^{\bbK}:=\eta_H:\bbK\to \bbK\otimes H \simeq H$.
\sk

If $A$ is a right $H$-comodule and also an algebra it is natural to require  the
additional structures of product $m_A: A\otimes A\to A$ and
unit $\eta_A: \bbK \to A$ to be morphisms in the category ${\cal M}^H$ (with $A \ot
A\in {\cal M}^H $ via  $\delta^{A\otimes A} $ as above). Explicitly, 
a right $H$-\textit{comodule algebra} $A$ is an algebra  which is 
a right $H$-comodule and such that
\be\label{eqn:Hcomodulealgebra}
\delta^A \circ m_A = (m_{A}\otimes \id) \circ \delta^{A\otimes A}~,\quad
\delta^A\circ \eta_A = (\eta_A\otimes \id)\circ\delta^\bbK~.
\ee
This is equivalent to require the coaction $\delta^A: A\to A\otimes H$
to be a morphism of unital algebras (where $A\otimes H$ has the tensor
product algebra structure),  for all $a,a^\prime\in A \, $,
\be
\delta^A(a\,a^\prime) =\delta^A(a)\,\delta^A(a')~~\;,\quad
\delta^A(\1_A) = \1_A\otimes \1_H~\; .
\ee
A morphism between two right $H$-comodule algebras is a morphism in ${\cal M}^H$
which preserves products and units. We shall denote by ${\cal A}^H$ the category of
right $H$-comodule algebras.
\sk

If $V$ is a right $H$-comodule and also a  left $A$-module, where $A\in {\cal A}^H$,
it is natural to require the $A$-module structure 
(or $A$-action) $\,\tra_V: A \ot V \ra V, ~a \ot v \mapsto a \tra_V v$
to be a morphism
in the category ${\cal M}^H$, i.e.\ $\delta^V\circ \tra_V =
(\tra_V\otimes \id)\circ \delta^{A\otimes V}$  (with $A\otimes V\in {\cal M}^H$ via
$\delta^{A\otimes V}$ as above). We thus define the category of {\it
  relative Hopf modules}:

\begin{defi} \label{def:AMH}
Let $H$ be a bialgebra and $A\in{\cal A}^H$.
An  
$(A,H)$-\textbf{relative Hopf module}
$V$ is a right $H$-comodule with a compatible left $A$-module structure,
i.e.\ the left  $A$-action $\tra_V$ is a morphism of $H$-comodules according to
the following commutative diagram
\be\label{AMHdiag}
\xymatrix{
\ar[d]_-{\tra_V} A\otimes V \ar[rr]^-{\delta^{A \ot V}} &&\ar[d]^-{\tra_V \otimes \id} A \otimes V\otimes H\\
V \ar[rr]^-{\delta^V}&& V\otimes H
}
\ee
Explicitly, for all $a\in A$ and $v\in V$,
\be\label{eqn:modHcov} 
\zero{(a \tra_V v)} \ot \one{(a \tra_V v)} = \zero{a} \tra_V \zero{v} \ot \one{a}\one{v} ~. 
\ee
A morphism of 
$(A,H)$-relative Hopf modules is a morphism of
right $H$-comodules  which is also a morphism of left $A$-modules.
We denote by $ {}_{A}{\cal M}^H$ the category of 
$(A,H)$-relative Hopf modules.
\end{defi}
\begin{rem}
If $V$ is a left $A$-module then $V\otimes H$ is a left $(A\otimes H)$-module
via the left $(A\otimes H)$-action $(a\ot h)\,
(v\ot h'):=a\tra_V v\ot hh'$, for all $a\in
A,\,v\in V,\, h,h'\in H$. The commutativity of the diagram
(\ref{AMHdiag}) is equivalent to
\be
\delta^V(a\tra_V v)=\delta^A(a)
_{\:\!}\delta^V(v)~,
\ee
i.e., $\delta^V$ is a left $A$-module morphism, where the left $A$-action on 
$V\otimes H$ is via $\delta^A : A\to A\otimes H$ and the left $(A\otimes H)$-action above. 
\end{rem}
\sk

Analogously  to Definition \ref{def:AMH} 
we define the categories of relative Hopf modules ${{\cal M}_A}^{\,H}$ and ${{}_{A}{\cal M}_A}^{\!H}\,$:
\begin{defi} \label{def:AMAH}
Let $H$ be a bialgebra and $A\in {\cal A}^H$.
\begin{itemize}
\item[(i)] 
The objects in the category
${{\cal  M}_A}^{\!H}$
are  right $H$-comodules with a compatible right $A$-module structure, 
i.e.\ $V$ is an object in ${{\cal  M}_A}^{\!H}$ if the right $A$-action $\trl_V$
is a morphism of $H$-comodules according to the following commutative diagram
\be\label{MAHdiag}
\xymatrix{
\ar[d]_-{\trl_V} V\otimes A \ar[rr]^-{\delta^{V \ot A}} &&\ar[d]^-{\trl_V \otimes \id} V \otimes A\otimes H\\
V \ar[rr]^-{\delta^V}&& V\otimes H
}
\ee
Explicitly, for all $a\in A$ and $v\in V$,
\be\label{eqn:rightmodHcov} 
\zero{(v \trl_V a)} \ot \one{(v \trl_V a)} = \zero{v} \trl_V \zero{a} \ot \one{v}\one{a} ~. 
\ee
The morphisms in the category ${{\cal M}_A}^{\!H}$ are morphisms of
right $H$-comodules  which are also  morphisms of right $A$-modules.

\item[(ii)]
The objects in the category ${_A{\cal  M}_A}^{\!H}$
are  right $H$-comodules with a compatible  $A$-bimodule structure, 
i.e.\ the commuting left and  right $A$-actions satisfy respectively (\ref{AMHdiag}) and (\ref{MAHdiag}).
The morphisms in the category ${_A{\cal M}_A}^{\!H}$ are  morphisms of
right $H$-comodules  which are also morphisms of $A$-bimodules.
\end{itemize}
\end{defi}
\sk

Given a left $A$-module  $V$ and a $\bbK$-module  $W$, 
the $\bbK$-module $V \ot W$ is a left $A$-module  with left action
defined by 
$\tra_{V \ot W}:=\tra_{V}\otimes\ \id$, i.e.\ 
\begin{eqnarray}\label{tra_ot}
\tra_{V \ot W}: A \ot V \ot W &\longrightarrow & V \ot W~, \\
a \ot v \ot w \;&\longmapsto &  (a \tra_V v) \ot  w ~ .\nonumber
\end{eqnarray}

\begin{lem}\label{lem:ot}
If $V\in {}_{A}{\cal M}^H$ and $W\in {\cal M}^H$, then the right $H$-comodule $V \ot W$
equipped with the left $A$-action given by \eqref{tra_ot} is an object in $ {}_{A}{\cal M}^H$.
\end{lem}
\begin{proof}
We prove that  the compatibility condition \eqref{eqn:modHcov} between the left $A$-action $\tra_{V \ot W}$  
(see \eqref{tra_ot}) and the  right $H$-coaction $\delta^{V \ot W}$ (see \eqref{deltaVW})   is satisfied:
\begin{eqnarray*}
\zero{\left(a \tra_{V \ot W} (v \ot w) \right)} \ot \one {\left(a \tra_{V \ot W} (v \ot w) \right)}
&=& 
\zero{\left( (a \tra_V  v )\ot w \right)} \ot \one {\left( (a \tra_V v) \ot w \right)}
\\
&=& 
\zero{ (a \tra_V  v )} \ot \zero{w} \ot \one{ (a \tra_V  v )} \one{w}
\\
&=& 
\left( \zero{a} \tra_V \zero{v} \right) \ot \zero{w} \ot \one{a} \one{v} \one{w}
\\
&=& 
\zero{a} \tra_{V \ot W} \left(  \zero{v} \ot \zero{w} \right)  \ot \one{a} \one{v} \one{w} 
\\
&=& 
\zero{a} \tra_{V \ot W}  \zero{(v \ot w)}   \ot \one{a} \one{(v  \ot w)}~,
\end{eqnarray*}
where  in the third passage we have used that  \eqref{eqn:modHcov} holds for the left $A$-action  $\tra_V$
 and the right $H$-coaction $\delta^V$. 
\end{proof}
\begin{rem}\label{rem:modulecategory}
More abstractly, Lemma \ref{lem:ot} states that ${}_{A}{\cal M}^H$ is a (right) module category
over the  monoidal category ${\cal M}^H$, see e.g.\ \cite{Ostrik}. Indeed, we have a
 bifunctor (denoted with abuse of notation also by $\otimes $) 
 $\otimes : {}_{A}{\cal M}^H \times {\cal M}^H \to {}_{A}{\cal M}^H$ which assigns to an
 object $(V,W)\in  {}_{A}{\cal M}^H \times {\cal M}^H$ 
 the object $V\otimes W\in  {}_{A}{\cal M}^H$ constructed in  Lemma \ref{lem:ot}
 and to a morphism $(f:V_1\to V_2\,,\; g: W_1\to W_2)$ in ${}_{A}{\cal M}^H \times {\cal M}^H$
 the ${}_{A}{\cal M}^H $-morphism $f\otimes g : V_1\otimes W_1\to V_2 \otimes W_2 \,,~v\otimes w \mapsto f(v)\otimes g(w)$.
 (It is easy to check that $f\otimes g$ is a morphism of left $A$-modules).
\end{rem}
\sk

Analogously, given a right $A$-module  $V$ and a $\bbK$-module  $W$, 
the $\bbK$-module $W\ot V$ is a right $A$-module  with right action
defined by 
$\trl_{W \ot V}:=\id\otimes \trl_{V}$. We omit the proof of the corresponding
\begin{lem}\label{lem:otl}
If $V\in {{\cal M}_A}^{\!H}$ and $W\in {\cal M}^H$, then the right $H$-comodule $W \ot V$ equipped with the right
  $A$-action given by $\trl_{W \ot V}:=\id\otimes \trl_{V}$ is an
  object in $ {{\cal M}_A}^{\!H}$.
\end{lem}
\begin{rem}\label{rem:modulecategory2}
Analogously to Remark \ref{rem:modulecategory}, Lemma \ref{lem:otl} states that
${{\cal M}_A}^{\!H}$ is a  (left) module category over the monoidal category ${\cal M}^H$.
We denote again with an abuse of notation the corresponding bifunctor  simply by
$\otimes : {\cal M}^H \times {{\cal M}_A}^{\!H} \to {{\cal M}_A}^{\!H}$.
\end{rem}
\begin{rem}\label{rem:bifunctorAMAH}
Combining Remark \ref{rem:modulecategory} and Remark \ref{rem:modulecategory2},
we have a bifunctor (denoted again by the same symbol) 
$\otimes: {}_{A}{\cal M}^{H}\times {\cal M}_{A}{}^{H}\to
{}_{A}{\cal M}_{A}{}^{H}$. Precomposing this functor with the forgetful functor
${}_{A}{\cal M}_{A}{}^{H}\times {}_{A}{\cal M}_{A}{}^{H} \to {}_{A}{\cal M}^{H}\times {\cal M}_{A}{}^{H}$
we obtain another bifunctor (denoted once more by the same symbol)
$\otimes : {}_{A}{\cal M}_{A}{}^{H} \times {}_{A}{\cal M}_{A}{}^{H}\to {}_{A}{\cal M}_{A}{}^{H}$.
Notice that this bifunctor satisfies the  associativity constraint, however it {\it does not} structure
${}_{A}{\cal M}_{A}{}^{H}$ as a monoidal category since there exists no unit object $I\in {}_{A}{\cal M}_{A}{}^{H}$.
We therefore consider the tensor product over the algebra $A$, 
 $\otimes_A : {}_{A}{\cal M}_{A}{}^{H}\times {}_{A}{\cal M}_{A}{}^{H}\to {}_{A}{\cal M}_{A}{}^{H}$
 (obtained via the standard quotient procedure) that turns 
 ${}_{A}{\cal M}_{A}{}^{H}$ into a  monoidal category with unit object $A\in {}_{A}{\cal M}_{A}{}^{H}$.
\end{rem}
\sk

In the following $H$ will be assumed to be a Hopf algebra.
\begin{defi} \label{def:hg}
Let $H$ be a Hopf algebra and $A\in{\cal A}^H$.
Let $B\subseteq A$ be the subalgebra of coinvariants, i.e.\
\be
B:= A^{coH}=\big\{b\in A ~|~ \delta^A (b) = b \ot \1_H \big\}~.
\ee 
The map 
\begin{eqnarray}\label{can}  \can := (m \ot \id) \circ (\id \otimes_B \delta^A ) : A \otimes_B A &\longrightarrow& A \ot H~  , 
 \\  a\ot_B a'
&\longmapsto&
a _{}a'_{\;(0)} \ot a'_{\;(1)} \nn
\end{eqnarray} is called the \textit{canonical map}. 
The extension $B\subseteq 
A$ is called an $H$-\textbf{Hopf-Galois extension} provided the
canonical map is bijective. 
\end{defi}

The notion of  Hopf-Galois extensions in this general context of (not
necessarily commutative) algebras appeared in \cite{KT}. It  generalizes the
classical notion of  Galois field extensions and with a  noncommutative
flavor it can be viewed as encoding the data of a principal bundle.
We refer the reader to the references \cite{Mont}, \cite[Part VII]{tok-notes} and examples therein.
See also  Example \ref{ex:principalbundle} below.
\sk

In the special case when $A$ is commutative (and hence also $B\subseteq A$ is commutative), then
$A\otimes_B A$ is an algebra and the canonical map $\can$  is an algebra morphism.
In general however $A$ is noncommutative and also $B$ is not contained in the center of $A$, so 
$A\otimes_B A$ does not even inherit an algebra structure. 
As we shall  now show, in the general case the canonical map $\can$ is 
 a morphism in the category of relative Hopf modules ${_A{\cal  M}_A}^{\!H}$.
\sk

The tensor product $A \ot A$ is an object in ${{}_A{\cal M}_A}^{\!H}$ because of
Lemma \ref{lem:ot} (take $V=A$ with left $A$-action given by the product in $A$ and $W=A$) 
and of Lemma  \ref{lem:otl} (take $V=A$ with right $A$-action given by the product in $A$ and $W=A$);
the  compatibility between the left and the right $A$-actions is immediate: 
 $c((a\otimes a')c')=ca\otimes a'c'=(c(a\otimes a'))c'$, for all $a,a',c,c'\in A $.
The right $H$-coaction $\delta^{A\otimes A}: A\otimes A\to A\otimes A\otimes H$ 
descends to the quotient $A\ot_B A$ because $B\subseteq A$ is the
subalgebra of $H$-coinvariants. We denote the induced right $H$-coaction by
$\delta^{A\otimes_B A}$. The left and right $A$-actions on $A\ot A$ also canonically descend to the
quotient $A\ot_B A$, hence they are compatible with the right $H$-comodule
structure (cf.\ (\ref{AMHdiag}) and (\ref{MAHdiag})) and therefore $A\ot_B
A \in  {{}_{\,A}{\cal M}_A}^{\!H}$.
\sk

The tensor product $A\otimes H$ is also an object in  ${_A{\cal M}_A}^{\!H}$.
First we regard the Hopf algebra $H$ as the right $H$-comodule
$\underline{H}$ defined to be the $\bbK
$-module $H$ with 
the right adjoint $H$-coaction
\be\label{adj}
\delta^{\underline{H}}=\mathrm{Ad} : \underline{H}\longrightarrow \underline{H}\ot H 
~,~~h \longmapsto \two{h}\otimes S(\one{h})\,\three{h} ~.
\ee
(The notation $\underline{H}$ is 
in order to distinguish this structure from the Hopf algebra
structure). 
Then, since $A\in {}_A{\cal M}^H$ and $\underline{H}\in  {\cal M}^H$,  Lemma \ref{lem:ot}
implies that 
$A\otimes \underline{H}\in {}_A{\cal M}^H$. Explicitly the 
right $H$-coaction $\delta^{A\otimes \underline{H}}: A\ot
\underline{H}\to A\ot \underline{H}\ot H$ is given by
\footnote{Similarly on the
  tensor product $A\otimes H$ we also  have the $H$-comodule structure
$\delta^{A\otimes {{H}}}$ induced by the right regular
coaction (coproduct) of $H$.
Notice that if $A$  is isomorphic to the $H$-comodule $B \ot H$ with right coaction $\id_B \ot \Delta$ (hence in particular if $A$ is cleft, see page \pageref{cleaving}), then the $H$-comodules
$\left(A \ot H, \delta^{A\otimes {{H}}}\right)$ and 
$\left(A \ot \underline{H}, \delta^{A\otimes {\underline{H}}}\right)$
are isomorphic. 
The isomorphism  is given by $A\otimes H \rightarrow A\otimes\underline{H}$, $(a \ot h)\mapsto (a h_1 \ot h_2)$, with inverse  $A\otimes\underline{H} \rightarrow A\otimes H $, $(a \ot h)\mapsto (a S(h_1) \ot h_2)$, where 
$a h$ indicates the action of $H$ on $A\simeq B \ot H$ given by right multiplication.
}
(cf.\ \eqref{deltaVW}), for all $ a\in A,~ h \in \underline{H}$,
\be\label{AHcoact}
\delta^{A\otimes \underline{H}}(a\otimes h) = \zero{a}\otimes \two{h} \otimes \one{a}\,S(\one{h})\, \three{h} \in A \ot \underline{H} \ot H~.
\ee 
Finally, $A\otimes \underline{H}$ is a right $A$-module with the action
\begin{eqnarray}\label{trl_ot}
\trl_{A \ot \underline{H}}: A \ot \underline{H} \ot A &\longrightarrow & A \ot \underline{H}~, \\
a \ot h \ot c \;&\longmapsto &  a\zero{c} \otimes h\one{c} ~ .\nonumber
\end{eqnarray}
This right $A$-action is easily seen to be a morphism in ${\cal M}^H$,
indeed the diagram
\be\label{AHAdiag}
\xymatrix{
\ar[d]_-{\trl_{A\otimes \underline{H}}} A\otimes \underline{H}\ot A \ar[rr]^-{\delta^{A \ot
      \underline{H}\ot A}} &&\ar[d]^-{\trl_{A\ot \underline{H}} \otimes \id} A \otimes \underline{H}\otimes A\otimes H\\
A\ot \underline{H} \ar[rr]^-{\delta^{A\ot \underline{H}}}&& A\otimes \underline{H}\ot H
}~
\ee
is commutative. Here, according to (\ref{deltaVW}),  $\delta^{A\ot
  \underline{H}\ot A}(a\ot h\ot c) =\zero{a}\ot \two{h}\ot \zero{c}\ot
\one{a}S(\one{h})\three{h}\one{c}$, for all $ a,c\in A,~ h\in \underline{H}$.
This shows that $A\otimes \underline{H}\in {{\cal M}_A}^{\!H}$. Since the left and
right $A$-actions commute we conclude that $A\otimes \underline{H}\in {_A{\cal M}_A}^{\!H}$.

\begin{prop}\label{prop_canMorph}  
The canonical map
 $\can = (m \ot \id) \circ (\id \otimes_B \delta^A ) : A \otimes_B A \ra A \ot \underline{H}$
is a morphism in ${{}_A{\cal M}_A}^{\!H}$ with respect to the
${{}_A{\cal M}_A}^{\!H}$-structures on $A \otimes_B A $ and $A \ot \underline{H}$
 described above.
 \end{prop}
\begin{proof}
We show that the canonical map is a morphism of right $H$-comodules,
for all $a,a'\in A$,
\begin{flalign}
\nonumber \delta^{A\otimes \underline{H}}\big(\can(a\otimes_B a^\prime)\big)&= \delta^{A\otimes \underline{H}}(a\, \zero{a^\prime}\otimes \one{a^\prime})
= \zero{a}\zero{a'} \otimes \three{a'} \ot \one{a} \one{a^\prime} S(\two{a^\prime} ) \four{a^\prime} \\
\nonumber &=
\zero{a}\zero{a'} \otimes \one{a'} \ot \one{a} \two{a^\prime} 
=
(\can\otimes\id)\left( \left(\zero{a} \ot_B \zero{a'}\right) \otimes \one{a}  \one{a'}  \right)
\\
&=
(\can\otimes\id)\big(\delta^{A\otimes_B A}(a\otimes_B a^\prime)\big)~.\nn
\end{flalign}
It is immediate to see that $\can$ is a morphism of left and right $A$-modules. 
\end{proof}

\begin{ex}
Let $A=H$ be the right $H$-comodule algebra with right $H$-coaction given by the coproduct $\Delta$.
Since $(\varepsilon\otimes \id)\Delta(h)=h$, for all $ h\in H$, we have
$H^{coH}\simeq \bbK$. The canonical map $\chi: H\otimes H\to H\otimes \underline{H}$ is
an isomorphism with inverse $\chi^{-1}(h\otimes h')=hS(\one{h'})\otimes\two{h'}$, for all $ h\in H$ and $ h'\in \underline{H}$. Hence, 
$\bbK\subseteq H$ is an $H$-Hopf-Galois extension.
Notice that $S(h)=(\id\otimes \varepsilon)\chi^{-1}(\1\otimes h)$, for all
$h\in H$; actually a bialgebra $H$ is a Hopf algebra if and only if
$\chi : H\otimes H\to H\otimes \underline{H}$ is an isomorphism. 
\end{ex}

\begin{ex}\label{ex:trivialprincipalbundle}
Let $B$ be an algebra with trivial right $H$-coaction, i.e.\ $\delta^B (b) = b\otimes \1$ for all $b\in B$.
 Let, as in the previous example, $H$ be the right $H$-comodule algebra with the coaction given by the
 coproduct $\Delta$.  Then $A:=B\otimes H$ is a right $H$-comodule algebra (with the usual tensor product algebra
and right $H$-comodule structure). We have $A^{coH}\simeq B$ and 
$\chi : (B\otimes H)\otimes_B (B\otimes H)\simeq B\otimes H\otimes
H \to B\otimes H\otimes \underline{H}\,,~b\otimes h\otimes h^\prime \mapsto b\otimes h\one{h^\prime}\otimes
\two{h^\prime}$ is easily seen to be invertible; hence
$B\subseteq A$ is an $H$-Hopf-Galois extension. 
\end{ex}

\begin{ex}\label{ex:principalbundle}
Let $G$ be a Lie group, $M$ a manifold and $\pi:P\to M$ a principal $G$-bundle over $M$
with right $G$-action denoted by $r_P : P\times G\to P\,,~(p,g)\mapsto p\,g$.
(All manifolds here are assumed to be finite-dimensional and second countable).
We assign to the total space $P$ its space of smooth functions $C^\infty(P)$ and recall that
it is a (nuclear) Fr{\'e}chet space with respect to the usual $C^\infty$-topology. 
Even more, the Fr{\'e}chet space  $A=C^\infty(P)$ is a unital Fr{\'e}chet algebra with (continuous) product
$m:= \mathrm{diag}_P^\ast : A\, \widehat{\otimes}\, A \to A$ and unit
$\eta := \mathrm{t}_P^\ast : \bbK \to A$. Here $A\, \widehat{\otimes}\, A\simeq C^\infty(P\times P)$
denotes the completed tensor product and the product and unit are defined as the pull-back on functions
of the diagonal map $\mathrm{diag}_P : P\to P\times P$ and the terminal map $\mathrm{t}_P : P\to \mathrm{pt}$ to a point.
Similarly, $B= C^\infty(M)$ is a Fr{\'e}chet algebra  
and $H =C^\infty(G)$ 
is a Fr{\'e}chet Hopf algebra  with co-structures and antipode defined
by the pull-backs of the Lie group structures on $G$. (In a Fr{\'e}chet Hopf algebra also the
antipode, counit and coproduct $\Delta : H\to
H\widehat{\ot} H$ are continuous maps).
The right $G$-action $r_P : P\times G\to P$ induces the structure of a Fr{\'e}chet 
right $H$-comodule algebra on $A$ and we denote the (continuous)  right $H$-coaction 
by $\delta^A := r_P^\ast  : A \to A\, {\widehat{\otimes}}\, H$.
The $H$-coinvariant subalgebra is $A^{\text{co}H}=C^\infty(P/G)$ and  $A^{\text{co}H}\simeq B=C^\infty(M)$
is the pull-back of the isomorphism  $M \simeq P/G$ of the principal
$G$-bundle $P\to M$.
The canonical map in the present case may be obtained by considering the
pull-back of the smooth map 
\begin{flalign}\label{eqn:canmapspaces}
P\times G \longrightarrow  P\times_M P~,~~(p,g)\longmapsto (p,p\,g)~,
\end{flalign}
where $P\times_M P:= \{(p,q)\in P\times P ~\vert ~\pi(p) = \pi(q) \}$ is the fibered product.
This map is an isomorphism of right $G$-spaces, because the $G$-action
on the fibers of $P$ is free and transitive. It follows that the
canonical map\footnote{
The topological tensor product over $B$ is defined as follows: Consider the two parallel continuous linear maps
$m \,\widehat{\otimes}\,\id$ and $\id\,\widehat{\otimes} \, m$ from $A\,\widehat{\otimes} \,B \,\widehat{\otimes} \,A$
to $A\,\widehat{\otimes}\,A$, which describe the right and respectively left $B$-action on $A$.
We set $A\,\widehat{\otimes}_B\, A := A\,\widehat{\otimes} \,A / \,\overline{\mathrm{Im}(m\widehat{\otimes}\id - \id\,\widehat{\otimes}\,m)}$,
where $\overline{\phantom{\mathrm{I}({\otimes\!\!\!\!})}}$
denotes the closure in the Fr{\'e}chet space $A\,\widehat{\otimes}\,A$.
Notice that $A\,\widehat{\otimes}_{B}\, A \simeq C^\infty(P\times_M P)$.
} 
$\can : A\,\widehat{\otimes}_B\, A \to A\,\widehat{\otimes}\, \underline{H}$
is an  ${{}_A{\cal M}_A}^{\!H}$-isomorphism, hence 
$B\subseteq A$ is a Fr{\'e}chet $H$-Hopf-Galois extension.

The previous two examples correspond to (algebraic versions of)
the principal $G$-bundle $G \to \mathrm{pt}$ over a point and to the trivial principal $G$-bundle $M\times G \to M$ over $M$.
\end{ex}
\sk
An $H$-Hopf-Galois extension $B:=A^{coH}\subseteq A$ is said to have the
normal basis property if there exists an
isomorphism  $A \simeq B \ot H$ of left $B$-modules and right $H$-comodules (where
$B \ot H$ is a left $B$-module via $m_B \ot \id$ and a right
$H$-comodule via $\id \ot \Delta$, cf.~Example
\ref{ex:trivialprincipalbundle}). This condition captures the
algebraic aspect of triviality of a principal bundle.
We recall that  the normal basis property is equivalent to the existence of a convolution
invertible map $j:H \ra A$ (called \textit{cleaving map}) that  is a right $H$-comodule morphism, i.e.\

\be\label{cleaving}
\delta^A\circ  j= (j \ot \id)\circ  \Delta~.
\ee
A comodule algebra $A$ for which there exists a convolution invertible morphism of $H$-comodules  $j:H \ra A$ is called a \textit{cleft extension} of $A^{coH}$.
Given an isomorphism $\theta: B \ot H \ra A$ of left $B$-modules and right $H$-comodules, then a cleaving map 
$j : H\to A$ and its  convolution inverse $\bar{j} : H\to A$ are
determined by
\begin{flalign}\label{j-theta}
j : H \longrightarrow A ~,~~&h \longmapsto \theta(\1 \ot h) ~,\\
\bar{j} : H\longrightarrow A~,~~ &h \longmapsto (\id \ot \varepsilon)\circ (\id \ot_B \theta^{-1}) \circ \can^{-1}(\1 \ot h)~.\nn
\end{flalign}
(In order to prove that $j\ast\bar{j}=\eta_A \circ \varepsilon$ use
$A$-linearity of $\can^{-1}$, then that $\theta$ is an $H$-comodule
map and then recall the definition of $\can$. In order to prove that
$\bar{j}\ast j=\eta_A \circ \varepsilon$ it is convenient to set
$\can^{-1}(h)=h^{<1>}\ot h^{<2>}$ for all $h\in H$, then use $(\can^{-1}\ot \id)(\id\ot
\Delta)=(\id\ot \delta^A)\can^{-1}$, observe that since $\theta$ is an
$H$-comodule map so is $\theta^{-1}$ and hence $(\id\ot\varepsilon\ot
\id)(\theta^{-1}\ot\id)\delta^A =\theta^{-1}$, and that, due to left
$B$-linearity of $\theta$, $(\id\otimes \varepsilon)(\theta^{-1}(a_{(0)})\theta(1\ot
a_{(1)})=a$ for all $a\in A$. See e.g.\ \cite[Theorem 8.2.4]{Mont}).
\sk

To conclude this subsection, let us recall the definition of principal
comodule algebra which, as it is the case  for principal bundles,
allows for the construction
of associated vector bundles (i.e.\ associated finitely generated and
projective $B$-modules). 
Among the equivalent formulations we consider the one here below \cite{DHG}  (see
also  \cite[Part VII, \S 6.3 and \S 6.4]{tok-notes}) because it will be
easily seen to be preserved by twist deformations. 

\begin{defi}  \label{def:pHcomodalg} 
Let $H$ be a Hopf algebra with invertible  antipode over a field $\bbK$. A
{\bf principal $H$-comodule algebra} is an object $A\in {\cal A}^H$ 
such that $B:=A^{coH}\subseteq A$ is an $H$-Hopf-Galois extension
and $A$ is equivariantly projective as a left $B$-module, i.e.\ there exists a left $B$-module and right $H$-comodule morphism $s: A\to
B\ot A$ that is a section of the (restricted) product $m: B\ot A\to A$,
i.e.\ such that $m\circ s=\id_A$.
\end{defi}  

In particular if $H$ is a Hopf algebra with bijective antipode over a field, 
the condition of equivariant projectivity of $A$ is equivalent to that of faithful
flatness of $A$ \cite{SchSch}. 
Moreover, by the   Theorem of characterization of faithfully flat extensions \cite{Sch},  if $H$ is cosemisimple and has a bijective antipode, then surjectivity of the canonical map is sufficient to prove the principality of $A$. 

\subsection{Deformations by 2-cocycles}\label{sec:twists}
 We review some results  from the theory of  
   deformations of Hopf algebras $H$ and their comodule (co)algebras by 2-cocycles 
   $\cot:H \otimes H \ra \bbK$.     The notion of 2-cocycle is dual to that of Drinfeld twist. In Appendix
\ref{app:twists} we  detail this duality for the reader more familiar
with the Drinfeld twist notation. 
We omit some of the  proofs of standard results,
see e.g.\ \cite{doi} and also
   \cite[\S 10.2]{KS}, or, in the dual Drinfeld twist picture,
\cite[\S 2.3]{Majid}.
 We also extend results from the category of $H$-comodules to those of
   relative Hopf (co)modules and bicomodules (cf.\ Proposition
   \ref{propo:moddef},  Proposition \ref{propo:leftdef}, and
   Proposition \ref{prop:leftrightdef}) which will be relevant to our construction in \S \ref{sec:twHG}.

\subsubsection{Hopf algebra 2-cocycles}\label{sec:twists-hopf}

Let $H$ be a Hopf algebra and recall that $H\ot H$ is canonically
a coalgebra with coproduct $\Delta_{H\ot H}(h\ot k)=\one{h}\ot
\one{k}\ot \two{h}\ot \two {k}$ and counit $\varepsilon_{H\otimes
  H}(h\otimes k)=\varepsilon(h)\varepsilon(k)$, for all $ h,k\in H$. In particular, we can
consider the convolution product of $\bbK$-linear maps $H\ot H\to \bbK$.
\begin{defi}
A $\bbK$-linear map $\cot:H \otimes H \ra \bbK$ is called a convolution
invertible, unital 2-cocycle on $H$, or simply a \textbf{2-cocycle}, provided $\cot$ is convolution invertible,
{unital}, i.e.\
$\co{h}{\1}= \varepsilon(h) = \co{\1}{h}$, for all  $ h\in H$, and satisfies the
 $2$-cocycle condition
\be\label{lcocycle}
\co{\one{g}}{\one{h}} \co{\two{g}\two{h}}{k} =  \co{\one{h}}{\one{k}} \co{g}{\two{h}\two{k}}~,
\ee
for all $g,h, k \in H$. 
\end{defi}

The following lemma is easily proven. The stated equalities will be used for computations in  the next sections.
\begin{lem}\label{lem:formula}
Let   $\cot:H \otimes H \ra \bbK$ be a convolution invertible $\bbK$-linear map, 
with inverse denoted by $\bar\gamma : H\otimes H\to\bbK$. Then the
following are equivalent:
\begin{enumerate}[(i)]
\item
$\cot$ satisfies \eqref{lcocycle}
\item\label{ii}
$\coin{\one{g}\one{h}}{k} \coin{\two{g}}{\two{h}}=  \coin{g}{\one{h}\one{k}} \coin{\two{h}}{\two{k}}\,,\;$ for all $g,h, k \in H$
\item\label{iii}
$
\co{\one{g}\one{h}}{\one{k}} \coin{\two{g}}{\two{h}\two{k}} =
\coin{g}{\one{h}} \co{\two{h}}{k}\,,\;$ for all $g,h, k \in H$
\item\label{iv}
$
\co{\one{g}}{\one{h}\one{k}} \coin{\two{g}\two{h}}{\two{k}}=  \co{g}{\two{h}} \coin{\one{h}}{k}\,,\;$
for all $g,h, k \in H$
\end{enumerate}
\end{lem}
\sk

{}Given a 2-cocycle $\cot$, with the help of  $(iii)$ and $(iv)$, it is
possible to prove that the maps
\begin{flalign}\label{uxS}
u_\cot:  H\longrightarrow \bbK ~, ~~ h\longmapsto \co{\one{h}}{S(\two{h})}  ~,\\
\bar{u}_\cot: H\longrightarrow \bbK~,~~ h \longmapsto \coin{S(\one{h})}{\two{h}} ~, \nn
\end{flalign}
are the convolution inverse of each other.

\begin{prop}\label{prop:co}
Let $\cot: H\otimes H\to \bbK$ be a  2-cocycle. Then 
\be\label{hopf-twist}
m_{\cot} (h \ot k):= h \mt k:= \co{\one{h}}{\one{k}} \,\two{h}\two{k}\, \coin{\three{h}}{\three{k}}~,
\ee
for all $h,k\in H$, 
defines a new associative product on (the $\bbK$-module underlying) $H$. The resulting algebra 
$\hg:=(H,m_\cot,\1_H)$ is a Hopf algebra when endowed with the unchanged coproduct $\Delta$ and 
counit $\varepsilon$ and with the new antipode  $S_\cot:= u_\cot *S
*\bar{u}_\cot$. We call $H_\cot$ the twisted Hopf algebra
of $H$ by $\cot$. 
\end{prop}

\begin{rem}
The twisted Hopf algebra $\hg$ can be `untwisted' by using the convolution inverse
$\bar\cot : H\otimes H\to\bbK$;
indeed, $\bar\cot$ is a 2-cocycle for $\hg$ and the twisted Hopf
algebra of  $\hg$ by $\bar\cot$ 
is isomorphic to $H$ via the identity map (see Corollary \ref{cor:untwist}). 
\end{rem}

\subsubsection{\label{sec:rightHcomod}Twisting of right $H$-comodules}

The deformation of a Hopf algebra $H$ by a 2-cocycle $\cot: H\otimes H\to\bbK$
 affects also the category ${\cal M}^H$ of right $H$-comodules. 
Indeed, if $V\in{\cal M}^H$  with coaction $\delta^V:V \ra V \ot H$, then $V$ with the
same coaction, but now thought of as a map with values in $V \ot H_\gamma$, is a right $\hg$-comodule.  
This is evident from the fact that the comodule condition (cf.\ \eqref{eqn:Hcomodule}) only involves the coalgebra structure of $H$, and  
$\hg$ coincides with $H$ as a coalgebra. When thinking of $V$ as an object in $\mathcal{M}^{\hg}$ we will denote it by 
$V_\cot$ and the coaction by $\delta^{V_\cot} : V_\cot \to V_\cot \otimes \hg$. 
Moreover, any morphism $\psi : V\to W $ in $ {\cal M}^H$
can be thought as a morphism $\psi  : V_\cot\to W_\cot$ in ${\cal M}^{H_\cot}$;
indeed, $H$-equivariance of $\psi: V\to W $ implies
$H_\cot$-equivariance of $\psi  : V_\cot\to W_\cot$
since by construction the right $H$-coaction in $V$ agrees with the right $H_\cot$-coaction
in $V_\cot$.
Hence we have a functor 
\begin{equation}\label{functGamma}
\Gamma : {\cal M}^H \to{\cal M}^{H_\cot}~,
\end{equation}
defined by $\Gamma(V):=V_\cot$ and $\Gamma(\psi):=\psi :V_\cot\to W_\cot$.
Furthermore this functor $\Gamma$ induces an equivalence of categories 
because we can use the convolution inverse $\bar\cot$ in order to
twist back $H_\cot$ to $(H_\cot)_{\bar{\cot}}=H$ and $V_\cot$ to
$(V_\cot)_{\bar{\cot}}=V$.

\sk
The equivalence between the
categories ${\cal M}^H$ and ${\cal M}^{H_\cot}$ extends to their monoidal structure.
We denote by 
$(\mathcal{M}^{\hg},\ot^\cot)$ the monoidal category corresponding to
the Hopf algebra $H_\cot$. 
Explicitly, for all objects  $V_\cot,W_\cot\in  \mathcal{M}^{\hg}$
(with coactions  $\delta^{V_\cot} : V_\cot \to V_\cot \otimes \hg$ and $\delta^{W_\cot} : W_\cot \to W_\cot \otimes\hg$), the right 
$H^\cot$-coaction on $V_\cot\ot^\cot W_\cot $,  according to \eqref{deltaVW}, is
given by 
\begin{eqnarray}\label{deltaVWcot}
\delta^{V_\cot \ot^\cot W_\cot} :V_\cot\otimes^\cot W_\cot &
                                                        \longrightarrow & V_\cot \otimes^\cot W_\cot\otimes H_\cot~,\\
 v\otimes^\cot w & \longmapsto & \zero{v}\otimes^\cot \zero{w} \otimes \one{v}\mt\one{w} ~,\nn
\end{eqnarray}

\begin{thm}\label{thm:funct}
The functor $\Gamma : {\cal M}^H \to{\cal M}^{H_\cot}$ induces an equivalence
between the  monoidal categories $(\mathcal{M}^H, \ot)$ and $(\mathcal{M}^{\hg},
\ot^\cot)$. The natural isomorphism   $\varphi : \otimes^\cot \circ
(\Gamma\times\Gamma)\Rightarrow \Gamma\circ \otimes$ is given by the $\mathcal{M}^{\hg}$-isomorphisms
\begin{eqnarray}\label{nt}
\varphi_{V,W}: V_\cot \ot^\cot W_\cot &\longrightarrow&  (V \ot W)_\cot  ~,
\\
v \ot^\cot w &\longmapsto &  \zero{v} \ot \zero{w} ~\coin{\one{v}}{\one{w}} ~,\nn
\end{eqnarray}
for all objects $V,W\in {\cal M}^H$.
\end{thm}
\begin{proof}
The invertibility of $\varphi_{V,W}$ follows immediately from the invertibility of the cocycle $\cot$.  
 The fact that it  is a morphism in the category $\mathcal{M}^{\hg}$ is easily shown as follows:
\begin{eqnarray*}
(\varphi_{V,W} \ot \id) \left(\delta^{V_\cot \ot^\cot W_\cot} (v \ot^\cot w) \right)&=&
\zero{v} \ot \zero{w} \,\coin{\one{v}}{\one{w}}\ot \co{\two{v}}{\two{w}}  \three{v}\three{w} \coin{\four{v}}{\four{w}}
\\
&=&\zero{v} \ot \zero{w} \ot {\one{v}}{\one{w}} \, \coin{\two{v}}{\two{w}} 
\\
&=& \delta^{(V \ot W)_\cot} (\zero{v} \ot \zero{w}) \,
\coin{\one{v}}{\one{w}}
\\
 &=& \delta^{(V \ot W)_\cot} \left( \varphi_{V,W} (v \ot^\cot w)\right)~,
\end{eqnarray*}
where the coaction $\delta^{(V \ot W)_\cot}$ is given by
$
\delta^{(V \ot W)_\cot}: v \ot w \longmapsto \zero{v} \ot \zero{w} \ot {\one{v}} {\one{w}}
$ (cf.\  \eqref{deltaVW}). Hence  $(\Gamma,\varphi):
(\mathcal{M}^H, \ot)\to  (\mathcal{M}^{\hg}, \ot^\cot)$ is a monoidal functor.

The monoidal categories are equivalent (actually they are isomorphic) because $\bar\cot$ twists back
$H_\cot$ to $H$ and $V_\cot$ to $V$  so that the monoidal functor $(\Gamma, \varphi)$ has
an inverse $(\overline{\Gamma}, \overline\varphi)$, where 
$\overline{\Gamma}: {\cal M}^{H_\cot}\to {\cal M}^{H}$ is the inverse of the functor $\Gamma$
 and $\overline{\varphi}_{V_\cot,W_\cot}: (V_\cot)_{\bar\cot}\otimes (W_\cot)_{\bar\cot}\to
(V_\cot\otimes^\cot W_\cot)_{\bar{\cot}}\,$, 
$\,v \ot w \mapsto   \zero{v} \ot^\cot \zero{w} ~\cot({\one{v}}\otimes {\one{w}}) $.
\end{proof}

If $V\in{\cal M}^H$ carries additional structures, i.e.\ if it is an object
in ${\cal A}^H$, ${}_{A}{\cal M}^H$, ${\cal M}_{A}{}^H$ or ${}_{A}{\cal M}_{A}{}^H$ (with $A\in{\cal A}^H $),
then these additional structures are also deformed by the $2$-cocycle $\cot: H\otimes H\to \bbK$.
Let us illustrate this for the category ${\cal A}^H$ of right $H$-comodule algebras:
Recall that an object $A\in {\cal A}^H$ is an object $A\in{\cal M}^H$ together
with two ${\cal M}^H$-morphisms, $m: A\otimes A\to A$ and $\eta : \bbK\to A$, which 
satisfy the axioms of an algebra product and unit. Using the functor $\Gamma$ of Theorem \ref{thm:funct},
we can assign to this data the object $\Gamma(A) = A_\cot\in{\cal M}^{H_\cot}$ and the
two ${\cal M}^{H_\cot}$-morphisms $\Gamma(m) : \Gamma(A\otimes A)\to \Gamma(A)$
and $\Gamma(\eta) : \Gamma(\bbK) \to \Gamma(A)$. The deformed algebra
structure $m_\cot , \eta_\cot$ on
$A_\cot\in {\cal A}^{H_\cot}$ is now defined by using the components
$\varphi_{\text{--},\text{--}}$ (cf.\ (\ref{nt})) of the natural
isomorphism $\varphi$, 
and the commutative diagrams
\begin{flalign}
\xymatrix{
\ar[d]_-{\varphi_{A,A}} A_\cot \otimes^\cot A_\cot \ar[rr]^-{m_\cot} && A_\cot && \bbK \ar[d]_-{\simeq} \ar[rr]^-{\eta_\cot} && A_\cot\\
 (A\otimes A)_\cot\ar[rru]_-{\Gamma(m)} &&  && \Gamma(\bbK)\ar[rru]_-{\Gamma(\eta)} &&
}
\end{flalign}
in the category ${\cal M}^{H_\cot}$. The deformed product $m_\cot$ is
associative due to the $2$-cocycle condition of $\cot$, and $\eta_\cot$ is the unit for $m_\cot$ since $\cot$ is unital.
Explicitly we have that $\eta_\cot = \eta$ and the deformed product reads as
 \be\label{rmod-twist} 
 m_\cot : A_\cot \otimes^\cot A_\cot \longrightarrow A_\cot ~,~~a\otimes^\cot a^\prime \longmapsto \zero{a} \zero{a'} \,\coin{\one{a}}{\one{a'}} =: a \mtco a'~.
 \ee
This construction provides us with a functor $\Gamma : {\cal A}^H \to
{\cal A}^{H_\cot}$; indeed it can be easily
checked that $\Gamma(\psi) :=\psi : A_\cot \to A^\prime_\cot$ is a morphism in ${\cal A}^{H_\cot}$
for any ${\cal A}^H$-morphism $\psi : A\to A^\prime$. Using again the convolution inverse
$\bar\cot$ of $\cot$ we can twist back $A_\cot$ to $A$.
In summary we have obtained
\begin{prop}\label{propo:algdef}
Given a 2-cocycle $\cot: H\otimes H\to \bbK$ the functor $\Gamma : {\cal A}^H\to {\cal A}^{H_\cot}$
 induces an equivalence of categories. 
\end{prop}

By a similar construction one obtains the functors (all denoted by the same symbol)
$\Gamma : {}_A{\cal M}^H\to {}_{A_\cot}{\cal M}^{H_\cot}$,
$\Gamma : {\cal M}_{A}{}^H\to {\cal M}{}_{A_\cot}{}^{H_\cot}$
and $\Gamma : {}_A{\cal M}_{A}{}^H\to {}_{A_\cot}{\cal M}_{A_\cot}{}^{H_\cot}$.
Explicitly, the deformed left $A_\cot$-actions are given by
\begin{eqnarray}\label{av}
\tra_{V_\cot}:  \pg \ot^\cot V_\cot &\longrightarrow& V_\cot \, ,\\
 a\ot^\cot v\; &\longmapsto & (\zero{a} \tra_V \zero{v}) \coin{\one{a}}{\one{v}}~,\nn
\end{eqnarray}
while the deformed right $A_\cot$-actions read as
\begin{eqnarray}\label{rAmod} 
\trl_{V_\cot}:    V_\cot \ot^\cot \pg&\longrightarrow& V_\cot ~ ,\\
 v\ot^\cot a\; &\longmapsto &(\zero{v} \trl_V \zero{a}) \, \coin{\one{v}}{\one{a}}\nn~.
\end{eqnarray}
The $A_\cot$-module and $A_\cot$-bimodule properties follow again from the $2$-cocycle condition and unitality 
of $\cot$.
Moreover the bifunctors $\otimes$ described in the Remarks \ref{rem:modulecategory}, \ref{rem:modulecategory2}
and \ref{rem:bifunctorAMAH} on module categories are preserved, i.e.,
the ${\cal M}^{H_\cot}$-isomorphisms
(\ref{nt}), that in the context of Theorem \ref{thm:funct} define the
natural isomorphism  $\varphi : \otimes^\cot \circ
(\Gamma\times\Gamma)\Rightarrow \Gamma\circ \otimes$, now are
respectively $ {}_{A_\cot}{\cal M}^{H_\cot}$, ${\cal
  M}{}_{A_\cot}{}^{H_\cot}$ and $ {}_{A_\cot}{\cal
  M}_{A_\cot}{}^{H_\cot}$-isomorphisms; they define the natural isomorphism $\varphi : \otimes^\cot \circ (\Gamma\times\Gamma)\Rightarrow \Gamma\circ \otimes$, where here $\Gamma$ is any of the three functors described above, and where $\ot^\cot$ is the tensor product corresponding to the Hopf
algebra $H_\cot$.
\sk

{}For example we here prove that, for all $V\in {}_{A}\mathcal{M}^H$
and $W \in\mathcal{M}^H$, the $\mathcal{M}^H$-isomorphism $\varphi_{V,W}: V_\cot \ot^\cot W_\cot \ra  (V \ot W)_\cot  $  given in \eqref{nt}
is an  $ {}_{A_\cot}{\cal M}^{H_\cot}$-isomorphism.  On the one hand, from definition \eqref{av} we have
\begin{eqnarray}
\varphi_{V,W} \left(a \tra_{V_\cot \ot^\cot W_\cot} (v \ot^\cot w) \right) &=&  
\varphi_{V,W} \left( (a\tra_{V_\cot} v) \ot^\cot w) \right) 
\\
&=& \nn \varphi_{V,W}   \left( (\zero{a} \tra_{V} \zero{v})   \ot^\cot {w}\right) ~\coin{\one{a}}{\one{v}}
\\
&=& \nn (\zero{a} \tra_V \zero{v}) \ot \zero{w}~ \coin{\one{a}\one{v}}{\one{w}} \coin{\two{a}}{\two{v}}
\end{eqnarray}
where in the third line we used \eqref{eqn:modHcov}. 
On the other hand, we have
\begin{eqnarray}
a \tra_{(V \ot W)_\cot} \left( \varphi_{V,W} (v \ot^\cot w) \right)
&=&  a \tra_{(V \ot W)_\cot}  (\zero{v} \ot \zero{w}) \coin{\one{v}}{\one{w}}  
\\ &=& \nn
(\zero{a} \tra_V \zero{v}) \ot \zero{w}~ \coin{\one{a}}{\one{v}\one{w}} \coin{\two{v}}{\two{w}}~.
\end{eqnarray}
These two expressions coincide because of the 2-cocycle condition (cf.\ (\ref{ii}) in Lemma \ref{lem:formula}).
\sk

In summary, we have obtained
\begin{prop}\label{propo:moddef}
Given a $2$-cocycle $\cot: H\otimes H\to \bbK$ the following functors
induce equivalences of categories:
\begin{itemize}
\item[(i)] $\Gamma :  {}_A{\cal M}{}^H\to {}_{A_\cot}{\cal M}^{H_\cot}$; 
the left $A_\cot$-actions are defined by (\ref{av}).
\item[(ii)] $\Gamma : {\cal M}_{A}{}^H\to {\cal M}{}_{A_\cot}{}^{H_\cot}$; the right $A_\cot$-actions 
are defined by (\ref{rAmod}).
\item[(iii)] $\Gamma : {}_A{\cal M}_{A}{}^H\to {}_{A_\cot}{\cal M}_{A_\cot}{}^{H_\cot}$;
the $A_\cot$-bimodule structures are defined by (\ref{av}) and  (\ref{rAmod}).
\end{itemize}
In all three cases we have that the maps in (\ref{nt}) are isomorphisms  in the
corresponding categories ($ {}_{A_\cot}{\cal M}^{H_\cot}$, $ {\cal
  M}{}_{A_\cot}{}^{H_\cot}$ and $ {}_{A_\cot}{\cal
  M}_{A_\cot}{}^{H_\cot}$ respectively); they are
the components of the natural isomorphism
 $\varphi : \otimes^\cot \circ (\Gamma\times\Gamma)\Rightarrow \Gamma\circ \otimes$,
 where the bifunctors $\otimes$ are defined in the Remarks \ref{rem:modulecategory}, \ref{rem:modulecategory2}
 and \ref{rem:bifunctorAMAH}. 

In particular {\it (i)} and {\it (ii)} induce
 the following equivalences of ${\cal M}^H$- and ${\cal M}^{H_\gamma}$-module categories:  $({}_A{\cal M}^H,\ot)\simeq ({}_{A_\cot}{\cal M}^{H_\cot},\ot^\cot)$ and 
$({\cal M}_A{}^H,\ot)\simeq ({\cal M}_{A_\cot}{}^{H_\cot},\ot^\cot)$.
\end{prop}
\sk

We finish this subsection by studying the twisting of the 
category ${\cal C}^H$ of right $H$-\textit{comodule coalgebras}.
An object $C\in {\cal C}^H$ is an object $C\in{\cal M}^H$ together
with two ${\cal M}^H$-morphisms $\Delta_C: C\to C\otimes C$
(coproduct) and $\varepsilon_C : C\to \bbK$ (counit), i.e.\
\be
\delta^{C\otimes C}\circ \Delta_C =(\Delta_C\otimes \id)\circ \delta^C~,\quad
\delta^{\bbK}\circ \varepsilon_C=(\varepsilon_C\otimes \id)\circ \delta^C \; ,
\ee
which satisfy the axioms of a coalgebra. Morphisms in ${\cal C}^H$ 
are $H$-comodule maps which are also coalgebra maps (i.e., preserve coproducts and counits).
Given now a 2-cocycle $\cot :H\otimes H\to \bbK$, we can use the functor
$\Gamma : {\cal M}^H\to{\cal M}^{H_\cot}$ in order to assign to an object
$C\in {\cal C}^H$ (with coproduct $\Delta_C$ and counit $\varepsilon_C$)
the object $C_\cot \in{\cal C}^{H_\cot}$ with coproduct ${\Delta_C}_\cot$ and
counit ${\varepsilon_C}_\cot$ defined by the commutative diagrams
\begin{flalign}
\xymatrix{
\ar[drr]_-{\Gamma(\Delta)}C_\cot \ar[rr]^-{{\Delta_C}_{\cot}} && C_\cot \otimes^\cot C_\cot\ar[d]^-{\varphi_{C,C}} && \ar[drr]_-{\Gamma(\varepsilon)} C_\cot \ar[rr]^-{{\varepsilon_C}_\cot} && \bbK\ar[d]^-{\simeq}\\
&& (C\otimes C)_\cot && && \Gamma(\bbK)
}
\end{flalign}
in the category ${\cal M}^{H_\cot}$.
Notice that ${\varepsilon_C}_\cot = \varepsilon_C$ and that the
deformed coproduct explicitly reads as
\be\label{cc-twist} 
 {\Delta_C}_\cot : C_\cot\longrightarrow C_\cot\otimes^\cot C_\cot ~,~~c \longmapsto  \zero{(\one{c})} \ot^\cot \zero{(\two{c})} \, \cot \left(\one{(\one{c})} \ot \one{(\two{c})}\right) ~.
\ee 
It is easy to check that $C_\cot$ is an object in ${\cal C}^{H_\cot}$ and that the assignment
$\Gamma : {\cal C}^H \to {\cal C}^{H_\cot}$ is a functor (as before, $\Gamma$ acts as the identity on morphisms).
In summary, we have obtained
\begin{prop}\label{propo:coalgdef}
Given a 2-cocycle $\cot:H\otimes H\to\bbK$ the functor $\Gamma : {\cal C}^H\to{\cal C}^{H_\cot}$
induces an equivalence of categories.
\end{prop}

\begin{ex}\label{Hcomco}
The right $H$-comodule $\underline{H}$ is a comodule {\it coalgebra} with
coproduct and counit canonically inherited from the Hopf algebra $H$, i.e.,
$\Delta_{\underline{H}}=\Delta$ and
$\varepsilon_{\underline{H}}=\varepsilon$. For ease of notation we
will omit the indices and denote by $\delta^{\underline{H}}$, $\Delta$,
$\varepsilon$ the comodule coalgebra structure of ${\underline{H}}$.
Cocycle deformations of $\underline{H}$ will be relevant in \S \ref{sec:def_sg}.
\end{ex}

\subsubsection{\label{secleftKcomod}Twisting of left $K$-comodules}
Of course, similar twist deformation constructions as in \S \ref{sec:rightHcomod} are available
for {\it left} Hopf algebra comodules rather than right ones.
We briefly collect the corresponding formulae as they will be needed in \S \ref{sec:twHG}.
As we later consider also the case where two (in general different) Hopf algebras coact from respectively the
left and the right, we  denote the Hopf algebra which coacts from the left by $K$.
\sk

Let $K$ be a Hopf algebra. A left $K$-comodule is a $\bbK$-module $V$
together with a $\bbK$-linear map $\rho^V : V\to K\otimes V$ (called a left $K$-coaction)
which satisfies
\be \label{eqn:Hcomoduleleft}
(\Delta \ot \id )\circ \da^V=(\id\otimes \da^V)\circ \da^V ~,\quad
(\varepsilon \ot \id) \circ \da^V =\id~. 
\ee
The  Sweedler notation for the left $K$-coaction is $\da^V(v) = \mone{v} \ot \mzero{v}$ (sum understood),
 with the  $K$-comodule properties (\ref{eqn:Hcomoduleleft}) reading as, for all $ v\in V$
\begin{eqnarray}
\one{(\mone{v})}\otimes \two{(\mone{v})} \otimes \mzero{v} &=& \mone{v} \otimes \mone{(\mzero{v})}\otimes \mzero{(\mzero{v})} =: \mtwo{v} \ot\mone{v} \ot \mzero{v}  ~,\\
\varepsilon (\mone{v}) \, \mzero{v}  &=& v~.\nn
\end{eqnarray}
We denote by ${}^K{\cal M}$ the category of left $K$-comodules; 
the morphisms in ${}^K{\cal M}$ are $\bbK$-linear maps that preserve the left $K$-coactions,
i.e.\ a $\bbK$-linear map $\psi : V\to W$ is a morphism in ${}^K{\cal M}$
provided that
\begin{flalign}
\rho^W \circ \psi = (\id\otimes \psi)\circ \rho^V~.
\end{flalign}
Notice that ${}^K{\cal M}$ is a  monoidal category with bifunctor
$\otimes : {}^K{\cal M}\times {}^K{\cal M} \to {}^K{\cal M}$ defined by equipping the 
tensor product (of $\bbK$-modules) $V\otimes W$ with the tensor product coaction
\begin{flalign}\label{eqn:lefttensor}
\rho^{V\otimes W} : V\otimes W &\longrightarrow K\otimes V\otimes W~,\\
\nn v\otimes w&\longmapsto \mone{v}\mone{w}\otimes\mzero{v}\otimes \mzero{w}~.
\end{flalign} 
The tensor product of morphisms is again $f\otimes g: V\otimes W\to V^\prime\otimes W^\prime\,,~v\otimes w\mapsto f(v)\otimes g(w)$
and the unit object in ${}^K{\cal M}$ is $\bbK$ together with the left $K$-coaction 
$\rho^\bbK  := \eta_K : \bbK \to K\otimes\bbK \simeq K$ given by the unit in $K$.
\sk

Let $\sigma: K\otimes K\to\bbK$ be a 2-cocycle on $K$ (we use the symbol $\sigma$ 
in order to distinguish it from 2-cocycles on the Hopf algebra $H$).
Proposition \ref{prop:co} provides us with a deformed Hopf algebra $K_\sigma$.
We can further construct a functor $\Sigma : {}^K{\cal M}\to {}^{K_\sigma}{\cal M}$
by assigning to an object $V\in {}^K{\cal M}$ the object
$\Sigma(V) := {}_{\sigma} V\in {}^{K_\sigma}{\cal M}$, where  ${}_{\sigma} V$ has the same underlying $\bbK$-module structure
of $V$ and left $K_\sigma$-coaction $\rho^{{}_\sigma V}  : {}_\sigma V\to
K_\sigma\otimes {}_\sigma V$ that, as a $\bbK$-linear map, is given
by the left $K$-coaction $\rho^V : V\to K\otimes V$. On morphisms
$\psi :  V\to W$ we have $\Sigma(\psi) :=\psi : {}_\sigma V\to {}_\sigma W$.
Analogously to Theorem \ref{thm:funct} we have
\begin{thm}\label{thm:functleft}
The functor $\Sigma : {}^K{\cal M} \to {}^{K_\sigma}{\cal M}$ induces an equivalence
between the monoidal categories $({}^{K}\mathcal{M}, \ot)$ and $({}^{K_\sigma}\mathcal{M},
\,{{}^{\sigma\!\!\:}\ot}\,)$. 
The natural isomorphism   $\varphi^\ell : {}^{\sigma\!\!\:}\otimes \circ\/
(\Gamma\times\Gamma)\Rightarrow \Gamma\circ \otimes\,$ is given by the ${}^{K_\sigma}\!\mathcal{M}$-isomorphisms
\begin{eqnarray}\label{nt-left} 
\col_{V,W}: {}_\sg V \,{^{\sigma}\ot}\,{}_\sg W &\longrightarrow&  {}_\sg(V \ot W)  ~,
\\
v \,{^{\sigma\!\!\:}\ot}\, w &\longmapsto &  \sig{\mone{v}}{\mone{w} }~\mzero{v} \ot \mzero{w}  ~,\nn
\end{eqnarray}
for all objects $V,W\in {}^K{\cal M}$.
\end{thm}

The category ${}^K{\cal A}$ of left $K$-comodule algebras and the categories 
${}^K{}_{A}{\cal M}$, ${}^K{\cal M}_{A}$ and ${}^K{}_{A}{\cal M}_{A}$
of relative Hopf modules are defined analogously to the case where the
Hopf algebra coacts from the  right. As in Remark
\ref{rem:modulecategory} and in Remark \ref{rem:modulecategory2}
${}^K{}_{A}{\cal M}$ and ${}^K{\cal M}_{A}$ are respectively right and left
 module categories over the monoidal category  ${}^K{\cal M}$.
As in  Proposition \ref{propo:algdef} and Proposition
\ref{propo:moddef} we obtain
\begin{prop}\label{propo:leftdef}
Let $\sigma: K\otimes K\to \bbK$ be a $2$-cocycle on $K$. 
Then the monoidal functor $(\Sigma,\varphi^\ell) : ( {}^K{\cal M},\otimes)\to ({}^{K_\sigma}{\cal M} ,\,{{}^{\sigma\!\!\:}\otimes}\,)$
leads to the following functors, which induce equivalences of categories:
\begin{itemize}
\item[(i)] $\Sigma : {}^K{\cal A}\to {}^{K_\sigma}{\cal A}$; the
  deformed products are defined by
\begin{flalign}
{}_\sigma m : {}_\sigma A\,{^{\sigma}\otimes}\,{}_\sigma A &\longrightarrow {}_\sigma A~,\\
\nn a\,{{}^{\sigma\!\!\:}\otimes}\, a^\prime & \longmapsto \sig{\mone{a}}{\mone{a'}} \, \mzero{a} \mzero{a'} =: a \,{{}_\sigma\bullet}\, a' ~.
\end{flalign}
\item[(ii)] $\Sigma : {}^K{}_{A}{\cal M} \to {}^{K_\sigma}{}_{{}_\sigma A}{\cal M}$; the left ${}_\sigma A$-actions are defined by
\begin{flalign}\label{eqn:leftleft}
\triangleright_{{}_\sigma V} : {}_\sigma A\, {{}^{\sigma\!\!\:} \otimes}\,{}_\sigma V &\longrightarrow {}_\sigma V~,\\
\nn a\,{{}^{\sigma\!\!\:}\otimes}\, v&\longmapsto \sig{\mone{a}}{\mone{v}}\,\mzero{a}\tra_V \mzero{v}~.
\end{flalign}
\item[(iii)]  $\Sigma : {}^K{\cal M}_{A} \to {}^{K_\sigma}{\cal M}_{{}_\sigma A}$; the right ${}_\sigma A$-actions are defined by
\begin{flalign}\label{eqn:leftright}
\triangleleft_{{}_\sigma V} : {}_\sigma V\, {{}^{\sigma\!\!\:} \otimes}\,{}_\sigma A &\longrightarrow {}_\sigma V~,\\
\nn v\,{{}^{\sigma\!\!\:}\otimes}\, a&\longmapsto \sig{\mone{v}}{\mone{a}}\,\mzero{v}\trl_V \mzero{a}~.
\end{flalign}
\item[(iv)] $\Sigma : {}^K{}_{A}{\cal M}_{A} \to {}^{K_\sigma}{}_{{}_\sigma A}{\cal M}_{{}_\sigma A}$; the ${}_\sigma A$-bimodule
structures are defined by (\ref{eqn:leftleft}) and (\ref{eqn:leftright}).
\end{itemize}
In the cases (ii), (iii) and (iv) we have that the maps in 
(\ref{nt-left}) are isomorphisms  in the
corresponding categories ($  {}^{K_\sigma}{}_{{}_\sigma A}{\cal M}$, $ {}^{K_\sigma}{\cal M}_{{}_\sigma A}$ and $  {}^{K_\sigma}{}_{{}_\sigma A}{\cal M}_{{}_\sigma A}$ respectively); 
they  are the components of the natural isomorphism
 $\varphi^{\ell} : \,{{}^{\sigma\!\!\:}\otimes}\, \circ (\Sigma\times\Sigma)\Rightarrow \Sigma\circ \otimes$,
 where the bifunctors $\otimes$ are the left comodule analogues of those
  in Remarks \ref{rem:modulecategory}, \ref{rem:modulecategory2}
  and \ref{rem:bifunctorAMAH}. 

In particular (ii) and (iii)  induce
 the following equivalences of $^K{\!\cal M}$- and $^{K_\sigma}{\!\cal M}$-module categories:
 $(^K{}_A{\cal M},\ot)\simeq (^{K_\sigma}{\!}_{{}_\sigma A}{\cal M},\,{}^{\sigma\!\!\:}\ot)$ and 
$(^K{\cal M}_A,\ot)\simeq (^{K_\sigma}{\!\cal M}_{{}_\sigma A},{}^{\sigma\!\!\:}\ot)$.
\end{prop}

\subsubsection{\label{secKHbicomod}Twisting of $(K,H)$-bicomodules}
Let $K$ and $H$ be two (in general different) Hopf algebras.
As our last scenario we consider the situation where we have 
$\bbK$-modules $V$ together with a left $K$-coaction $\rho^V : V\to K\otimes V$
and a right $H$-coaction $\delta^V : V\to V\otimes H$ which are compatible in the sense 
of $(K,H)$-bicomodules, i.e.,\
\be\label{compatib}
(\da^V \ot \id)\circ \delta^V = (\id \ot \delta^V ) \circ \da^V ~.
\ee
Evaluated on an element $v \in V$ this condition reads
\be\label{hyp:cov}
\mone{(\zero{v})} \ot  \mzero{(\zero{v})} \ot \one{v} =
\mone{v}\ot \zero{(\mzero{v})} \ot \one{(\mzero{v})} =: \mone{v}\otimes \zero{v}\otimes  \one{v}~.
\ee
We denote by $\KMH$ the category of $(K,H)$-bicomodules, where
$\KMH$-morphisms are $\bbK$-linear maps that are both ${\cal M}^{H}$-comodule and ${}^K{\!\cal M}$-comodule morphisms.
It is a monoidal category; the tensor product 
of $V,W\in{}^K{\!\cal M}{}^{H}$ is the object $V\otimes W\in{}^K{\!\cal M}{}^{H}$
with left $K$-comodule structure $\rho^{V\otimes W}$ given in (\ref{eqn:lefttensor})
and right $H$-comodule structure $\delta^{V\otimes W}$ given in  (\ref{deltaVW}).
Notice that $\delta^{V\otimes W}$ and $\rho^{V\otimes W}$ are compatible
in the sense of (\ref{compatib}), 
\begin{eqnarray}
\nn( \rho^{V \ot W} \ot \id )\circ  \delta^{V \ot W} (v \ot w) &=&
\mone{(\zero{v})}\mone{(\zero{w})} \ot \mzero{(\zero{v})}\ot \mzero{(\zero{w})} \ot \one{v}\one{w}\\
\nn &=& \mone{v}\mone{w}\otimes\zero{v}\otimes\zero{w}\otimes \one{v}\one{w}
\\\nn &=& \mone{v}\mone{w} \ot  \zero{(\mzero{v})}\ot \zero{(\mzero{w})} \ot
\one{(\mzero{v})}\one{(\mzero{w})} 
\\ \nn&=& \mone{v}\mone{w} \ot \delta^{V \ot W} \left( \mzero{v}\ot \mzero{w}\right)
\\
&=& ( \id \ot \delta^{V \ot W})\circ  \rho^{V \ot W} (v \ot w) ~,
\end{eqnarray}
for all $v \in V$ and $w \in W$, where in the second and third passage
we have used that both $V$ and $W$ are objects in ${}^K\!\mathcal{M}^H$ and so their coactions 
 satisfy the compatibility condition \eqref{hyp:cov}. 
 \sk
 
Given a 2-cocycle $\sigma : K\otimes K\to\bbK$ and a 2-cocycle $\cot :H\otimes H\to \bbK$,
we have by \S \ref{secleftKcomod} and \S \ref{sec:rightHcomod}
the monoidal functors $(\Sigma,\varphi^\ell) : ( {}^{K}{\cal M}^H , \otimes) \to ({}^{K_\sigma}{\cal M}^H , \,{{}^{\sigma\!\!\:} \otimes}\,)$
and $(\Gamma ,\varphi) :  ( {}^{K}{\cal M}^H , \otimes) \to ({}^{K}{\cal M}^{H_\cot} , \otimes^\cot)$.
We therefore can construct two monoidal functors
\begin{eqnarray}\label{eqn:leftrightquant}
(\Sigma,\varphi^\ell) \circ (\Gamma , \varphi) &:& ( {}^{K}{\!\cal M}^H , \otimes) \longrightarrow ({}^{K_\sigma}{\!\cal M}^{H_\cot} , \,{{}^{\sigma\!\!\:} \otimes^\cot}\,)~,\nn \\
  (\Gamma,\varphi) \circ (\Sigma , \varphi^\ell) &:& ( {}^{K}{\!\cal M}^H , \otimes) \longrightarrow ({}^{K_\sigma}{\!\cal M}^{H_\cot} , \,{{}^{\sigma\!\!\:} \otimes^\cot}\,)~.
\end{eqnarray}
\begin{prop}\label{monfuneq}
The two monoidal functors in (\ref{eqn:leftrightquant}) are equal.
\end{prop}
\begin{proof}
As functors, $\Sigma\circ \Gamma$ is equal to $\Gamma\circ \Sigma$ as
both functors act as the identity on objects and on morphisms.
Thus, we just have to prove that the diagram

\begin{flalign}\label{commdiagSG}
\xymatrix{
\ar[d]_-{\varphi^\ell_{V_\cot,W_\cot}}{}_\sigma V_{\cot} \,{{{}^{\sigma\!\!\:}}\otimes^\cot}\,{}_\sigma W_\cot 
\ar[rr]^-{\varphi_{{}_\sigma V,{}_\sigma W}} && 
({}_\sigma V \,{{}^{\sigma\!\!\:} \otimes}\, {}_\sigma W)_\cot \ar[d]^-{\Gamma(\varphi^\ell_{V,W})}\\
{}_\sigma (V_\cot\otimes^\cot W_\cot)\ar[rr]^-{\Sigma(\varphi_{V,W})}&&{}_\sigma(V\otimes W)_\cot
}
\end{flalign}
in ${}^{K_\sigma}{\!\cal M}^{H_\cot}$ commutes, for all objects $V,W\in{}^K{\!\cal M}^{H}$;
indeed,
\begin{flalign*}
\varphi^{\ell}_{V,W} \big( \varphi_{{}_\sigma V,{}_\sigma W} (v \,{{}^{\sigma\!\!\:}\otimes^\cot}\, w)\big) &=\varphi^\ell_{V,W}(\zero{v}\,{{}^{\sigma\!\!\:}\otimes}\, \zero{w})~\bar\cot\left(\one{v}\otimes\one{w}\right)\\
&= \sig{\mone{(\zero{v})}}{\mone{(\zero{w})}}~\mzero{(\zero{v})}\otimes\mzero{(\zero{w})} ~\bar\cot\left(\one{v}\otimes\one{w}\right)\\
&=\sig{\mone{v}}{\mone{w}}~\zero{(\mzero{v})}\otimes\zero{(\mzero{w})} ~\bar\cot\left(\one{(\mzero{v})}\otimes\one{(\mzero{w})}\right)\\
&=\sig{\mone{v}}{\mone{w}}~\varphi_{V,W}(\mzero{v}\otimes^\cot \mzero{w})\\
&=\varphi_{V,W}\big(\varphi_{V_\cot,W_\cot}^\ell (v\,{{}^{\sigma\!\!\:}\otimes^\cot}\, w)\big)~,
\end{flalign*}
for all $v\in V$ and $w\in W$. In the third equality we have used the bicomodule property \eqref{hyp:cov} for $V$ and $W$.
\end{proof}
In short the above proposition states that it does not matter if we first deform by $\sigma$ and then by $\cot$
or if we first deform by $\cot$ and then by $\sigma$. 

\sk
Let us now consider the category ${}^K{\!\cal A}{}^{H}$
of $(K,H)$-bicomodule algebras, where objects and morphisms are  in
${}^K{\!\cal A}{}^{H}$ if they are in  ${\cal A}{}^{H}$,
${}^K{\!\cal A}$ and ${}^K{\!\cal M}{}^{H}$.
For
$A\in {}^K{\!\cal A}{}^{H}$, we further consider the categories of relative Hopf modules  ${}^K{\!}_{A}{\cal M}^{H}$,
${}^K{\!\cal M}_{A}{}^{H}$ and ${}^K{\!}_{A}{\cal M}_{A}{}^H$, where
by definition
objects and morphisms are in ${}^K{\!}_{A}{\cal M}_{A}{}^H$ if they
are in ${\!}_{A}{\cal M}_{A}{}^H$,  ${}^K{\!}_{A}{\cal M}_{A}$ and
${}^K{\!{\cal M}}{}^H$ (the categories ${}^K{\!}_{A}{\cal
  M}^{H}$ and
${}^K{\!\cal M}_{A}{}^{H}$ are similarly defined).
In complete analogy to Propositions \ref{propo:algdef}, \ref{propo:moddef}
and \ref{propo:leftdef} and because of the compatibility between the $K$-
and the $H$-coactions we obtain the following
\begin{prop}\label{prop:leftrightdef}
Let $\sigma: K\otimes K\to \bbK$ and $\cot : H\otimes H\to\bbK$ be two  $2$-cocycles.
Then the  monoidal functor $(\Gamma,\varphi) \circ (\Sigma,\varphi^\ell) = 
(\Sigma,\varphi^\ell)\circ (\Gamma,\varphi): {}^K{\cal M}^{H}\to {}^{K_\sigma}{\cal M}^{H_\cot}$
leads to the following functors, which induce equivalences of categories.
\begin{itemize}
\item[(i)] $\Gamma\circ \Sigma = \Sigma\circ \Gamma : {}^K{\!\cal
    A}{}^H\to {}^{K_\sigma}{\!\cal A}{}^{H_\cot}$; the deformed products are defined by
\begin{flalign}
{}_\sigma m_\cot : {}_\sigma A_\cot \,{^{\sigma\!\!\:}\otimes^\cot}\,{}_\sigma A_\cot &\longrightarrow {}_\sigma A_\cot~,\\
\nn a\,{{}^{\sigma\!\!\:}\otimes^\cot}\, a^\prime & \longmapsto \sig{\mone{a}}{\mone{a'}} \, \zero{a} \zero{a'}
~\bar\cot\left(\one{a}\otimes \one{a'}\right) =: a \,{{}_\sigma\bullet_\cot}\, a' ~.
\end{flalign}
\item[(ii)] $\Gamma\circ \Sigma = \Sigma\circ \Gamma : {}^K{\!}_{A}{\cal M}{}^{H} \to {}^{K_\sigma}{\!}_{{}_\sigma A_\cot}{\cal M}{}^{H_\cot}$; the left ${}_\sigma A_\cot$-actions 
are defined by
\begin{flalign}\label{eqn:leftleftbi}
\triangleright_{{}_\sigma V_\cot} : {}_\sigma A_\cot\, {{}^{\sigma\!\!\:} \otimes^\cot}\,{}_\sigma V_\cot &\longrightarrow {}_\sigma V_\cot~,\\
\nn a\,{{}^{\sigma\!\!\:}\otimes^\cot}\, v&\longmapsto \sig{\mone{a}}{\mone{v}}\,\zero{a}\tra_V\zero{v}~\bar\cot\left(\one{a}\otimes \one{v}\right)~.
\end{flalign}
\item[(iii)]  $\Gamma\circ \Sigma = \Sigma\circ \Gamma : {}^K{\cal M}_{A}{}^H \to {}^{K_\sigma}{\!\cal M}_{{}_\sigma A_\cot}{}^{H_\cot}$; 
the right ${}_\sigma A_\cot$-actions are defined by
\begin{flalign}\label{eqn:leftrightbi}
\triangleleft_{{}_\sigma V_\cot} : {}_\sigma V_\cot\, {{}^{\sigma\!\!\:} \otimes^\cot}\,{}_\sigma A_\cot &\longrightarrow {}_\sigma V_\cot~,\\
\nn v\,{{}^{\sigma\!\!\:}\otimes^\cot}\, a&\longmapsto \sig{\mone{v}}{\mone{a}}\,\zero{v}\trl_V \zero{a}~\bar\cot\left(\one{v}\otimes \one{a}\right)~.
\end{flalign}
\item[(iv)] $\Gamma\circ \Sigma = \Sigma\circ\Gamma : {}^K{\!}_{A}{\cal M}_{A}{}^H \to {}^{K_\sigma}{\!}_{{}_\sigma A_\cot}{\cal M}_{{}_\sigma A_\cot}{}^{H_\cot}$; 
the ${}_\sigma A_\cot$-bimodule
structures are defined by (\ref{eqn:leftleftbi}) and (\ref{eqn:leftrightbi}).
\end{itemize}
In the cases (ii), (iii) and (iv),  
$\Gamma(\varphi^\ell_{V,W})\circ \varphi_{{}_\sigma V,{}_\sigma W}$
are isomorphisms  in the
corresponding categories ($ {}^{K_\sigma}{\!}_{{}_\sigma A_\cot}{\cal
  M}{}^{H_\cot}$, $ {}^{K_\sigma}{\!\cal M}_{{}_\sigma
  A_\cot}{}^{H_\cot}$ and $  {}^{K_\sigma}{\!}_{{}_\sigma A_\cot}{\cal
  M}_{{}_\sigma A_\cot}{}^{H_\cot}$ respectively); 
they are the components of the natural isomorphism
 $\Gamma(\varphi^\ell)\circ \varphi 
  : \,{{}^{\sigma\!\!\:}\otimes^\cot}\, \circ (\Gamma\circ \Sigma\times\Gamma \circ \Sigma)\Rightarrow \Gamma\circ \Sigma\circ \otimes$,
 where the bifunctors $\otimes$ are the bicomodule analogues of those
  in the Remarks \ref{rem:modulecategory}, \ref{rem:modulecategory2}
  and \ref{rem:bifunctorAMAH}.

In particular {\it (ii)} and {\it (iii)} induce the following equivalences of
 ${}^K{\!\cal M}{}^H$- and  ${}^{K_\sigma}{\!\cal
  M}{}^{H_\gamma}$-module categories:  $({}^K{\!}_A{\cal
  M}^H,\ot)\simeq ({}^{K_\sigma}{\!}_{{}_\sigma A_\cot}{\cal
  M}^{H_\cot}, {}^{\sigma\!\!\:}\ot^\cot)$ and 
$({}^K{\!\cal M}_A{}^H,\ot)\simeq ({}^{K_\sigma}{\!\cal M}_{{}_\sigma
  A_\cot}{}^{H_\cot}, {}^{\sigma\!\!\:}\ot^\cot)$.

\end{prop}

\section{Twisting of Hopf-Galois extensions}\label{sec:twHG}
Suppose $B=A^{coH}\subseteq A$ is an $H$-Hopf-Galois extension with \textit{total space} $A$, 
\textit{base space} $B$ and \textit{structure Hopf algebra} $H$ (see Definition \ref{def:hg}).
We are interested in studying how  the invertibility of the canonical map $\can$
 behaves  under deformations via 2-cocycles. We are in particular interested in deforming classical principal bundles (cf.\ Example \ref{ex:principalbundle})
in order to obtain {\it noncommutative principal bundles} or {\it quantum
principal bundles}, i.e., principal
comodule algebras (cf.\ Definition \ref{def:pHcomodalg}) obtained via
deformation of principal bundles. 
Let us observe that if we consider a 2-cocycle
on the structure Hopf algebra $H$ and use it to twist the data $(A,B,H)$, the deformation of the base space $B$
turns out to be trivial as a direct consequence of the triviality of the right $H$-coaction on coinvariants.
In the language of noncommutative principal bundles this means that by twisting a classical principal bundle with a 2-cocycle on $H$, we 
only have the possibility to obtain a noncommutative principal bundle with a classical (i.e.\ not deformed) base space.
In order to obtain a more general theory, which also allows for
deformations of the base space, we shall also consider the case of $A$ carrying an external 
symmetry (described by a second Hopf algebra $K$) that is compatible with the right $H$-comodule structure. 
Indeed, by assuming the total space $A$ to be a $(K,H)$-bicomodule algebra, we can use a 
2-cocycle on $K$ to induce a deformation of $A$, which in general also
deforms the subalgebra $B$ of $H$-coinvariants.

Notice that it would be also possible to develop this theory by assuming the existence of an action of a  Hopf algebra 
$\U$ (dual to $K$) and use a twist $\F \in \U \ot \U$, rather than a
2-cocycle on $K$, to implement the
deformation (see the discussion in Appendix \ref{app:twists}).   Nevertheless, we shall use here 
the language of coactions as usual in the literature on Hopf-Galois extensions. 

\begin{ex}
In the setting of Example \ref{ex:principalbundle}, a natural choice for the Hopf algebra 
$\U$ is the universal enveloping algebra of the Lie algebra of $G$-equivariant vector fields on $P$,
i.e.\ the Lie algebra of derivations of $A=C^\infty(P)$ which commute with the right $G$-action.
The Hopf algebra $\U$ describes the infinitesimal automorphisms of the principal $G$-bundle
$\pi:P\to M$.  A natural choice for the Hopf algebra $K$ would be the Hopf algebra of functions
on a finite-dimensional Lie subgroup of the group of automorphisms $\phi:P \ra P$ of the bundle.
\end{ex}
\sk

We therefore consider the following three scenarios:
\begin{enumerate}
\item[\S\ref{sec:def_sg}] A deformation based on a $2$-cocycle on the structure Hopf algebra $H$. Here the total space $A$ and the structure Hopf algebra
$H$ are twisted, while the base space $B$ is undeformed.

\item[\S\ref{sec:def_es}] A deformation based  on a $2$-cocycle on an external Hopf algebra $K$ of symmetries  of  $A$ with $H$-equivariant 
coaction. Here the total space $A$ and the base $B$ are twisted, while the structure Hopf algebra $H$ is undeformed.

\item[\S\ref{sec:combidef}] The  combination of the previous two cases. Here the total space $A$, 
the structure Hopf algebra $H$ and the base space $B$ are all twisted in a compatible way.
 \end{enumerate}

In all these cases we shall show that
Hopf-Galois extensions and  principal comodule algebras are respectively
deformed into  Hopf-Galois extensions and  principal comodule algebras.
In particular principal bundles are deformed into noncommutative
principal bundles.
Our proof relies on relating the canonical map of the
twisted bundle with the canonical map  of the original bundle via a
commutative diagram in the appropriate category.
\sk

\subsection{Deformation via a 2-cocycle on the structure Hopf algebra $H$}\label{sec:def_sg}

Let $\cot:H\otimes H\to\bbK$ be a  2-cocycle on $H$ which we use to deform 
 $H$ into a new Hopf algebra $\hg$ with the same co-structures and unit, but different product
 and antipode given in Proposition \ref{prop:co}. 
 Using the techniques from \S \ref{sec:rightHcomod}, we can deform $A\in {\cal A}^H$
 into $A_\cot\in {\cal A}^{H_\cot}$ by introducing the twisted product \eqref{rmod-twist}. 
 As we have already observed above, the algebra structure of 
the subalgebra of $H$-coinvariants $B \subseteq A$ does not change
under our present class of 2-cocycle deformations, since the coaction of $H$ on the elements of $B$ is trivial.
 In other words, the subalgebra of coinvariants
 $B_\cot=A_\cot^{coH_\cot}$ of $\pg$ 
 is isomorphic (via the identity map) to $B=A^{coH}$ as an algebra (see \eqref{rmod-twist}). 
\sk

We shall relate the twisted canonical map $\chi_\cot: \pg \ot_B^\cot \pg \ra  \pg \ot^\cot  \underline{\hg}$ with
 the original one $\chi: A \ot_B A \ra  A \ot  \underline{H}$ by understanding both as morphisms
 in the category ${}_{\pg}{\mathcal{M}}_{\pg}{}^{\hg}$. Our strategy is as follows:
First, we notice that applying the functor
 $\Gamma: {}_{A}\mathcal{M}_A{}^{H} \ra  {}_{\pg}\mathcal{M}_{\pg}{}^{\hg}$ from Proposition \ref{propo:moddef} (iii)
on the original canonical map $\chi$  (which is a morphism in ${}_{A}\mathcal{M}_A{}^{H} $, cf.\ Proposition \ref{prop_canMorph}),
 we obtain the morphism $\Gamma(\chi): (A \ot_B A)_\cot \to (A \ot  \underline{H})_\cot$ in $_{\pg}\mathcal{M}_{\pg}{}^{\hg}$.
Next, we relate the two morphisms $\Gamma(\chi)$ and $\chi_\cot$ in
$_{\pg}\mathcal{M}_{\pg}{}^{\hg}$ via the natural transformation
$\varphi_{\text{--},\text{--}}$ (cf.\ (\ref{nt})) and an isomorphism $\Q$ introduced in Theorem \ref{prop:mapQ} below 
after a few technical lemmas.  The role of the isomorphism $\Q$ is to relate the 
two twist deformations of $H$ into right $H_\cot$-comodule coalgebras:
${\underline{H}}_\cot$ and
${\underline{H_\cot}}$; recall Example \ref{Hcomco}. While 
${\underline{H}}_\cot$ is the deformation of the $H$-comodule
coalgebra $\underline{H}_{}$, in ${\underline{H_\cot}}$ we  first
deform the Hopf algebra $H$ to $H_\cot$ and then regard it as an  
$H_\cot$-comodule coalgebra. The isomorphism $\Q$ is related to
the natural isomorphism proving the equivalence of the categories of
Hopf algebra modules and twisted Hopf algebra 
modules as closed categories, cf.\ Appendix \ref{appC}.
\sk

Let us now discuss the construction in detail. 
By Proposition \ref{prop_canMorph} we have that 
the canonical map $\chi: A \ot_B A \ra  A \ot  \underline{H}$ is a morphism in ${}_{A}\mathcal{M}_A{}^{H}$.
Applying the functor $\Gamma: {}_{A}\mathcal{M}_A {}^H\ra  {}_{\pg}\mathcal{M}_{\pg}{}^{\hg}$ 
from Proposition \ref{propo:moddef} (iii)
we obtain the ${}_{\pg}\mathcal{M}_{\pg}{}^{\hg}$-morphism $\Gamma(\chi)=\chi : (A\ot_B A)_\cot\to (A\ot \underline{H})_\cot$,
where $(A\ot_B A)_\cot=\Gamma(A \ot_B A )$ and $(A\ot
\underline{H})_\cot=\Gamma(A \ot  \underline{H})$ are objects in
${}_{\pg}\mathcal{M}_{\pg}{}^{\hg}$, while $A \ot_B A $ and $ A \ot
\underline{H}$ are in  ${}_{A}\mathcal{M}_A{}^H$.
Explicitly, the left $\pg$-action on $(A \ot_B A)_\cot $ reads
\begin{eqnarray}
\tra_{(A \ot_B A)_\cot }: \pg \ot^\cot (A \ot_B A)_\cot &\longrightarrow& (A \ot_B A)_\cot ~, \\
 c \ot^\cot (a \ot_B a') &\longmapsto & \zero{c} \zero{a} \ot_B \zero{a'} \,\coin{\one{c}}{\one{a}\one{a'}} \nn~.
\end{eqnarray}
The right $\hg$-coaction on $(A \ot_B A)_\cot$ is given by the right
$H$-coaction on $A\otimes_B A$, and the right $\pg$-action reads
\begin{eqnarray}
\trl_{(A \ot_B A)_\cot }:  (A \ot_B A)_\cot \ot^\cot \pg&\longrightarrow& (A \ot_B A)_\cot ~, \\
(a\otimes_B a^\prime)\otimes^\cot c &\longmapsto &   \zero{a} \ot_B \zero{a'}  \zero{c} 
\, \coin{\one{a}\one{a'}}{\one{c}}~.\nn
\end{eqnarray}
Analogously, on $(A \ot \underline{H})_\cot $ the left $\pg$-action reads
\begin{eqnarray}
\tra_{(A \ot \underline{H})_\cot }: \pg \ot^\cot (A \ot \underline{H})_\cot &\longrightarrow& (A \ot \underline{H})_\cot~, \\
  c \ot^\cot (a \ot h) &\longmapsto & \zero{c} \zero{a} \ot \two{h}\, \coin{\one{c}}{\one{a}S(\one{h})\three{h}}~,\nn
\end{eqnarray}
 the right $\hg$-coaction on $(A \ot \underline{H})_\cot$ is
 given by the right $H$-coaction on $A\otimes\underline{H}\,$, and the
 right $\pg$-action reads
\begin{eqnarray}\label{rightAHgamma} 
\trl_{(A \ot \underline{H})_\cot }:  (A \ot \underline{H})_\cot \ot^\cot \pg&\longrightarrow& (A \ot \underline{H})_\cot~, \\
 (a\ot h)\ot^\cot c&\longmapsto & \zero{a} \zero{c} \ot \two{h}\one{c} 
\, \coin{\one{a}S(\one{h})\three{h}}{\two{c}}~.\nn
\end{eqnarray}
\sk

We proceed with the second step and introduce the isomorphism $\Q$ mentioned above relating  
the two deformations of  $H$ when thought of as a Hopf algebra or as a right $H$-comodule coalgebra 
$\underline{H}$.  We first need the following technical lemmas. 
Recall the definition of $u_\cot$ and its convolution inverse from \eqref{uxS}.
\begin{lem}\label{lemma2co}
Every 2-cocycle $\cot:H\otimes H\to \bbK$ satisfies
\be\label{simpl0}
u_\cot(\one{h})\bar\cot(S(\two{h})\otimes k)=\cot(\one{h}\otimes
S(\two{h})k)~,
\ee 
\be\label{simpl}
\bar{u}_\cot (\one{h})\co{\two{h}}{{k}} = \coin{S(\one{h})}{\two{h}{k}}~,
\ee
\be\label{abc}
\co{g\one{h}}{S(\two{h})k} = \coin{\one{g}}{\one{h}} ~u_\cot(\two{h}) ~\coin{S(\three{h})}{\one{k}} ~
\co{\two{g}}{\two{k}}~,
\ee
for all $g,h,k\in H$.
 \end{lem}
\begin{proof}
Recalling the definition of $u_\cot$ and Lemma \ref{lem:formula}
\eqref{iv} we have (\ref{simpl0}). Similarly from Lemma \ref{lem:formula}
\eqref{iii} we have  (\ref{simpl}). From Lemma \ref{lem:formula}  \eqref{iii}  we also
obtain (recall $\bar\cot\ast\cot=\varepsilon\otimes\varepsilon$)
\begin{eqnarray*}
\co{gh}{k}= \coin{\one{g}}{\one{h}} \co{\two{h}}{\one{k}} \co{\two{g}}{\three{h}\two{k}}~, 
\end{eqnarray*}
for all $g,h,k \in H$. 
Use of this identity and of (\ref{simpl0}) proves (\ref{abc}):
\begin{eqnarray*}
\co{g\one{h}}{S(\two{h})k} &=&  \coin{\one{g}}{\one{h}} \co{\two{h}}{S(h_{\scriptscriptstyle{(5)}})\one{k}} \co{\two{g}}{\three{h}S(h_{\scriptscriptstyle{(4)}})\two{k}}
\\
&=&  \coin{\one{g}}{\one{h}} \co{\two{h}}{S(h_{\scriptscriptstyle{(3)}})\one{k}} \co{\two{g}}{\two{k}}
\\
&=& 
\coin{\one{g}}{\one{h}} ~u_\cot(\two{h}) ~\coin{S(\three{h})}{\one{k}} ~
\co{\two{g}}{\two{k}}~.
\end{eqnarray*}
\end{proof}

\begin{lem}
Let $\underline{H}$ be the right $H$-comodule coalgebra 
with adjoint coaction $\delta^{\underline{H}} = \mathrm{Ad}: \underline{H}\to\underline{H}\otimes H\,,~ 
h \mapsto \two{h}\otimes S(\one{h})\,\three{h}$.
Then the twisted right $\hg$-comodule coalgebra $\underline{H}_\cot$  has coproduct 
\be\label{cop-ad}
\Delta_\cot (h) = \three{h} \ot^\cot  h_{\scriptscriptstyle{(7)}} \,
\coin{S(\two{h})}{\four{h}} ~
u_\cot (h_{\scriptscriptstyle{(5)}}) ~
\coin{S(h_{\scriptscriptstyle{(6)}})}{h_{\scriptscriptstyle{(8)}}}
~
\co{S(h_{\scriptscriptstyle{(1)}})}{h_{\scriptscriptstyle{(9)}}}~.
\ee
\end{lem}
\begin{proof}
By the general theory, the twisted right $\hg$-comodule coalgebra $\underline{H}_\cot$  has coproduct given  by
$\Delta_\cot (h) = \two{h} \ot^\cot  h_{\scriptscriptstyle{(5)}} 
\co{S(\one{h})\three{h}}{S(\four{h})h_{\scriptscriptstyle{(6)}} } $
(see  \eqref{cc-twist}), then
\eqref{cop-ad} follows by applying \eqref{abc} above.
\end{proof}

On the other hand, the twisted Hopf algebra $\hg$  
can be considered as a right $\hg$-comodule coalgebra, denoted by $\underline{\hg}$, 
via the $\hg$-adjoint coaction 
\be\label{Ad-cot}
\delta^{\underline{\hg}} = \mathrm{Ad}_{\cot}: \underline{\hg} \longrightarrow \underline{\hg}\otimes \hg~,~~
h \longmapsto \two{h}\otimes S_\cot (\one{h}) \mt \three{h} ~ 
\ee
and the coproduct $\Delta : \underline{\hg}\to \underline{\hg}\otimes^\cot \underline{\hg}\,,~h \mapsto \one{h}\otimes^\cot \two{h}$.

\begin{thm}\label{prop:mapQ}
The $\bbK$-linear map 
\begin{eqnarray}\label{mapQ}
\Q : \underline{\hg} \longrightarrow \underline{H}_\cot ~,~~
h \longmapsto \three{h} \, u_\cot(\one{h}) \, \coin{S(\two{h})}{\four{h}}
\end{eqnarray}
is an isomorphism of right $\hg$-comodule coalgebras, with inverse
\be\label{mapQinv}
\Q^{-1} : \underline{H}_\cot \longrightarrow \underline{\hg} ~,~~
h \longmapsto \three{h} \, \bar{u}_\cot (\two{h}) \, \co{S(\one{h})}{\four{h}}~.
\ee
\end{thm}
\begin{proof}
It is easy to prove by a direct calculation that $\Q^{-1}$ is the inverse of $\Q$. 
We now show that $\Q$ is a right $\hg$-comodule morphism, for all $h \in \underline{\hg}$,
\begin{flalign*}
&(\Q \ot \id) (\mathrm{Ad}_\cot (h) )=
\\
&\qquad 
= \Q(\two{h}) \ot S_\cot (\one{h}) \mt \three{h}
\\
&\qquad=
\Q(\four{h}) \ot u_\cot(\one{h}) S (\two{h}) \bar{u}_\cot(\three{h}) \mt h_{\scriptscriptstyle{(5)}}
\\
&\qquad=
u_\cot (h_{\scriptscriptstyle{(6)}}) h_{\scriptscriptstyle{(8)}} \coin{S(h_{\scriptscriptstyle{(7)}})}{h_{\scriptscriptstyle{(9)}}} \ot u_\cot(\one{h}) \bar{u}_\cot (h_{\scriptscriptstyle{(5)}}) \co{S(\four{h})}{h_{\scriptscriptstyle{(10)}}} S(\three{h}) h_{\scriptscriptstyle{(11)}}
\coin{S(\two{h})}{h_{\scriptscriptstyle{(12)}}}
\\
&\qquad=
h_{\scriptscriptstyle{(6)}} \coin{S(h_{\scriptscriptstyle{(5)}})}{h_{\scriptscriptstyle{(7)}}} \ot u_\cot(\one{h})  \co{S(\four{h})}{h_{\scriptscriptstyle{(8)}}} S(\three{h}) h_{\scriptscriptstyle{(9)}}
\coin{S(\two{h})}{h_{\scriptscriptstyle{(10)}}}
\\
&\qquad=
u_\cot(\one{h}) h_{\scriptscriptstyle{(4)}}  \ot S(\three{h}) {h_{\scriptscriptstyle{(5)}}} 
\coin{S(\two{h})}{h_{\scriptscriptstyle{(6)}}} =\mathrm{Ad}(\Q(h)) ~,
\end{flalign*}
where in the fourth passage we used
${u}_\cot(h_{\scriptscriptstyle{(6)}})\bar u_\cot(h_{\scriptscriptstyle{(5)}})=\varepsilon(h_{\scriptscriptstyle{(5)}})$, and
in the fifth $h_{\scriptscriptstyle{(6)}}\bar{\cot} (S(h_{\scriptscriptstyle{(5)}})\otimes
h_{\scriptscriptstyle{(7)}})\cot(S(\four{h})\otimes h_{\scriptscriptstyle{(8)}})=h_{\scriptscriptstyle{(5)}}\varepsilon(\four{h})\varepsilon(h_{\scriptscriptstyle{(6)}})$.
Next, we prove that $\Q$ is a coalgebra morphism, i.e.\ $(\Q \ot^\cot \Q)\circ \Delta= \Delta_\cot \circ \Q$,
\begin{eqnarray*}
(\Q \ot^\cot \Q) \circ \Delta \circ \Q^{-1} (h) &=& \bar{u}_\cot(\two{h}) \Q(\three{h}) \ot^\cot \Q(\four{h}) \co{S(\one{h})}{h_{\scriptscriptstyle{(5)}}}
\\
&=& \bar{u}_\cot(\two{h}) u_\cot(\three{h}) \coin{S(h_{\scriptscriptstyle{(4)}})}{h_{\scriptscriptstyle{(6)}}} h_{\scriptscriptstyle{(5)}} \ot^\cot \Q(h_{\scriptscriptstyle{(7)}})
\co{S(\one{h})}{h_{\scriptscriptstyle{(8)}}}
\\
&=& \coin{S(h_{\scriptscriptstyle{(2)}})}{h_{\scriptscriptstyle{(4)}}} h_{\scriptscriptstyle{(3)}} \ot^\cot \Q(h_{\scriptscriptstyle{(5)}})
\co{S(\one{h})}{h_{\scriptscriptstyle{(6)}}}
\\
&=&
\coin{S(\two{h})}{\four{h}}
 \three{h} \ot^\cot  h_{\scriptscriptstyle{(7)}} 
u_\cot(h_{\scriptscriptstyle{(5)}})
\coin{S(h_{\scriptscriptstyle{(6)}})}{h_{\scriptscriptstyle{(8)}}}
\co{S(h_{\scriptscriptstyle{(1)}})}{h_{\scriptscriptstyle{(9)}}} 
\\ &=&\Delta_\cot (h)~,
\end{eqnarray*}
for all $h\in \underline{H}_\cot $.
The last equality follows from comparison with \eqref{cop-ad}.
\end{proof}

\begin{rem} 
If we dualize this picture by considering a dually paired Hopf
algebra $H'$ (and dual modules) then the right  $H$-adjoint coaction
dualizes into the right $H'$-adjoint action,
$\zeta\blacktriangleleft \xi=S(\one{\xi})\zeta\two{\xi}$ for all
$\zeta,\xi \in H'$. If we further consider a
 mirror construction by using left  adjoint actions rather than right ones, 
 then the analogue of the isomorphism $\Q$ is the isomorphism $D$ studied in  \cite{NCG2} and more in
 general in \cite{AS}. Explicitly, as explained in Appendix \ref{appC},
 the isomorphism $\Q$ is dual to the isomorphism $D$ relative to the Hopf
 algebra ${H'^{op}}^{\,cop}$ with opposite product and coproduct;
 this latter is a component of a
 natural transformation determining the equivalence of the closed monoidal categories of
 left  ${H'^{op}}^{\,cop}$-modules and  left  ${(H_{\gamma}')^{op}}^{\,cop}$-modules.
\end{rem}

After these preliminaries we can now relate  the twisted canonical map $\chi_\cot$
with the original one $\chi$. 
\begin{thm}\label{theo:diagr-can-pre}
Let $H$ be a Hopf algebra and $A$ an $H$-comodule algebra.
Consider the algebra extension $B = A^{coH}\subseteq A$ and the associated
canonical map  $\can : A \otimes_B A \longrightarrow  A \ot H$.
Given a $2$-cocycle $\cot : H\otimes H\to\bbK$ the diagram
\begin{flalign}\label{diagr-can-pre}
\xymatrix{
\ar[dd]_-{\varphi_{A,A}} A_\cot \otimes_B^\cot A_\cot \ar[rr]^-{\chi_\cot} && A_\cot \otimes^\cot \underline{\hg}\ar[d]^-{\id\otimes^\cot \Q}\\
&& A_\cot \otimes^\cot \underline{H}_\cot \ar[d]^-{\varphi_{A,\underline{H}}}\\
(A\otimes_B A)_\cot \ar[rr]^-{\Gamma(\chi)} && (A\otimes \underline{H})_\cot
}
\end{flalign}
in  ${}_{A_\cot}{\cal M}_{A_\cot}{}^{H_\cot}$ commutes.
\end{thm}

\begin{proof}
A straightforward computation shows that   $\varphi_{A,A} : A_\cot \otimes^\cot A_\cot \to (A\otimes A)_\cot$
descends to a well-defined isomorphism on the quotients $\varphi_{A,A} : A_\cot \otimes_B^\cot A_\cot \to (A\otimes_B A)_\cot$, i.e.\
the left vertical arrow is well defined. 
We prove that the diagram \eqref{diagr-can-pre}  commutes.
We obtain for the composition $(\id\ot^\cot \Q) \circ \can_\cot$ the following expression
\begin{eqnarray*}
(\id \ot^\cot \Q) \big( \can_\cot (a \ot_B^\cot a') \big) &=& 
\zero{a} \zero{a'} \ot^\cot \Q (\two{a'}) \coin{\one{a}}{\one{a'}}
\\
&=& 
\zero{a} \zero{a'} \ot^\cot \four{a'} u_\cot(\two{a'}) \coin{S(a'_{\scriptscriptstyle{(3)}})}{a'_{\scriptscriptstyle{(5)}}}  \coin{\one{a}}{\one{a'}} ~.
\end{eqnarray*}
On the other hand, from \eqref{nt} and  \eqref{adj} we have 
$$
\varphi_{A,\underline{H}}^{-1}(a \ot h) =
\zero{a} \ot^\cot \two{h} \co{\one{a}}{S(\one{h})\three{h}} ~,
$$
so that for the composition $\varphi^{-1}_{A,\underline{H}}\circ \Gamma(\can)
\circ \varphi_{A,A}$ we obtain (recalling that $\Gamma(\can)=\can$)
\begin{flalign*}
&\varphi_{A,\underline{H}}^{-1} \left(\Gamma(\can) (\varphi_{A,A} (a \ot^\cot_B a'))\right) \\
 &\qquad = 
\varphi_{A,\underline{H}}^{-1}  (\zero{a} \zero{a'} \ot \one{a'}) \,\coin{\one{a}}{\two{a'}}
\\
&\qquad =
\zero{a} \zero{a'} \ot^\cot \three{a'}  \co{\one{a}\one{a'}}{S(a'_{\scriptscriptstyle{(2)}})a'_{\scriptscriptstyle{(4)}}} \coin{\two{a}}{a'_{\scriptscriptstyle{(5)}}}  
\\
&\qquad =
\zero{a} \zero{a'} \ot^\cot \four{a'}  \coin{\one{a}}{\one{a'}} u_\cot(\two{a'}) \coin{S(\three{a'})}{a'_{\scriptscriptstyle{(5)}}} \co{\two{a}}{a'_{\scriptscriptstyle{(6)}}}  \coin{\three{a}}{a'_{\scriptscriptstyle{(7)}}}~,
\end{flalign*}
where we have used \eqref{abc}. The last two terms simplify, giving the desired identity.

{}From the properties of the canonical map (Proposition
\ref{prop_canMorph}) and from  Proposition \ref{propo:moddef} it
immediately follows that all arrows in the diagram are ${}_{A_\cot}{\cal M}^{H_\cot}$-morphisms. 
In
\S\ref{sec:proof} below we introduce a right 
 $A_\cot$-module structure on $A_\cot\otimes^\cot \underline{H}_\cot$
 such that $A_\cot\otimes^\cot \underline{H}_\cot$ is an object in ${}_{A_\cot}{\cal
   M}_{A_\cot}{}^{H_\cot}$,
and  show that all arrows in the diagram  are also morphisms in ${}_{A_\cot}{\cal
   M}_{A_\cot}{}^{H_\cot}$.
\end{proof}

\begin{cor}\label{cor-can}
$B=A^{coH}\subseteq A$ is an $H$-Hopf-Galois extension if and only if $B \simeq \pg^{coH_\cot} \subseteq \pg$ is 
an $\hg$-Hopf-Galois extension. Moreover, $B\subseteq A$ is cleft if and only if $B \subseteq \pg$ is cleft.
\end{cor}
\begin{proof}
The main statement follows from the invertibility of the morphisms $\varphi_{A,\underline{H}}$,  $\varphi_{A,A}$ and $\Q$
in  diagram (\ref{diagr-can-pre}).
For the  statement about cleftness, recall, from the end of \S \ref{sec:HG},
that the Hopf-Galois extension $B \subseteq A$ is cleft if and only if
there exists an isomorphism $\theta:B \ot H \ra A$ of left $B$-modules and right $H$-comodules, 
where here $B \ot H$ is the object in $_B\mathcal{M}^H$ with left $B$-action given by $m_B \ot \id$ 
and right $H$-coaction given by $\id \ot \Delta$.  
Now, due to the functor $\Gamma: {}_{B}\mathcal{M}^H \ra  {}_{B_\cot}\mathcal{M}^{\hg}$ from Proposition \ref{propo:moddef} (i),
the $_B\mathcal{M}^H$-isomorphism $\theta$ defines the $_{B_\cot}\mathcal{M}^{\hg}$-isomorphism 
$\Gamma(\theta) : (B\otimes H)_\cot \to A_\cot$,
which composed with the $_{B_\cot}\mathcal{M}^{\hg}$-isomorphism $\varphi_{B,H}$ defines 
the $_{B_\cot}\mathcal{M}^{\hg}$-isomorphism $\theta_\cot : B_\cot\otimes^\cot H_\cot\to A_\cot$; explicitly,
\begin{flalign*}
\xymatrix{
\ar[d]_{\varphi_{B,H}} B_\cot \otimes^\cot H_\cot \ar[rr]^-{\theta_\cot} && A_\cot\\
(B\otimes H)_\cot\ar[rru]_-{\Gamma(\theta)}
}
\end{flalign*}
\end{proof}
Notice that since on $B=A^{coH}\subseteq A$ the $H$-coaction is
trivial,  it follows that $\varphi_{B,H}=\id$ and as $\bbK$-linear
maps $\theta_\cot=\theta$.

\begin{rem}
Montgomery and Schneider in 
\cite[Th. 5.3]{MS05} prove the above corollary by using that as 
$\bbK$-modules $A\ot_B A=A_\cot\ot_B A_\cot$ and
$A\ot H=A_\cot\ot H_\cot$, and showing that the canonical map $\can$ is the composition of $\can_\cot$ with
an invertible map. The proof is not within the natural categorical
setting of twists of Hopf-Galois extensions that we consider.
\end{rem}

Recalling from Definition \ref{def:pHcomodalg} the notion of principal $H$-comodule algebra
it is easy to show that deformations by $2$-cocycles $\cot:H\otimes H\to\bbK$ preserve this structure.
\begin{cor}\label{cor-pcomodalg}
$A$ is a principal $H$-comodule algebra if and only if $A_\cot$ is a principal $H_\cot$-comodule algebra.
\end{cor}
\begin{proof}
The ${}_{B_\cot}{\cal M}^{H_\cot}$-morphism $m_\cot : B_\cot\otimes^\cot A_\cot \to A_\cot$
is related to the ${}_B{\cal M}^H$-morphism $m: B\otimes A\to A$ via 
$m_\cot = \Gamma(m)\circ \varphi_{B,A}$.
Given now a ${}_B{\cal M}^H$-morphism $s: A\to B\otimes A$
that is a section of $m$, we define the ${}_{B_\cot}{\cal M}^{H_\cot}$-morphism $s_\cot := \varphi_{B,A}^{-1}\circ \Gamma(s)
: A_\cot \to B_\cot\otimes^\cot A_\cot$. We obtain the commutative diagram
\begin{flalign*}
\xymatrix{
\ar[drr]_-{\Gamma(s)}A_\cot \ar[rr]^-{s_\cot}
\ar@/^2pc/[rrrr]^-{\id=\Gamma(s\circ m)}  &&\ar[d]_-{\varphi_{B,A}} B_\cot\otimes^\cot A_\cot \ar[rr]^-{m_\cot} && A_\cot\\
&& (B\otimes A)_\cot \ar[rru]_-{\Gamma(m)}&& 
}
\end{flalign*}
in the category ${}_{B_\cot}{\cal M}^{H_\cot}$, from which it follows that $s_\cot$ is a section of $m_\cot$.
The reverse implication 
follows using the convolution inverse
$\bar\cot$ of $\cot$ that twists back $A_\cot$ to $A$ and $(B\ot
A)_\cot$ to $B\ot A$, so that, given the section $s_\cot$ of $m_\cot$, the section of $m$ is
$\overline\Gamma(\varphi_{B,A}\circ s_\cot)=\varphi_{B,A}\circ s_\cot:
A\to B\ot A$.
\end{proof}

\addtocontents{toc}{\protect\setcounter{tocdepth}{1}}
\subsubsection{
Completion of the proof of Theorem \ref{theo:diagr-can-pre} \label{sec:proof}
(the right $A_\cot$-module structure on $A_\cot\otimes^\cot \underline{H}_\cot$)}

\addtocontents{toc}{\protect\setcounter{tocdepth}{2}}
We here complete the proof of  Theorem \ref{theo:diagr-can-pre}, i.e.,
we show that the diagram (\ref{diagr-can-pre}) is a diagram in the category
${}_{A_\cot}{\cal M}_{A_\cot}{}^{H_\cot}$ (not just in
${}_{A_\cot}{\cal M}{}^{H_\cot}$). This is the case if
all morphisms in (\ref{diagr-can-pre}) are  in ${}_{A_\cot}{\cal M}_{A_\cot}{}^{H_\cot}$. By Proposition \ref{prop_canMorph} it is clear that the 
morphism  $\chi_\cot$ is in  ${}_{A_\cot}{\cal M}_{A_\cot}{}^{H_\cot}$, and using  Proposition \ref{propo:moddef} (iii)
we observe that also $\varphi_{A,A}$ and $\Gamma(\chi)$ are morphisms
in ${}_{A_\cot}{\cal   M}_{A_\cot}{}^{H_\cot}$.
In order to prove that the remaining morphisms $\id\otimes^\cot \Q$ and ${\varphi_{A,\underline{H}}}$
in the right column in (\ref{diagr-can-pre}) are morphisms in
${}_{A_\cot}{\cal   M}_{A_\cot}{}^{H_\cot}$,  we just have to introduce a right $A_\cot$-module structure on
$A_\cot\otimes^\cot \underline{H}_\cot$ and prove that these morphisms
are morphisms of right $A_\cot$-modules (Lemma
\ref{lem:diagr-can}). Indeed, since they are also ${}_{A_\cot}{\cal M}{}^{H_\cot}$-morphisms and furthermore they are bijective, then
$A_\cot\otimes^\cot \underline{H}_\cot$ is an object in
${}_{A_\cot}{\cal M}_{A_\cot}{}^{H_\cot}$ because it is  the  target  of
the first (or equivalently because it is the source of the second), and 
 $\id\otimes^\cot \Q$ and ${\varphi_{A,\underline{H}}}$ are then
 ${}_{A_\cot}{\cal M}_{A_\cot}{}^{H_\cot}$-isomorphisms. \sk 

We are therefore left to introduce  a right $A_\cot$-module structure on
$A_\cot\otimes^\cot \underline{H}_\cot$ and prove that 
$\id\otimes^\cot \Q$ and ${\varphi_{A,\underline{H}}}$
are right $A_\cot$-modules morphisms (Lemma \ref{lem:diagr-can}). To
this aim let us recall that the right $A$-action on $A\otimes \underline{H}$ is given by 
$(a\otimes h)\triangleleft_{A\otimes\underline{H}} c = a\zero{c}\otimes h\one{c}$, for all
$a,c\in A$ and $h\in \underline{H}$ (cf.\ (\ref{trl_ot})).
Applying Proposition \ref{propo:moddef} (iii), we observe that the right $A_\cot$-module structure on
$(A\otimes\underline{H})_\cot$ is given by
\begin{flalign}\label{3rightac}
(a\otimes h)\triangleleft_{(A\otimes\underline{H})_\cot} c = \zero{a} \zero{c}\otimes \two{h}\one{c} ~\coin{\one{a} S(\one{h})\three{h}}{\two{c}}~,
\end{flalign}
for all $a\in A$, $h\in\underline{H}$ and $c\in A_\cot$. Again by (\ref{trl_ot}), the right $A_\cot$-module structure
on $A_\cot \otimes^\cot \underline{\hg}$ reads
\begin{flalign}\label{1rightac}
(a\otimes^\cot h)\triangleleft_{A_\cot \otimes^\cot \underline{\hg}} c = a\mtco \zero{c} \otimes^\cot h\mt \one{c}~,
\end{flalign}
for all $a,c\in A_\cot$ and $h\in \underline{\hg}$.
The right $A_\cot$-module structure on $A_\cot \otimes^\cot \underline{H}_\cot$
is induced by the Hopf algebra structure on $\underline{H}_\cot$ that is inherited from
the Hopf algebra structure on $H_\cot$ and the isomorphism $\Q$ of Theorem \ref{prop:mapQ}.
Explicitly, we have the following Corollary of Theorem \ref{prop:mapQ}:
\begin{cor} 
The right $\hg$-comodule coalgebra isomorphism $\Q:\underline{\hg} \to
\underline{H}_\cot $ defines the $\bbK$-linear isomorphism $\Q:{\hg} \to
\underline{H}_\cot $ (with slight abuse of notation we use the same
letter $\Q$) that  induces on  $\underline{H}_\cot$ a Hopf algebra structure from the
Hopf algebra structure on $\hg$. Explicitly, 
we define the product $m_{\underline{H}_\cot}$ on $\underline{H}_\cot $ via the commutative diagram 
of $\bbK$-linear maps
\be\label{multHtilde}
\xymatrix{
 \ar[d]_-{\Q^{-1} \ot \Q^{-1}} \underline{H}_\cot\otimes \underline{H}_\cot \ar[rr]^-{m_{\underline{H}_\cot}} && \underline{H}_\cot \\
 \hg \otimes \hg \ar[rr]^-{m_\cot} &&  \hg \ar[u]_-{\Q}
}
\ee
and the antipode $S_{\underline{H}_\cot }$ on $\underline{H}_\cot$ via the commutative diagram of $\bbK$-linear maps
\be
\xymatrix{
 \ar[d]_{\Q^{-1}} \underline{H}_\cot \ar[rr]^-{S_{ \underline{H}_\cot}}&&  \underline{H}_\cot\\
\hg \ar[rr]^-{S_\cot}  && \hg\ar[u]_-{\Q}
}
\ee
By construction, $H_\cot$ and $\underline{H}_\cot $ are isomorphic Hopf algebras via the isomorphism $\Q$.
\end{cor}

As a simple consequence of this corollary, every right $H_\cot$-comodule is also a right $\underline{H}_\cot$-comodule;
just use the isomorphism $\Q$ between the Hopf algebras $H_\cot$ and $\underline{H}_\cot$
in order to turn a right $H_\cot$-comodule structure into a right $\underline{H}_\cot$-comodule structure.
In particular, we have that the right $H_\cot$-comodule algebra $\pg$ is a 
right $\underline{H}_\cot$-comodule algebra with coaction given by
\begin{eqnarray}
\underline{\delta}^{\pg}:=(\id \ot \Q)\circ \delta^{\pg} : \pg &\longrightarrow &\pg \ot \underline{H}_\cot~,
\\
a &\longmapsto& \zero{a} \ot \Q(\one{a}) ~.\nn
\end{eqnarray}
Using this right $\underline{H}_\cot$-comodule structure on $A_\cot$ we canonically define the right $A_\cot$-module structure on $A_\cot\otimes^\cot \underline{H}_\cot$ by (cf. (\ref{trl_ot})),
\begin{flalign}\label{2rightac} (a \ot^\cot h) \trl_{\pg \ot^\cot  \underline{H}_\cot} c= a \mtco \zero{c} \ot^\cot m_{\underline{H}_\cot}(h \ot \Q(\one{c}))~.
\end{flalign}

\begin{lem}\label{lem:diagr-can}
The vertical arrows $\id\otimes^\cot \Q$ and
$\varphi_{A,\underline{H}}$  in diagram (\ref{diagr-can-pre}) are
right ${A_\cot}$-module morphisms.
\end{lem}
\begin{proof} Using (\ref{1rightac}), (\ref{2rightac}) and
  (\ref{multHtilde}) we immediately obtain that $\id\otimes^\cot \Q$
is a morphism of right  $A_\cot$-modules:
\be
\xymatrix{
\ar[d]_-{\id \ot^\cot \Q \,\ot \,id} \pg\ot^\cot \underline{\hg} \ot \pg \ar[rr]^-{\trl_{\pg\ot^\cot \underline{\hg}}} &&  
\pg\ot^\cot \underline{\hg} \ar[d]^-{\id \ot^\cot \Q}
\\
\pg \ot^\cot  \underline{H}_\cot \ot \pg  \ar[rr]^-{\trl_{\pg \ot^\cot  \underline{H}_\cot}}&& \pg \ot^\cot  \underline{H}_\cot
}
\ee
We now show that $\varphi_{A,\underline{H}}:\pg \ot^\cot  \underline{H}_\cot
\ra (A\otimes \underline{H})_\cot$ is a morphism of right
$\pg$-modules. (This is automatic in the case of Hopf-Galois extension
because of the invertibility of all maps in the commutative diagram 
\eqref{diagr-can-pre}).

The structure of right $\pg$-modules of  $\pg \ot^\cot  \underline{H}_\cot$
was given just above in \eqref{2rightac}. We compute it explicitly by
using the expression of the product in $\underline{H}_\cot$ and
of the map $\Q^{-1}$:
\begin{eqnarray*}
(a \ot h) \trl_{\pg \ot^\cot  \underline{H}_\cot} c&=& a \mtco \zero{c} \ot 
m_{\underline{H}_\cot}(h \ot \Q(\one{c})) 
\\  && \!\!\!\!\!\!\!\!\!\!\!\!\!\! \hspace{-2cm} = 
a \mtco \zero{c} \ot  \Q(\four{h}\two{c}) \co{\three{h}}{\one{c}}\coin{\five{h}}{\three{c}} 
\bar{u}_\cot (\two{h}) \co{S(\one{h})}{h_{\scriptscriptstyle{(6)}}}
\\ &&\!\!\!\!\!\!\!\!\!\!\!\!\!\! \hspace{-2cm} =
a \mtco \zero{c} \ot  h_{\scriptscriptstyle{(6)}} \four{c} 
\co{\four{h}\two{c}}{S(\three{c})S(\five{h})h_{\scriptscriptstyle{(7)}}\five{c}}
\co{\three{h}}{\one{c}} \bar{u}_\cot (\two{h}) 
\coin{h_{\scriptscriptstyle{(8)}}}{{c}_{\scriptscriptstyle{(6)}}} 
\co{S(\one{h})}{h_{\scriptscriptstyle{(9)}}}
\end{eqnarray*}
where in the last passage we used the expression, 
$\Q(g)=\three{g}\cot(\one{g}\ot S(\two{g})\four{g})$, for all $g\in
H$, that immediately
follows from the definition of $\Q$ using (\ref{simpl0}).
Now we apply (\ref{simpl}) in the form 
$
\co{\three{h}}{\one{c}} \bar{u}_\cot (\two{h})= \coin{S(\two{h})}{\three{h}\one{c}}
$ and then, by applying Lemma \ref{lem:formula}  \eqref{iii} to the
resulting term  
$$\co{\four{h}\two{c}}{S(\three{c})S(\five{h})h_{\scriptscriptstyle{(7)}}\five{c}}\coin{S(\two{h})}{\three{h}\one{c}}$$
the above expression becomes
\begin{eqnarray*}
&&\!\!\!\!\!\!\!\!\!\!a \mtco \zero{c} \ot  h_{\scriptscriptstyle{(6)}} \five{c} 
\co{\one{c}}{S(\four{c})S(\five{h})h_{\scriptscriptstyle{(7)}}{c}_{\scriptscriptstyle{(6)}}}
\coin{S(\two{h})}{\three{h}\two{c}S(\three{c})S(\four{h}) h_{\scriptscriptstyle{(8)}}
c_{\scriptscriptstyle{(7)}}}
 \co{S(\one{h})}{h_{\scriptscriptstyle{(10)}}}\\
&&\qquad\qquad \qquad\qquad \qquad\qquad \qquad\qquad \qquad\qquad
   \qquad\qquad \qquad\qquad \qquad\qquad ~~~~~\,~\coin{h_{\scriptscriptstyle{(9)}}}{{c}_{\scriptscriptstyle{(8)}}} 
\\ &&\!\!\!\!\!\!\!\!\!\!=
a \mtco \zero{c} \ot  h_{\scriptscriptstyle{(4)}} \three{c} 
\co{\one{c}}{S(\two{c})S(\three{h}) h_{\scriptscriptstyle{(5)}}{c}_{\scriptscriptstyle{(4)}}}
\coin{S(\two{h})}{h_{\scriptscriptstyle{(6)}}\five{c}}
 \co{S(\one{h})}{h_{\scriptscriptstyle{(8)}}}
\coin{h_{\scriptscriptstyle{(7)}}}{{c}_{\scriptscriptstyle{(6)}}} 
\\ &&\!\!\!\!\!\!\!\!\!\!=
a \mtco \zero{c} \ot  h_{\scriptscriptstyle{(5)}} \three{c} 
\co{\one{c}}{S(\two{c})S(\four{h}) h_{\scriptscriptstyle{(6)}}{c}_{\scriptscriptstyle{(4)}}}
\coin{S(\three{h})}{h_{\scriptscriptstyle{(7)}}\five{c}}
 \co{S(\two{h})}{h_{\scriptscriptstyle{(8)}}{c}_{\scriptscriptstyle{(6)}}}
\coin{S(\one{h})h_{\scriptscriptstyle{(9)}}}{{c}_{\scriptscriptstyle{(7)}}} 
\end{eqnarray*}
where in the last passage we have used the cocycle condition \eqref{iv}.
Finally by simplifying  the convolution product term $\coin{S(\three{h})}{h_{\scriptscriptstyle{(7)}}\five{c}}
 \co{S(\two{h})}{h_{\scriptscriptstyle{(8)}}{c}_{\scriptscriptstyle{(6)}}}
$
we obtain
\be\label{lhs}
(a \ot h) \trl_{\pg \ot^\cot  \underline{H}_\cot} c=
a \mtco \zero{c} \ot  h_{\scriptscriptstyle{(3)}} \three{c} 
\co{\one{c}}{S(\two{c})S(\two{h}) h_{\scriptscriptstyle{(4)}}{c}_{\scriptscriptstyle{(4)}}}
\coin{S(\one{h})h_{\scriptscriptstyle{(5)}}}{{c}_{\scriptscriptstyle{(5)}}} ~.
\ee
The invertible map $\varphi_{A,\underline{H}}$ is a right
$\pg$-modules isomorphism if this expression coincides with 
$\varphi_{A,\underline{H}}^{-1}( (\varphi_{A,\underline{H}}(a \ot
h))\trl_{(A \ot \underline{H})_\cot } c)$.
Recalling the right $\pg$-module structure of $(A\otimes
\underline{H})_\cot$, explicitly  given in
\eqref{3rightac},
 we have
\begin{eqnarray*}
&&\!\!\!\!\!\!\!\!\!\!\!\varphi_{A,\underline{H}}^{-1}\left( (\varphi_{A,\underline{H}}(a \ot h))\trl_{(A \ot \underline{H})_\cot } c\right) \,=\,
\coin{\two{a}}{S(\one{h})\five{h}} \varphi_{A,\underline{H}}^{-1}\left(  \zero{a}\zero{c} \ot \three{h}\one{c}
\right) \coin{\one{a}S(\two{h})\four{h}}{\two{c}}
\\ &&\qquad \qquad \!\!\!= 
\coin{\three{a}}{S(\one{h}){h}_{\scriptscriptstyle{(7)}}} 
\co{\one{a}\one{c}}{S(\two{c})S(\three{h})\five{h}\four{c}}
 \zero{a}\zero{c} \ot \four{h}\three{c}
 \coin{\two{a}S(\two{h}){h}_{\scriptscriptstyle{(6)}}}{\five{c}}~.
\end{eqnarray*}
By using \eqref{abc} to expand the term $\co{\one{a}\one{c}}{S(\two{c})S(\three{h})\five{h}\four{c}}$, the above expression becomes
\begin{eqnarray*}
&& \!\!\!\!\!\!\!\!\!\!\zero{a}\zero{c} \coin{\one{a}}{\one{c}}\ot \five{h} \four{c} 
u_\cot(\two{c})
\coin{S(\three{c})}{S(\four{h}){h}_{\scriptscriptstyle{(6)}}\five{c}} 
\co{\two{a}}{S(\three{h}) h_{\scriptscriptstyle{(7)}} c_{\scriptscriptstyle{(6)}}}
\coin{\three{a}S(\two{h}){h}_{\scriptscriptstyle{(8)}}}{c_{\scriptscriptstyle{(7)}}} 
\\ && \qquad
\qquad \qquad \qquad \qquad \qquad \qquad \qquad \qquad
      \qquad\qquad\qquad \qquad \qquad \qquad \qquad \!\!\!\!  \coin{\four{a}}{S(\one{h}){h}_{\scriptscriptstyle{(9)}}} 
\\[.2em] && \quad \!\!\!\!\!\!= \zero{a}\zero{c} \coin{\one{a}}{\one{c}}\ot \five{h}
      \four{c} 
u_\cot(\two{c})
\coin{S(\three{c})}{S(\four{h}){h}_{\scriptscriptstyle{(6)}}\five{c}} 
\coin{S(\three{h}) h_{\scriptscriptstyle{(7)}}}{ c_{\scriptscriptstyle{(6)}}}
\co{\two{a}}{S(\two{h}){h}_{\scriptscriptstyle{(8)}}}
\\ && \qquad
\qquad \qquad \qquad \qquad \qquad \qquad \qquad \qquad \qquad
      \qquad\qquad \qquad \qquad \qquad \qquad \!\!\!\! \coin{\three{a}}{S(\one{h}){h}_{\scriptscriptstyle{(9)}}}
\\[.2em] && \quad\!\!\!\!\!\! =  \zero{a}\zero{c} \coin{\one{a}}{\one{c}}\ot \three{h} \four{c} 
u_\cot(\two{c})
\coin{S(\three{c})}{S(\two{h}){h}_{\scriptscriptstyle{(4)}}\five{c}} 
\coin{S(\one{h}) h_{\scriptscriptstyle{(5)}}}{ c_{\scriptscriptstyle{(6)}}}
\\ && \quad \!\!\!\!\!\! =  {a} \mtco \zero{c}\ot \three{h} \three{c} 
u_\cot(\one{c})
\coin{S(\two{c})}{S(\two{h}){h}_{\scriptscriptstyle{(4)}}\four{c}} 
\coin{S(\one{h}) h_{\scriptscriptstyle{(5)}}}{ c_{\scriptscriptstyle{(5)}}}
\end{eqnarray*}
where in the second step we have used Lemma \ref{lem:formula} \eqref{iv}. 
Finally, by using (\ref{simpl0}) the expression of $\varphi_{A,\underline{H}}^{-1}( (\varphi_{A,\underline{H}}(a \ot h))\trl_{(A \ot \underline{H})_\cot } c)$ coincides with \eqref{lhs} above.
  \end{proof}

\subsection{Deformation via a 2-cocycle based on an external symmetry $K$}\label{sec:def_es}

In this section we first define the notion of external symmetry (with
Hopf algebra $K$) of an $H$-Hopf-Galois extension, and study the
corresponding category. Then we deform this extension with a 2-cocycle on $K$.

Consider a Lie group $L$ acting via diffeomorphisms on both the
total manifold and the base manifold of a bundle $P\to M$, these actions being
compatible with the bundle projection (hence $L$ acts via automorphisms of
$P\to M$).  We say that $L$ is an external
symmetry of $P\to M$. Considering algebras rather than manifolds (cf.\
Example \ref{ex:principalbundle}), we term a Hopf algebra $K$ an external symmetry of
the extension $B\subset A$, if $A$ is a (left) $K$-comodule algebra with $B$
a $K$-subcomodule algebra. 
If we consider principal $G$-bundles $P\to M$ then
we also require $G$-equivariance of the $L$-action on the total
manifold leading to algebras $A$ that are $(K,H)$-bicomodules
algebras, whose category is denoted  ${}^K{\cal A}^H$ and
defined in \S \ref{secKHbicomod} before Proposition \ref{prop:leftrightdef}.

We are thus led to term a Hopf algebra $K$ an external symmetry of 
an $H$-Hopf Galois extension $B=A^{coH}\subseteq A$, 
if  $A\in {}^K{\cal A}^H$ and  if  $B=A^{coH}$ is a $K$-subcomodule.
It immediately follows that $B=A^{coH}$ is a $(K,H)$-subbicomodule algebra.

The requirement that $B=A^{coH}$ is a
$K$-subcomodule of $A$
holds automatically true in particular if $K$ is a flat module. We recall that $K$ is a flat
$\bbK$-module  if any short exact sequence of
$\bbK$-modules $0\to U\stackrel{i}{\to} V\stackrel{j}{\to} W\to 0$
implies the short exact sequence of $\bbK$-modules $0\to K\otimes U\stackrel{\id_K\ot i}{\longrightarrow}K\ot V\stackrel{\id_K\ot j}{\longrightarrow}  K\ot W\to 0$. 
 In particular all modules are flat if $\bbK$ is a
  field or the ring of formal power series with coefficients in a
field.

\begin{prop}\label{flat} Let $H$ and $K$ be Hopf algebras, 
let $K$  be  flat as $\bbK$-module, and  let $A\in {}^K{\cal A}^H$;  then $B=A^{coH}$ is a $K$-subcomodule algebra.
\end{prop}
\begin{proof}
The short exact sequence
\(
0\longrightarrow A^{coH}\stackrel{i}{-\!\!\!-\!\!\!-\!\!\!-\!\!\!-\!\!\!-\!\!\!\longrightarrow} A\stackrel{\delta^A-\id_A\ot \eta_H}{-\!\!\!-\!\!\!-\!\!\!-\!\!\!-\!\!\!-\!\!\!\longrightarrow}
\mathrm{Im}{(\delta^A-\id_A\ot \eta_H)}\longrightarrow 0
\), where $\id_A\ot \eta_H: A\ot \bbK\simeq A\to A\otimes H\,,~ a\mapsto a\otimes 1_H$,
defines the algebra of $H$-coinvariants $A^{coH}$. If $K$ is flat we have the associated short exact sequence 
\[
0\longrightarrow K\ot
A^{coH}\stackrel{\id_K\otimes\, i}{-\!\!\!-\!\!\!-\!\!\!-\!\!\!-\!\!\!-\!\!\!\longrightarrow}
K\ot A\stackrel{\id_K\otimes(\delta^A-\id_A\ot \eta_H)}{-\!\!\!-\!\!\!-\!\!\!-\!\!\!-\!\!\!-\!\!\!-\!\!\!-\!\!\!-\!\!\!-\!\!\!-\!\!\!\longrightarrow}
K\ot \mathrm{Im}(\delta^A-\id_A\ot \eta_H)\longrightarrow 0~.
\]
Now the compatibility between the $H$- and $K$-coactions $\delta^A$
and $\rho^A$ (cf.\ (\ref{compatib})) implies that, for all $b\in A^{coH}$,
we have $(\id_K\ot
\delta^A)[\rho^A(b)]=(\rho^A\ot\id_H)[\delta^A(b)]=\rho^A(b)\ot
\1_H=(\id_K\ot\id_A\ot\eta_H)[\rho^A(b)]$ and therefore
$\rho^A(A^{coH})\in \ker[\id_K\ot (
\delta^A-\id_A\ot\eta_H)]=K\ot A^{coH}$,
where the last equality is due to the exact sequence.
This proves that $B=A^{coH}$ is a $K$-subcomodule of $A$, and hence a
$K$-subcomodule algebra.
\end{proof}

Consider now an object $A$ in ${}^K{\cal A}^H$, with right
$H$-coaction denoted by $\delta^A : A\to A\otimes H\,,~a\mapsto \zero{a}\otimes \one{a}$
and left $K$-coaction by $\rho^A: A\to K\otimes A\,,~a\mapsto \mone{a}\otimes\mzero{a}$. 
We trivially have $\underline{H}\in {}^K{\cal A}^H$ with the
$K$-coaction $\rho^{\underline{H}}: \underline{H}\to K\ot \underline{H}\,, ~h\mapsto 1_K\ot h$.
Since the category of $(K,H)$-bicomodules ${}^K{\cal M}^H$ is a
monoidal category and $A,\underline{H}$ are in particular objects in
${}^K{\cal M}^H $, then $A\otimes A$ and $A\otimes\underline{H}$ are
objects  in ${}^K{\cal M}^H $. Moreover $A\otimes A$ and $A\otimes\underline{H}\in {}^K{}_{A}{\cal M}_{A}{}^H$
since the left and the right $A$-actions are $K$-comodule morphisms,
indeed we easily prove commutativity of the diagrams:
\be\label{KAHAdiag}
\xymatrix{
\ar[d]_-{\tra_{A\otimes \underline{H}}} A\ot A\otimes \underline{H}
\ar[rr]^-{\rho^{A\ot A \ot
      \underline{H}}} &&\ar[d]^-{\id\,\ot\,\tra_{A\ot \underline{H}}
    } K\ot A \ot A\otimes \underline{H}\\
A\ot \underline{H} \ar[rr]^-{\rho^{A\ot \underline{H}}}&& K\ot A\otimes \underline{H}
}~~~~~~~~
\xymatrix{
\ar[d]_-{\trl_{A\otimes \underline{H}}} A\otimes \underline{H}\ot A \ar[rr]^-{\rho^{A \ot
      \underline{H}\ot A}} &&\ar[d]^-{\id\,\ot\,\trl_{A\ot \underline{H}}
    } K\ot A \otimes \underline{H}\otimes A\\
A\ot \underline{H} \ar[rr]^-{\rho^{A\ot \underline{H}}}&& K\ot A\otimes \underline{H}
}~
\ee
and of the corresponding ones for $A\ot A$ (the proof that $A\ot
\underline{H}$ and $A\ot A\in  {}^K{}_{A}{\cal M}$ and that $A\ot
A\in  {}^K{}{\cal M}_{A}$ can be also seen to follow from the property that ${}^K{}_{A}{\cal M}$ and ${}^K{\cal M}_{A}$ are respectively right and left module categories over the monoidal category  ${}^K{\cal M}$).

Furthermore, since $B$ is a $K$-subcomodule then it is
easy to see that the $K$-comodule structure of $A\otimes A$ is induced
on the quotient $A\ot_BA$, that is therefore an object in the relative
Hopf module category ${}^K{}_{A}{\cal M}_{A}{}^H$.
We have thus proven the following

\begin{prop}
Let $H$ and $K$ be Hopf algebras, $A\in {}^K{\cal A}^H$ and
$B=A^{coH}$ be a $K$-subcomodule. Then $A\ot_B A$ and $A\ot
\underline{H}$ are objects in  ${}^K{}_{A}{\cal M}_{A}{}^H$.
\end{prop}
Explicitly the $K$-coactions on $A\ot_BA$ and on $A\ot \underline{H}$ read
\begin{flalign}
\rho^{A\otimes_B A} : A\otimes_B A \longrightarrow K\otimes(A\otimes_B A)~,~~
a\otimes_B c \longmapsto \mone{a}\mone{c}\otimes(\mzero{a}\otimes_B \mzero{c})~
\end{flalign}
and
\begin{flalign}
\rho^{A\otimes \underline{H}} : A\otimes\underline{H}\longrightarrow K\otimes A\otimes\underline{H}~,~~
a\otimes h\longmapsto \mone{a}\otimes \mzero{a}\otimes h~.
\end{flalign}
The canonical map preserves this additional structure:
\begin{prop}\label{canmapKmorph}
If $A\in {}^K{\cal A}^H$ and $B=A^{coH}$ is a $K$-subcomodule,
then the canonical map $\chi: A\otimes_B A\to A\otimes\underline{H}$ is a morphism
in ${}^{K}{}_{A}{\cal M}_{A}{}^{H}$.
\end{prop}
\begin{proof}
Recalling from Proposition \ref{prop_canMorph} that the canonical map $\chi$
is a morphisms in  ${}_A{\cal M}_A{}^H$, we just have to show that it preserves the left $K$-coactions,
i.e.\  $\rho^{A \ot \underline{H}}\circ \chi= (\id \ot \chi)\circ
\rho^{A \ot_B A}$. This indeed holds true:
\begin{flalign*}
\rho^{A \ot \underline{H}}\big(\chi(a \ot_B c)\big)&= \mone{(a\zero{c})} \ot \mzero{(a\zero{c})}  \ot \one{c}\\
&=\mone{a}\mone{(\zero{c})} \ot \mzero{a}\mzero{(\zero{c})}  \ot \one{c}\\
&= \mone{a}\mone{c}\ot \zero{a}\zero{c}\ot \one{c}\\
&= \mone{a} \mone{c} \ot \chi\left( \mzero{a}\ot_B \mzero{c}\right) \\
&= (\id \ot \chi)\big(\rho^{A \ot_B A}(a \ot_B c)\big)~,
\end{flalign*}
where in the third and fourth equality we have used the equivariance condition \eqref{hyp:cov}. 
\end{proof}

Let us now consider a $2$-cocycle $\sigma:K\otimes K\to\bbK$  on $K$. 
We deform according to Proposition \ref{prop:co} the Hopf algebra $K$
into the Hopf algebra $K_\sigma$. Using the machinery
of \S \ref{secleftKcomod} and \S \ref{secKHbicomod} we can also deform the $(K,H)$-bicomodule algebra $A$ into the $(K_\sigma,H)$-bicomodule algebra ${}_\sigma A\in {}^{K_\sigma}{\cal A}^{H}$
(choose in Proposition \ref{prop:leftrightdef} the trivial 2-cocycle
$\cot(h\otimes h') = \varepsilon(h)\,\varepsilon(h')$ on $H$). If
$B=A^{coH}$ is a $K$-comodule then it is a $(K,H)$-bicomodule algebra and is
as well deformed into the $(K_\sigma,H)$-bicomodule algebra ${}_\sigma
B:={}_\sigma(A^{coH})=({}_\sigma A)^{coH} \in{}^{K_\sigma}{\cal A}^H$. 
 As a consequence we have the twisted canonical
map ${}_\sigma \chi : {}_\sigma A \, {{}^{\sigma\!\!\:}\ot_{{}_\sigma B}}\, {}_\sigma A \to {}_\sigma A \,{{}^{\sigma}\otimes} \, \underline{H}\,$,
which by Proposition \ref{canmapKmorph} is a morphism in ${}^{K_\sigma}{}_{{}_\sigma A}{\cal M}_{{}_\sigma A}{}^{H}$.
\sk

The following theorem relates the twisted canonical map ${}_\sigma\chi$ with the original canonical map $\chi$.
\begin{thm}\label{Th:diagr-can2} Let $A\in {}^K{\cal A}^H$ and $B=A^{coH}$ a $K$-subcomodule.
Given a $2$-cocycle $\sigma : K\otimes K\to \bbK$ the diagram
\begin{equation}\label{cd}
\xymatrix{
 \pgls \, {{}^{\sigma\!\!\:}\ot_{{}_\sigma B}}\, \pgls  \ar[rr]^-{{}_\sigma\chi} \ar[d]_{\col_{A,A}}  && \pgls \,{{}^{\sigma\!\!\:}\ot}\, \underline{H} \ar[d]^{\col_{A,\underline{H}} } 
\\
{}_\sg(A \ot_B A) \ar[rr]^{\Sigma(\chi)} && {}_\sg(A \ot \underline{H})
 }
\end{equation}
in ${}^{K_\sigma}{}_{{}_\sigma A}{\cal M}_{{}_\sigma A}{}^{H}$
commutes.
\end{thm}
\begin{proof}
First we notice that the left vertical arrow is the induction to
the quotient of the isomorphism $\col_{A,A} : \pgls \, {{}^{\sigma\!\!\:}\ot}\,
\pgls\longrightarrow {}_\sg(A \ot A)$ defined in (\ref{nt-left}); it is well defined 
 thanks to the cocycle condition (\ref{lcocycle}) for $\sigma$.
Next let us observe that $\col_{A,\underline{H}}$ is  the identity; indeed, 
since $\underline{H}$ is equipped with the trivial left $K$-coaction 
$h \mapsto \1_K \ot h$ and $\sg$ is unital, we have
$$
\col_{A,\underline{H}} (a \,{{}^{\sigma\!\!\:}\ot}\, h)= \sig{\mone{a}}{\mone{h}} \mzero{a} \ot \mzero{h}=
\sig{\mone{a}}{\1_K} \mzero{a} \ot {h} = a \ot h~,
$$
for all $a\in{}_\sigma A$ and $h\in \underline{H}$.
Furthermore, it  is clear by Proposition \ref{canmapKmorph} and 
\S \ref{secKHbicomod} that all arrows are morphisms in ${}^{K_\sigma}{}_{{}_\sigma A}{\cal M}_{{}_\sigma A}{}^{H}$,
so it remains to prove the commutativity of the diagram:
\begin{flalign*}
\chi \big(\col_{A,A} (a \, {{}^{\sigma\!\!\:}\ot_{{}_\sigma B}}\, a') \big) &= \sig{\mone{a}}{\mone{a'}} ~ \chi(\mzero{a} \ot_B \mzero{a'}) \\
&=\sig{\mone{a}}{\mone{a'}} ~ \mzero{a}  \zero{(\mzero{a'})} \ot
\one{(\mzero{a'})}
\\
&= \sig{\mone{a}}{\mone{(\zero{a'})}} ~ \mzero{a}  \mzero{(\zero{a'})} \ot
\one{{a'}}\\
&=
a \, {\mtcols}\, \zero{a'}  \ot
\one{{a'}}\\
&= {}_\sigma \chi(a \, {{}^{\sigma\!\!\:}\ot_{{}_\sigma B}}\, a')~,
\end{flalign*}
for all $a,a^\prime\in {}_\sigma A$.
\end{proof}

\begin{cor}\label{cor:diagr-can2}
$B \subseteq A$ is an $H$-Hopf-Galois extension if and only if $\bgls
\subseteq \pgls$ is an $H$-Hopf-Galois extension.
\end{cor}
\begin{proof}
The statement follows from the invertibility of the morphisms
$\varphi_{A,\underline{H}}$ and $\varphi_{A,A}$
in diagram (\ref{cd}).
\end{proof}

In order to prove that twist deformations of
principal $H$-comodule algebras are principal $H$-comodule
algebras we need the following
\begin{prop}\label{LQ}
Let $B\in  {}^K{\cal A}^H$,  $V,W\in {^K}_B {\cal M}^H$ and 
$_BHom^H(V,W)$ be the $\bbK$-module of $\bbK$-linear maps $V\to W$ that
are left $B$-module maps and right $H$-comodule maps.
Let $\sigma : K\otimes K\to \bbK$ be a $2$-cocycle on $K$, then
there is a $\bbK$-module  isomorphism
\begin{eqnarray}
\fS : {}_BHom^H(V,W)&\longrightarrow& {}_{_\sigma B} Hom^H({}_\sigma V, {}_\sigma W)\\
s&\longmapsto &\fS(s)\nn
\end{eqnarray}
defined by, for all $v\in {}_\sigma V$,
\begin{flalign}\label{fS}
\fS(s)(v)=\sigma\big(\mtwo{v}\ot
S(\mone{v})\,\mone{s(\zero{v})}\big)\,\zero{s(\zero{v})}~,
\end{flalign}
with inverse 
$
\fS^{-1} : {}_{_\sigma B}Hom^H({}_\sigma
V,{}_\sigma W) \longrightarrow
{}_BHom^H(V,W)$, $ \tilde s\mapsto \fS^{-1}( \tilde s)\,$ given by, for all $v\in V$,
\begin{flalign}\label{invfS}
\fS^{-1}(\tilde s)(v)=\sigma\big(S(\mtwo{v})\ot
\mone{\tilde s(\zero{v})}\big)\,\bar{u}_\sigma(\mone{v})\zero{\tilde
  s(\zero{v})}~.
\end{flalign}
On $K$-comodule maps we have $\fS=\Sigma$. 

\end{prop}
\begin{proof}
Use  of property (\ref{simpl}) for the 2-cocycle $\sigma$ shows that 
an expression equivalent to  (\ref{fS}) is 
$\fS(s)(v)=u_\sigma(\mtwo{v})\bar\sigma({S(\mone{v})}\ot
\mone{s(\zero{v})})\,\zero{s(\zero{v})}$. Then it is easy to directly
check that (\ref{invfS}) defines the inverse of $\fS$.
The $H$-comodule property of $\fS(s)$ is a straightforward
consequence of the $H$-comodule property of $s$ and of the
compatibility between the $H$- and $K$-coactions. 
We now show that $\fS(s)$ is a left ${}_\sigma B$-linear map:
for all $b\in {}_\sigma B$, $v\in {}_\sigma V$,
\begin{flalign*}
\fS(s)(b  {\mtcols}\, v)
  &=\sigma(\mone{b}\otimes\mone{v})\,\fS(s)(\zero{b}\zero{v})\\
&=\sigma(\mfour{b}\otimes\mthree{v})\,\sigma\big(\mthree{b}\mtwo{v}\otimes S(\mtwo{b}\mone{v})\mone{b}\mone{s(\zero{v})}\big)\,\zero{b}\zero{s(\zero{v})}\\
&=\sigma(\mtwo{b}\otimes\mthree{v})\,\sigma\big(\mone{b}\mtwo{v}\otimes 
S(\mone{v})\,\mone{s(\zero{v})}\big)\,\zero{b}\zero{s(\zero{v})}\\
&=\sigma(\mthree{b}\otimes\mfour{v})\,\bar{\sigma}(\mtwo{b}\otimes\mthree{v})\,\sigma\big(\mtwo{v}\otimes 
S(\mone{v})\,\mtwo{s(\zero{v})}\big)\,\sigma(\mone{b}\otimes
\mone{s(\zero{v})})\,\zero{b}\zero{s(\zero{v})}\\
&=\sigma(\mone{b}\otimes
\mone{s(\zero{v})})\,\sigma\big(\mtwo{v}\otimes 
S(\mone{v})\,\mtwo{s(\zero{v})}\big)\,\zero{b}\zero{s(\zero{v})}\\
&=b  {\mtcols}\, \big(\sigma\big(\mtwo{v}\otimes 
S(\mone{v})\,\mone{s(\zero{v})}\big)\,
\zero{s(\zero{v})}\Big)\\
&=b  {\mtcols}\,\fS(s)(v)
\end{flalign*}
where in the second equality we used left $B$-linearity of $s$, and in the
fourth Lemma \ref{lemma2co}. Finally  if $s$ is a
$K$-comodule map then we immediately see that $\fS(s)=\Sigma(s)$.
\end{proof}

\begin{rem}
Consider a Hopf algebra ${\cal V}$ dually paired to $K$.  
If $V,W$ are left $K$-comodules, then they are right
${\cal V}$-modules and left ${\cal V}^{op}$-modules; let  $\triangleright_V$
 and $\triangleright_W$ be the corresponding ${\cal V}^{op}$-actions,
cf.\ Appendix \ref{dualconst}  (recall that ${\cal V}^{op}$ is the Hopf algebra with opposite
product, inverse antipode and same coproduct and counit as ${\cal V}$). The
set of $\bbK$-linear maps  $Hom_{\bbK}(V,W)$ is
canonically a left ${\cal V}^{op}$-module with the adjoint action
${\cal V}^{op}\otimes Hom_{\bbK}(V,W)\to Hom_{\bbK}(V,W)$, $(\nu,s)\mapsto
{\nu}\blacktriangleright^{cop}s:=\two{\nu}\triangleright_W\circ\, s\circ
S^{{\cal V}^{op}}(\one{\nu})\,\triangleright_V\,$. 

A twist
${\cal G}$ on ${\cal V}$ defines a 2-cocycle
$\sigma$ on $K$ (cf.\ (\ref{tw2coc}) in Appendix  \ref{dualconst})  and a twist ${\cal
G}^{op}={\cal G}^{-1}$ on ${\cal V}^{op}$. 
In this case  ${}_\sigma V, {}_\sigma W$ are right ${\cal V}^{\cal G}$-modules and henceforth left ${\cal V}^{{\cal
    G}^{\:op}}=({\cal V}^{op})^{{\cal
    G}^{op}}\!$-modules and so is $Hom_\bbK({}_\sigma V, {}_\sigma W)$ with
the ${\cal V}^{{\cal G}^{\:op}}\!$-adjoint action; also
${}_\sigma(Hom_\bbK( V,  W))$ is a left ${\cal V}^{{\cal
    G}^{\:op}}\!$-module. As proven in \cite{AS} Theorem 4.8, these last two are isomorphic
left ${{\cal V}^{{\cal G}^{\,op}}}$-modules via $D_{{\cal  G}^{op}}^{cop}: {}_\sigma(Hom_\bbK( V,
W))\to Hom_\bbK({}_\sigma V, {}_\sigma W)$. 
This map is the same as that in Proposition \ref{LQ} (with $B$ and $H$ trivial).
Explicitly $$D_{{\cal  G}^{op}}^{cop}(s) =\two{\tg_\alpha} \triangleright_W\circ\, s\circ
\tg^\alpha  S(\one{\tg_\alpha})\, \triangleright_V =\fS(s)\,,$$ where 
${\cal G}=\tg^\alpha\otimes \tg_\alpha$, and  the product, coproduct and
antipode are those of ${\cal V}$. We refer to \cite{AS} for further
properties of this left deformation map.

The categorical viewpoint is also instructive. We first define
the functor $hom: {({}^{\,{\cal V}^{op}}{\!\!\cal M})}^{op}\times
{}^{{\cal V}^{op}}{\!\!\cal M}\to {}^{{\cal V}^{op}}{\!\!\cal M}$, that on objects is
given by $hom(V,W)=Hom_\bbK(V,
W)$ (the ${\cal V}^{op}$-module of $\bbK$-linear maps $V\to W$),  while on
morphisms $V\stackrel{\;f}{\leftarrow}X$,
$W\stackrel{g\,}{\rightarrow} Y$  it is given by 
$hom(f,g)(L)=g\circ L\circ f\in hom(X,Y)$ for all
$\bbK$-linear maps $L\in hom(V,W)$. The functor $hom$ is an internal
hom functor in the monoidal category $({}^{{\,\cal V}^{op}}{\!\!\cal M},
\ot)$ because, for all $V$, the functor $hom(V, -)$ is right
adjoint to the tensor product functor $V\ot -\,$; a proof is in \cite{BSS}
\S 2.5,  where the quasi-Hopf algebra $H$ there is the Hopf algebra
${{\cal V}^{op}}^{\, cop}$ (i.e.,
${{\cal V}^{op}}$ with opposite coproduct and antipode). Thus $hom$
defines a closed monoidal category structure on
 $( {}^{{\,\cal V}^{op}}{\!\!\cal M}, \ot)$. Similarly we have the closed monoidal category $({}^{{\,\cal V}^{{\cal G}^{\,op}}}{\!\!\!\cal M}, {}^{\sigma\:\!\!}\ot, hom_\sigma)$.
Then, in this framework, the left ${}{{\cal V}^{{\cal G}^{\,op}}}$-module isomorphisms  $\fS: {}_\sigma(Hom_\bbK( V,
W))\to Hom_\bbK({}_\sigma V, {}_\sigma W)$, that could be denoted
$\fS_{_{V,W}}$, are the components of a natural isomorphism $\Sigma\circ
hom\Rightarrow hom_\sigma\circ (\Sigma^{op}\times \Sigma)$ between 
functors from the category ${({}^{{\,\cal V}^{op}}{\!\!\cal M})}^{op}\times
{}^{{\cal V}^{op}}{\!\!\cal M}$ to the category ${}^{{\cal V}^{{\cal G}^{\,op}}}{\!\!\!\cal M}$. 
Via this natural isomorphism ${}^{{\cal V}^{op}}{\!\!\cal M}$ and ${}^{{\cal V}^{{\cal G}^{\,op}}}{\!\!\!\cal M}$ are equivalent as closed monoidal categories.
\end{rem}

\begin{cor}\label{cor2:diagr-can2}
$A$ is a principal $H$-comodule algebra if and only if ${}_\sigma A$ is a principal $H$-comodule algebra.
\end{cor}
\begin{proof}
The proof is similar to that of Corollary \ref{cor-pcomodalg} with the caveat that since
$s$ is not a $K$-comodule map we have to consider its deformation via
the map $\fS$. Hence we consider the commutative diagram
\begin{flalign*}
\xymatrix{
\ar[drr]_-{\fS(s)}{}_\sigma A \ar[rr]^-{{}_\sigma s}
&&\ar[d]_-{\varphi^\ell_{B,A}}
{}_\sigma B \,{{}^{\sigma\!\!\:}\ot} \: {}_\sigma A \ar[rr]^-{{}_\sigma m} &&
{}_\sigma A\\
&& {}_\sigma (B\otimes A) \ar[rru]_-{\Sigma(m)}&& 
}
\end{flalign*}
where by definition ${}_\sigma s:=\fS(s)\circ ({\varphi^\ell_{B,A}})^{-1}$.
Left ${}_\sigma B$-linearity and right $H$-colinearity of $\fS(s)$ and of $\varphi^\ell_{B,A}$ imply
 that also  ${}_\sigma s$ is a map in ${}_{{}_\sigma B}{\cal
  M}^H$. Furthermore it  is a section of the (restricted) product
${}_\sigma m: {}_\sigma B {}^{\sigma\!\!\:}\ot {}_\sigma A\to {}_\sigma A$ since, for all $a\in {}_\sigma A$
\begin{flalign*}
({}_\sigma m\circ {}_\sigma s)(a)&=(\Sigma(m)\circ \fS(s)) (a)\\
&=\sigma\big(\mtwo{a}\ot S(\mone{a})\,\mone{s(\zero{a})}\big)\,m(\zero{s(\zero{a})})\\
&=\sigma\big(\mtwo{a}\ot S(\mone{a})\,\mone{m(s(\zero{a}))}\big)\,\zero{m({s(\zero{a})})}\\
&=\sigma\big(\mtwo{a}\ot S(\mone{a})\,\mone{\zero{a}}\big)\,\zero{\zero{a}}\\&=a
\end{flalign*}
where in the third equality we used that $m$ is a $K$-comodule map,
and in the fourth that $m\circ s=\id_A$.
\end{proof}

\begin{rem}\label{cleft?} 
Concerning cleftness of the extension ${}_\sigma B\simeq {}_\sigma
A^{coH}\subset {}_\sigma A$, if $\theta: B\ot H\to A$ is an
isomorphisms of left $B$-modules, left $K$-comodules and right
$H$-comodules then, as in Corollary \ref{cor-can}, $\Sigma(\theta):
{}_\sigma B\ot H\to {}_\sigma A$ is an
isomorphisms of left ${}_\sigma B$-modules, left ${}_\sigma K$-comodules and right
$H$-comodules, with inverse $\Sigma(\theta^{-1})$. In general however
$\theta: B\ot H\to A$ is not a left $K$-comodule map, then we can consider
$\fS(\theta): {}_\sigma B\ot H\to {}_\sigma A$ that is a map in
${}_{{}_\sigma B}{\cal M}^H$; if this map is invertible then cleftness
of $A^{coH}\subset A$ implies cleftness of ${}_\sigma
A^{coH}\subset {}_\sigma A$. In the context of formal deformation
quantization this is always the case, and considering also $\fS^{-1}$ we obtain 
that $A^{coH}\subset  A$ is cleft if and only if ${}_\sigma
A^{coH}\subset {}_\sigma A$ is cleft.
\end{rem}
\sk

\begin{ex}[The instanton bundle on the isospectral sphere $S^4_\theta$ \cite{gw, gs}.]\label{exCL}
The classical $SU(2)$-Hopf bundle $\pi:  S^7 \ra S^4$ over the four-sphere $S^4$ can be described in different ways.  We take here a pure algebraic approach well suited for the application of the deformation  theory developed above. 

Let $\mathcal{O}(\IR^8)$ be the commutative  $*$-algebra over
$\bbK=\mathbb{C}$ generated by elements
$\{z_i,~z_i^*, ~i=1,\dots ,4\}$. Let 
$A:= \mathcal{O}(S^7)$ be the  algebra of coordinate functions
 on the $7$-sphere $S^7$ obtained as the quotient of
 $\mathcal{O}(\IR^8)$ by the two-sided $*$-ideal generated by the
 element $\sum z^*_i z_i-1$.
Let $H:=\mathcal{O}(SU(2))$ be the 
Hopf algebra of coordinate functions on $SU(2)$ realized as the $*$-algebra generated by commuting elements 
$\{w_i,~w_i^*, ~i=1,2\}$ with $\sum w^*_i w_i=1$ and standard Hopf algebra structure induced from the group structure of $SU(2)$.

The classical principal action of $SU(2)$ on $S^7$ can be described at the algebraic level by the data of the following
right coaction of $\mathcal{O}(SU(2))$ on  
$\mathcal{O}(S^7)$: 
\begin{eqnarray}\label{princ-coactSU2}
\delta^{\mathcal{O}(S^7)}:\quad  \mathcal{O}(S^7) \quad  &\longrightarrow & 
\mathcal{O}(S^7) \ot \mathcal{O}(SU(2))
\\ \nn
\, u
 &\longmapsto&
u
\overset{.}{\otimes}
\begin{pmatrix}
 w_1 & -w_2^*
\vspace{2pt}
\\
w_2 & w_1^*\end{pmatrix} \quad , \qquad 
u:=
\begin{pmatrix}
z_1& z_2 & z_3& z_4
\vspace{2pt}
\\
-z_2^*  &
  z_1^* &
  -z_4^*
& z_3^*
\end{pmatrix}^t 
\end{eqnarray}
where $\overset{.}{\otimes}$  stands for the composition of $\ot$ with the matrix multiplication. The 
 map $\delta^{\mathcal{O}(S^7)}$ defined  above on the algebra generators, and extended to the whole $\mathcal{O}(S^7)$  as a $*$-algebra morphism, structures $\mathcal{O}(S^7)$ as a right $\mathcal{O}(SU(2))$-comodule algebra.  As expected, the 
 subalgebra $ B:=\mathcal{O}(S^7)^{co(\mathcal{O}(SU(2)))}\subset\mathcal{O}(S^7)$ of coinvariants 
under the coaction $\delta^{\mathcal{O}(S^7)}$ 
can be identified with the algebra of coordinate  functions  on the 
 4-sphere $S^4$. Indeed  the entries of the matrix
\be\label{proj-inst}
\textsf{P}:=uu^*= \frac{1}{2}
\begin{pmatrix}
1+x & 0 & \alpha & - \beta^*
\\
0 & 1+x & \beta & \alpha*
\\
\alpha^* & \beta^* & 1-x & 0
\\
-\beta & \alpha & 0 & 1-x 
\end{pmatrix},
\ee
 where 
\be\label{4sphere-coinv}
\alpha:= 2(z_1 z_3^* + z^*_2 z_4)~, \quad 
\beta:= 2(z_2 z_3^* - z^*_1 z_4)~, \quad 
x:= z_1 z_1^* + z_2 z_2^* - z_3 z_3^* -z_4 z_4^* ~, \quad 
\ee 
(and their $*$-conjugated $\alpha^*, \beta^*$,  with $x^*=x$),  form a set of generators for  $B$ and from the $7$-sphere 
relation $\sum z_i^* z_i=1$ 
it follows that they satisfy 
$$
\alpha^* \alpha + \beta^* \beta + x^2=1 . 
$$
Thus the subalgebra $B$ of coinvariants is isomorphic to the algebra 
$\mathcal{O}(S^4)$ of coordinate  functions  on  $S^4$. The algebra 
inclusion $\mathcal{O}(S^4) \hookrightarrow \mathcal{O}(S^7)$ dualizes
the Hopf map $\pi:S^7 \ra S^4$.
The algebra  $\mathcal{O}(S^7)$ is a non cleft Hopf-Galois extension
of $\mathcal{O}(S^4)$. Moreover, since $\mathcal{O}(SU(2))$ is
cosemisimple and has a bijective antipode, then $\mathcal{O}(S^7)$  is a 
principal comodule algebra (recall the last paragraph
of \S \ref{sec:HG}).\\
 
We now apply the theory developed above and deform this extension of commutative algebras  by using a symmetry of the classical Hopf bundle. 
Let $K:=\mathcal{O} (\mathbb{T}^2) $ be the  commutative $*$-Hopf algebra of functions 
on the 2-torus  $\mathbb{T}^2 $ with generators $t_j,~t_j^*=t_j^{-1}$, $j=1,2$  and co-structures $\Delta(t_i)=t_i \ot t_i$, $\varepsilon(t_i)=1$, $S(t_i)=t_i^{-1}=t_i^*$. Let $\sigma$ be
the exponential 2-cocycle on $K$ 
which is determined by its value on the generators:
\be\label{cocycleT2}
\sig{t_j}{t_k}= \exp(i \pi \,\Theta_{jk}) ~~,\quad \Theta= \frac{1}{2}\begin{pmatrix} 0 & \theta 
\\
- \theta & 0 \end{pmatrix}  ~~,\quad \theta \in \mathbb{R}~
\ee
and  extended to the whole algebra by requiring 
$\sig{ab}{c}=\sig{a}{\one{c}}\sig{b}{\two{c}}$ and 
 $\sig{a}{bc}=\sig{\one{a}}{c}\sig{\two{a}}{b}$, for all $a,b,c \in \mathcal{O} (\mathbb{T}^n)$.
There is a left coaction of $\mathcal{O} (\mathbb{T}^2) $ on the  algebra $\mathcal{O}(S^7)$: it is
 given on the generators as 
\be\label{coazioneT-S7}
\da^{\mathcal{O}(S^7)}:  \mathcal{O}(S^7) \longrightarrow \mathcal{O} (\mathbb{T}^2)  \ot \mathcal{O}(S^7) ~,
\quad z_i \longmapsto \tau_i \ot z_i ~,
\ee
where $(\tau_i):=(t_1,t_1^*,t_2,t^*_2)$, 
 and it is extended to the whole of $\mathcal{O}(S^7)$ as a $*$-algebra homomorphism. 
It is easy to prove that the two coactions 
$\delta^{\mathcal{O}(S^7)}$ and $\da^{\mathcal{O}(S^7)}$ satisfy the compatibility condition \eqref{compatib},
hence they structure $\mathcal{O}(S^7)$ as a
$(\mathcal{O}(\mathbb{T}^2),\mathcal{O}(SU(2)))$-bicomodule algebra;
furthermore $\mathcal{O}(S^4)$ is a
$(\mathcal{O}(\mathbb{T}^2),\mathcal{O}(SU(2)))$-subbicomodule algebra
as can be easily checked on its
generators, or indirectly inferred from Proposition \ref{flat} (since
vector spaces are flat). Explicitly the $\mathcal{O} (\mathbb{T}^2) $-coaction reads 
\be\label{coazioneT-S4}
\alpha \longmapsto t_1 t_2^* \ot \alpha~,  \quad 
\beta \longmapsto t_1^* t_2^* \ot \beta ~, \quad 
x \longmapsto 1 \ot x ~.
\ee
We can therefore apply the theory of deformation  by 2-cocycles to both the comodule algebras $\mathcal{O}(S^7)$ and $\mathcal{O}(S^4)$ (recall \S \ref{secKHbicomod} and the discussion above Theorem \ref{Th:diagr-can2}). The resulting noncommutative algebras, denoted respectively by $\mathcal{O}(S^7_\theta)$ and $\mathcal{O}(S^4_\theta)$, are two representatives of the class of  $\theta$-spheres in \cite{cl}.
In particular, the classical Hopf fibration $\mathcal{O}(S^4) \hookrightarrow \mathcal{O}(S^7)$ described above deforms to  a quantum Hopf bundle  on $\mathcal{O}(S^4_\theta)\simeq \mathcal{O}(S^7_\theta)^{coH}$ with undeformed  structure 
Hopf algebra $H=\mathcal{O}(SU(2))$. Indeed, from Corollary \ref{cor2:diagr-can2}, we further obtain
\begin{prop}\label{propLS}
The algebra $\mathcal{O}(S^7_\theta)$ is a  principal $\mathcal{O}(SU(2))$-comodule algebra. 
\end{prop}

The noncommutative bundle so obtained is the quantum Hopf bundle on the  Connes-Landi
 sphere $\mathcal{O}(S^4_\theta)$  that was originally constructed in \cite{gw}, and further studied  
 in the context of 2-cocycles deformation in \cite{gs}.  The principality of the algebra inclusion 
$\mathcal{O}(S^4_\theta) \subseteq
\mathcal{O}(S^7_\theta)$
was first proven in \cite[\S 5]{gw} by explicit
construction of the inverse of the canonical map. Proposition
\ref{propLS} follows instead  as a straightforward result of  the
general theory developed in the present section (out of the principality of the underlying classical bundle).
\end{ex}

\subsection{\label{sec:combidef}Combination of deformations}
We now consider the combination of the previous two constructions. 
This leads to Hopf-Galois extensions in which the structure Hopf
algebra,  total space and base space are all deformed. 
\sk

As before, we let $H$ and $K$ be Hopf algebras and $A\in{}^K{\cal A}^H$ a $(K,H)$-bicomodule
algebra, with $B=A^{coH}$ a $K$-subcomodule.
Let $\sigma: K\otimes K\to \bbK$ and $\cot : H\otimes H\to \bbK$ be $2$-cocycles
and denote by $K_\sigma$ and $H_\cot$ the twisted Hopf algebras and by
${}_\sigma A_\cot := {}_\sigma(A_\cot) = ({}_\sigma A)_\cot\in{}^{K_\sigma}{\cal A}^{H_\cot}$ the deformed
$(K_\sigma,H_\cot)$-bicomodule algebra, see \S \ref{secKHbicomod}. We also have the deformed
$(K_\sigma,H_\cot)$-bicomodule algebra ${}_\sigma B := {}_\sigma B_\cot\subseteq
{}_\sigma A_\cot$ of $H_\cot$-coinvariants in ${}_\sigma A_\cot$.
The canonical map ${}_\sigma \chi_{\cot} : {}_\sigma A_\cot \, {{{}^{\sigma\!\!\:}\otimes^\cot}_{\!\!\!{}_\sigma B}}\, {}_\sigma A_\cot \to
{}_\sigma A_\cot \,{^{\sigma}\otimes^\cot}\, \underline{H_\cot}$
 is a ${}^{K_\sigma}{}_{{}_\sigma A_\cot}{\cal M}_{{}_\sigma   
   A_\cot}{}^{H_\cot}$-morphism because of  Proposition \ref{canmapKmorph}.
There are two equivalent ways to relate it to the canonical map $\chi: A \otimes_ B A
 \to A\otimes \underline{H}$. We can apply the functor $\Sigma$ to the
 commutative
 diagram (\ref{diagr-can-pre}) of Theorem \ref{theo:diagr-can-pre} and
 then top the resulting diagram with the analogue of the commutative diagram
 (\ref{cd}) of Theorem \ref{Th:diagr-can2}, or we can first apply the
 functor $\Gamma$ to (\ref{cd}) and then top it with the analogue of
 (\ref{diagr-can-pre}). 
 
\begin{thm}\label{Th:diagr-can3}
Given two $2$-cocycles $\sigma: K\otimes K\to \bbK$ and $\cot: H\otimes H\to \bbK$
the diagrams 
\begin{equation}
\xymatrix{
{}_\sg{(\pg)} \, {{{}^{\sigma\!\!\:}\otimes^\cot}_{\!\!\!{}_\sigma B}}\, {}_\sg (\pg) 
\ar[d]_{\col_{A_\cot,A_\cot}} 
\ar[rr]^-{{}_\sigma \chi_{\cot}} && {}_\sg(\pg) \,{{}^{\sigma\!\!\:}\ot^\cot}\, \underline{\hg} 
\ar[d]^{\col_{A_\cot,\underline{H_\cot}}}
\\
\ar[dd]_-{\Sigma(\varphi_{A,A})} 
 {}_\sigma(\pg \otimes_{B}^\cot \pg) \ar[rr]^-{\Sigma(\can_\cot)} &&  
 {}_\sigma(\pg\ot^\cot \underline{\hg} ) 
 \ar[d]^-{\Sigma(\id \ot^\cot \Q)}\\
 && {}_\sigma(\pg \ot^\cot  \underline{H}_\cot) 
 \ar[d]^-{\Sigma(\varphi_{A,\underline{H}})}\\
{}_\sigma((A\otimes_B A)_\cot)  \ar[rr]^-{\Sigma(\Gamma(\can))}  && {}_\sigma((A\otimes \underline{H})_\cot)
} 
\qquad
\xymatrix{
\ar[dd]_-{\varphi_{\pgls,\pgls}} 
 {(\pgls)}_\cot \, {{{}^{\sigma\!\!\:}\otimes^\cot}_{\!\!\!{}_\sigma B}}\, (\pgls)_\cot 
 \ar[rr]^-{{}_\sigma\can_{\cot}} && 
  (\pgls)_\cot\,{{}^{\sigma\!\!\:} \ot^\cot}\, \underline{\hg}
  \ar[d]^-{\id \,{^{\sigma\!\!\:}\ot^\cot}\, \Q}\\
 && 
 (\pgls)_\cot \,{^{\sigma\!\!\:} \ot^\cot}\,  \underline{H}_\cot
 \ar[d]^-{\varphi_{A,\underline{H}}}
\\
(\pgls \, {{{}^{\sigma\!\!\:}\otimes}_{{}_\sigma B}}\, \pgls)_\cot 
\ar[rr]^-{\Gamma({}_\sigma\can)} \ar[d]_{\Gamma(\col_{A,A})}  && 
(\pgls\,{{}^{\sigma\!\!\:}\otimes}\, \underline{H})_\cot 
\ar[d]^{\Gamma(\col_{A,\underline{H}})}
\\
(_\sigma (A \ot_B A))_\cot  \ar[rr]^{\Gamma(\Sigma(\chi))} && (_\sigma (A \ot \underline{H}) )_\cot
}
\end{equation}
in ${}^{K_\sigma}{}_{{}_\sigma A_\cot}{\cal M}_{{}_\sigma
  A_\cot}{}^{H_\cot}$ commute and have the same external square
diagram. Moreover:\\
(i)  $B \subseteq A$ is an $H$-Hopf-Galois extension if and only if ${}_\sigma B \subseteq {}_\sg \pg$ is an $\hg$-Hopf-Galois
 extension.  \\
(ii) $A$ is a principal $H$-comodule algebra if and only if ${}_\sigma A_\cot$ is a principal $H_\cot$-comodule 
 algebra.
\end{thm}
\begin{proof}
Commutativity follows from commutativity of the internal diagrams,
statements (i) and (ii) also immediately follow combining the analogue
statements for each of the 2-cocycles $\gamma$ and $\sigma$. The
equality of the external square diagrams follows from diagram
(\ref{commdiagSG}) of Proposition \ref{monfuneq} applied to the
left vertical arrows, and from
$\col_{A_\cot,\underline{H_\cot}}=\Gamma({\varphi^\ell_{A,\underline{H}}})=\id$
as well as the triviality of the functor $\Sigma$ on morphisms. 
\end{proof}

\begin{ex}[Formal deformation quantization]\label{FDQ}
Let $G$ be a Lie group, $M$ a manifold and 
$\pi : P\to M$ a principal $G$-bundle over $M$ with right $G$-action denoted by $r_P : P\times G\to P$.
Then, by Example \ref{ex:principalbundle},
we have a Fr{\'e}chet $H=C^\infty(G)$-Hopf-Galois extension
$B=C^\infty(M)\simeq A^{coH}\subseteq A=C^\infty(P)$ with $\bbK=\mathbb{C}$.
 Let us further assume that there exists another Lie group $\LL$ acting
 on the $G$-principal bundle $P\to M$, i.e.\ that there are left $\LL$-actions $l_P : \LL \times P\to P$ and $l_M : \LL \times M\to M$,
such that the diagrams
\begin{flalign}\label{eqn:diagleftaction}
\xymatrix{
 \ar[d]_-{\id\times \pi}\LL\times P \ar[rr]^-{l_P} && P\ar[d]^-{\pi}&&\ar[d]_-{\id\times r_P^{}}\LL \times P \times G \ar[rr]^-{l_P \times \id} && P\times G\ar[d]^-{r_P^{}}\\
\LL \times M \ar[rr]^-{l_M}&& M && \LL \times P \ar[rr]^-{l_P} && P 
}
\end{flalign}
commute. For example $\LL$ may be a finite-dimensional 
Lie subgroup of the automorphism group of the bundle, which comes
with a  canonical left action on $P$ and $M$. 
Due to the left $\LL$-actions on $P$ and $M$ we obtain a Fr{\'e}chet left $K = C^\infty(\LL)$-comodule
structure on $A$ and $B$, which is compatible with the right $H$-coaction on $A$ and 
the canonical map because of the diagrams in (\ref{eqn:diagleftaction}), i.e.\ $A =C^\infty(P)$
is a Fr{\'e}chet $(K=C^\infty(\LL),H=C^\infty(G))$-bicomodule algebra.

In order to deform this example into a noncommutative Hopf-Galois
extension, in the  context of formal power series in a deformation 
parameter $\hbar$, we consider the formal power series extension
of the $\mathbb{C}$-modules $H$, $A$, $B$ and $K$, denoted as usual
$H[[\hbar]]$, $A[[\hbar]]$, $B[[\hbar]]$ and $K[[\hbar]]$. The natural
topology on these $\mathbb{C}[[\hbar]]$-modules is a combination of the original Fr{\'e}chet topology in each order of $\hbar$
together with the $\hbar$-adic topology, see e.g.\ \cite[Chapter XVI]{Kassel}.
The canonical map induces a continuous 
$\mathbb{C}[[\hbar]]$-linear isomorphism (denoted with abuse of notation by the same symbol)
\begin{flalign}
\chi : A[[\hbar]]\,\widehat{\otimes}_{B[[\hbar]]}\, A[[\hbar]] \simeq C^\infty(P\times_M P)[[\hbar]] 
\longrightarrow A[[\hbar]] \,\widehat{\otimes}\,\underline{H}[[\hbar]]\simeq C^\infty(P\times G)[[\hbar]]~,
\end{flalign}
where now $\widehat{\otimes}$ denotes the completion of the algebraic tensor product with respect to the 
natural topologies described above. Hence we have obtained a topological
$H[[\hbar]]$-Hopf-Galois extension $B[[\hbar]] \simeq A[[\hbar]]^{coH[[\hbar]]}\subseteq A[[\hbar]]$.

Notice that for $G$ a Lie group we have a (in general degenerate) pairing between
the universal enveloping algebra $U(\mathfrak{g})$ of its Lie algebra
$\mathfrak{g}$ and $C^\infty(G)$; it is determined by evaluating at
the unit element $e\in G$ left invariant vector fields on functions.
Explicitly, $\langle \,\cdot\,,
\,\cdot\,\rangle : U(\mathfrak{g})\times C^\infty(G) \to \mathbb{C}$
is defined by extending 
\begin{flalign}
\langle 1,h\rangle := h(e)~~,\quad \langle v  , h  \rangle :=
\frac{d}{dt} h\big(r_P^{} (e,\exp(-t v))\big)\big\vert_{t=0}~,
\end{flalign}
for all $h\in C^\infty(G)$ and $v\in \mathfrak{g}$, to all
$U(\mathfrak{g})$ via linearity and requiring $\langle \xi \xi'  , h  \rangle =\langle \xi  ,
\one{h}  \rangle \/\langle \xi'  , \two{h}  \rangle$ for all $\xi,\xi'\in U(\mathfrak{g})$.
The $\bbK$-linear maps $\langle v , \,\cdot\,\rangle : C^\infty(G) \to \bbK$
are continuous and since the coproduct $\Delta: C^\infty(G) =H \to H\,\widehat{\ot}\,
H\simeq C^\infty(G\times G)$ is continuous also the
$\bbK$-linear maps $\langle \xi , \,\cdot\,\rangle : C^\infty(G) \to \bbK$
are continuous for all $\xi\in U(\mathfrak{g})$.
(These maps are actually the Lie derivative along $\xi $,  $L_{\xi}:
C^\infty(G)\to C^\infty(G)$, composed with the  counit in
$C^\infty(G)$; where $L_v$ is the Lie derivative along the left
invariant vector field defined by $v\in \mathfrak{g}$, and $L$ is extended to all
$U(\mathfrak{g})$ by $L_{\xi\xi'}=L_\xi\circ L_{\xi'}$).
Because of this pairing we can assign to a twist $\F =
\f^\alpha\,\widehat{\otimes}\, \f_\alpha \in 
U(\mathfrak{g})[[\hbar]]\,\widehat{\ot}\,U(\mathfrak{g})[[\hbar]]\simeq
(U(\mathfrak{g}) \otimes U(\mathfrak{g}))[[\hbar]]$ a continuous $2$-cocycle
$\gamma :  H[[\hbar]]\,\widehat{\otimes}\,H[[\hbar]] \to\bbK[[\hbar]]$
by defining $\gamma(h\otimes k) = \langle\f^\alpha,h\rangle \, \langle \f_\alpha, k\rangle$
on the dense subset $H[[\hbar]]\otimes H[[\hbar]] \subseteq H[[\hbar]]\,\widehat{\otimes}\,H[[\hbar]] $
and extending it by continuity. (See Appendix \ref{app:twists} for more on the duality between twists 
and $2$-cocycles). Similarly we may consider a twist ${\cal G} \in
U(\mathfrak{\fL})[[\hbar]]\,\widehat{\ot}\, U(\mathfrak{\fL})[[\hbar]]  \simeq
(U(\mathfrak{\fL})\otimes U(\mathfrak{\fL}))[[\hbar]] $, where $\fL$
is the Lie algebra of $L$,
and define a continuous $2$-cocycle $\sigma  : K[[\hbar]]\,\widehat{\otimes}\,K[[\hbar]] \to\bbK[[\hbar]]$.

We now twist the $\mathbb{C}[[h]]$-modules $H[[\hbar]]$, $A[[\hbar]]$, $B[[\hbar]]$ and
$K[[\hbar]]$  as described in general in Section \ref{sec:twists}, and
obtain a noncommutative topological $H[[\hbar]]_\gamma$-Hopf-Galois extension
${}_\sigma B[[\hbar]] \simeq {}_{\sigma}A[[\hbar]]_\gamma^{coH[[\hbar]]_\gamma} \subseteq {}_{\sigma}A[[\hbar]]_\gamma$.
(The canonical map ${}_\sigma\chi_{\gamma}$ is a continuous isomorphism,
 since $\chi$ and all vertical arrows in the diagrams in Theorem
 \ref{Th:diagr-can3} are continuous isomorphisms). Recalling Remark \ref{cleft?},
the Hopf-Galois extension $A[[\hbar]]^{coH[[\hbar]]} \subseteq
A[[\hbar]]$ is cleft if and only if ${}_{\sigma}A[[\hbar]]_\gamma^{coH[[\hbar]]_\gamma} \subseteq
{}_{\sigma}A[[\hbar]]_\gamma$ is cleft. 
\end{ex}

\section{Applications}\label{appsect4}
We apply the theory so far developed first to the study of
deformations of quantum homogeneous spaces in \S \ref{sec:qhom},
including the explicit example of the even $\theta$-spheres $S^{2n}_\theta$ in \S
\ref{S2n}, and then to the study of deformations
of sheaves of Hopf Galois extensions in \S \ref{sec:sheaves},
providing the example of the Hopf bundle over $S^4_\theta$ as a twisted sheaf in \S
\ref{S4astwistedsheaf}.

\subsection{Twisting quantum homogeneous spaces associated with quantum subgroups}\label{sec:qhom}

The theory of twists, in particular the combination of  deformations developed in \S  \ref{sec:combidef}, can be used to  study 
 deformations of bundles over quantum homogeneous
spaces arising from Hopf algebra projections. This is the subject of the present subsection.

Recall that given a Hopf algebra $G$, a quantum subgroup of $G$ is a  Hopf algebra $H$ together with  
a surjective bialgebra (and thus Hopf algebra) homomorphism $\pi: G \ra H$. The restriction via 
$\pi$ of the coproduct
of $G$
\be\label{rrc}
\delta^G : = (\id \ot \pi)\circ \Delta: G \longrightarrow G \ot H~
\ee
induces on $G$ the  structure of a right $H$-comodule algebra. The subalgebra 
$B:=G^{coH}\subseteq G$ of coinvariants is called a  {\it quantum homogeneous $G$-space.}  
When the associated canonical map  
 \begin{flalign}
 \chi: G \ot_B G \longrightarrow G \ot H, \quad g \ot_B g' \longmapsto g\one{g'} \ot \pi(\two{g'})
 \end{flalign}
is bijective, i.e.\ $B\subseteq G$ is a Hopf Galois-extension, we call
$G$ a {\it quantum principal bundle
over the quantum homogeneous space} $B$.  (See e.g.\ \cite[\S 11.6.2]{KS}, \cite[\S 5.1]{BM93}). 
\\

Given a quantum principal bundle $B=G^{coH}\subseteq G$ 
over a quantum homogeneous space $B$  and a $2$-cocycle $\gamma$ on
$H$ we can consider two different constructions:
\begin{itemize}\item[-]
On the one hand we can
lift the $2$-cocycle $\gamma$ on $H$  to a 2-cocycle $\tilde{\gamma}$ on $G$ (see Lemma \ref{lem:coc-by-pb} below) and thus apply the theory of 2-cocycle
deformations for Hopf algebras (\S \ref{sec:twists-hopf})
to deform both $G$ and $H$ into new 
Hopf algebras $G_{\widetilde{\gamma}}$ and $H_\gamma$. It turns out that the
condition for $H$ to be a quantum subgroup of $G$ is preserved under deformation, 
i.e.\  $H_\gamma$ is a quantum subgroup of $G_{\widetilde{\gamma}}$, and thus there is an 
 associated   twisted quantum homogeneous space
 $B_{\widetilde{\gamma}}$. 
\item[-]
On the other hand, we can direct the attention to  the algebra
inclusion $B=G^{coH}\subseteq G$ as a Hopf-Galois extension, and
twist it. In this case,  we forget the Hopf algebra structure of $G$
and use $\gamma$ to deform $G$ just as  an object in $\A^{H}$, as in
\S\ref{sec:def_sg}. Denote by $G_{\gamma}$ the resulting  comodule
algebra.
\end{itemize}
These two deformations $G_{\widetilde{\gamma}}$ and $G_{\gamma}$ of $G$ do not coincide. In particular,  $G_{\gamma}$ is not in general a Hopf algebra and thus the base space of the twisted bundle is no longer a quantum homogeneous space of the total space. 
Nevertheless the second construction can be reconciled with the first
one by applying a further twist deformation and thus considering 
a combination of deformations as in \S \ref{sec:combidef}. As a
corollary of this second approach we obtain that $B_{\widetilde{\gamma}}\subseteq
  G_{\widetilde{\gamma}}$ is a Hopf-Galois extension.
Indeed we show below that,  given a $2$-cocycle $\gamma$ on $H$,
quantum principal bundles $B=G^{coH}\subseteq G$ 
over quantum homogeneous spaces $B$  deform into new quantum principal bundles over  new quantum homogeneous spaces. 
\bigskip

We proceed by first showing that given a  $2$-cocycle $\gamma$ on $H$
we can twist both the Hopf algebras $H$ and $G$ is such a way to still
have  a quantum homogeneous space. 
\begin{lem}\label{lem:coc-by-pb}
Let $\gamma: H \ot H \ra \bbK$ be a $2$-cocycle on $H$. Then
\begin{flalign}
\widetilde{\gamma}: G \ot G \longrightarrow \bbK~,~~ g\otimes g'\longmapsto \co{\pi(g)}{\pi(g')}
\end{flalign}
is a $2$-cocycle on $G$.
\end{lem}
\begin{proof}
The proof relies on  the fact that $\gamma$ is a 
 2-cocycle and  $\pi$ is a bialgebra homomorphism, i.e.\ in particular
$\Delta_H \circ \pi= (\pi \ot \pi)\circ \Delta_G$ and $\varepsilon_H \circ\pi = \varepsilon_G.$  We have
$$
\cog{1}{g}= \co{1}{\pi(g)}= \varepsilon_H(\pi(g))= \varepsilon_G(g)~,
$$
and similarly $\cog{g}{1} = \varepsilon_G(g)$, for all $g\in G$. For the cocycle property \eqref{lcocycle} we compute 
\begin{eqnarray*}
\cog{\one{g}}{\one{h}} \cog{\two{g}\two{h}}{k}&= & 
\co{\pi(\one{g})}{\pi(\one{h})} \co{\pi(\two{g})\pi(\two{h})}{\pi(k)} \\
&=&  
\co{\one{\pi(g)}}{\one{\pi(h)}} \co{\two{\pi(g)}\two{\pi(h)}}{\pi(k)}
\\
&=&
\co{\one{\pi(h)}}{\one{\pi(k)}} \co{\pi(g)}{\two{\pi(h)}\two{\pi(k)}}~,
\end{eqnarray*}
for all $g,h,k \in G$,
and proceeding in a similar way one proves that 
$
\cog{\one{h}}{\one{k}} \cog{g}{\two{h}\two{k}}
$ has the same expression.
The convolution inverse of $\widetilde{\gamma}$ is 
$\coing{g}{h}= \coin{\pi(g)}{\pi(h)}$ as easily proven by using again the fact that $\pi$ intertwines the coproducts. 
\end{proof}

We can deform the algebra product and antipode in the Hopf algebra $G$, 
and $H$, by using the 2-cocycles $\widetilde{\gamma}$ and $\gamma$ respectively.
By Proposition \ref{prop:co} we obtain two new Hopf algebras which we denote by $G_{\widetilde{\gamma}}$ and $H_\gamma$. 
Their algebra products are given respectively by
\be\label{Gtilde}
g\cdot_{\tilde{\gamma}} g'= \co{\pi(\one{g})}{\pi( \one{g'})} \two{g}\two{g'} \coin{\pi(\three{g})}{\pi(\three{g'})} ~,
\ee
for all $g,g' \in G_{\widetilde{\gamma}}$,
and
\be
h\mt h'= \co{\one{h}}{\one{h'}} \two{h}\two{h'} \coin{\three{h}}{\three{h'}} ~,
\ee
for all $h,h' \in H_{{\gamma}}$.
The map $\pi : G\to H$ remains a Hopf algebra homomorphism with respect to the deformed Hopf algebra structures on $G_{\widetilde{\gamma}}$
and $H_\cot$:
\begin{lem}\label{lem:qhs}
The map 
\begin{flalign}
\pi_\gamma: G_{\tilde{\gamma}} \longrightarrow H_\gamma ~,~~ g\longmapsto \pi (g)
\end{flalign}
is a surjective bialgebra homomorphism.
\end{lem}
\begin{proof}
Since the coproducts and counits of $G_{\widetilde{\gamma}}$ and $H_\gamma$ are not deformed by 
the twisting procedure  it is clear that $\pi_\gamma$ is still a coalgebra map.
We can easily check that $\pi_\gamma$  preserves also  the deformed algebra product:
\begin{eqnarray*}
\pi_\gamma (g \cdot_{\tilde{\gamma}} g') &=& 
\co{\pi(\one{g})}{\pi(\one{g'})} 
\pi(\two{g})\pi(\two{g'}) \coin{\pi(\three{g})}{\pi(\three{g'})}
\\
&=& 
\co{\one{\pi(g)}}{\one{\pi(g')}} 
\pi(\two{g})\pi(\two{g'}) \coin{\three{\pi(g)}}{\three{\pi(g')}}
= \pi_\gamma(g) \mt \pi_\gamma(g')~,
\end{eqnarray*}
for all $g,g'\in G_{\widetilde{\cot}}$.
\end{proof}
It follows that the twisting procedure deforms the quantum homogeneous space $B=G^{coH}$
into another quantum homogeneous space $B_{\widetilde{\gamma}}=G_{\widetilde{\gamma}}^{coH_\gamma}$,
 which  is isomorphic to $B$ only as a $\bbK$-module but not in general as an algebra.
\sk

On the other hand, given a $2$-cocycle $\gamma$ on $H$, we can deform $H$ into the Hopf algebra $H_\gamma$ as above, but consider $G$ simply as a right $H$-comodule algebra with coaction given in \eqref{rrc}
and twist its algebra product accordingly to \eqref{rmod-twist}. In this way we get an $H_\cot$-comodule {algebra}, 
$G_\gamma$, with product
\be\label{Gg}
g\mtco g'= \one{g} \one{g'} \coin{\pi(\two{g})}{\pi(\two{g'})} ~,
\ee
for all $g,g' \in G_\gamma$.
By Corollary \ref{cor-can}, the extension
$B= G_\gamma^{co H_\gamma} \subseteq G_\gamma$  is an
$H_\cot$-Hopf-Galois extension if and only if the original extension $B \subseteq G$ was $H$-Hopf-Galois.
However, as already remarked above, this twisted bundle has a total space which
 is just an algebra and the condition for $H_\gamma$   to be a quantum subgroup is lost, and so that of $B$ to be a quantum homogeneous space.
To resolve this problem let us consider $K=H$ as an external Hopf algebra of symmetries 
coacting from the left on $G$. The Hopf algebra $G$ is also a left  $H$-comodule algebra via
\be\label{lrc}
\rho^G : = (\pi \ot \id)\circ \Delta: G \longrightarrow H \ot G ~,~~ g \longmapsto \pi(\one{g}) \ot \two{g} ~.
\ee
Clearly, the left and right $H$-coactions $\rho^G$ and $\delta^G$ satisfy the compatibility condition \eqref{compatib},
hence they structure $G$ as an $(H,H)$-bicomodule. Assume $B$ is a subcomodule for the left $H$-coaction. We can therefore 
twist the product in $G$ accordingly to Proposition \ref{prop:leftrightdef} (i) (with the special choice
$\sigma = \cot : H\otimes H\to \bbK$) in order to get an $(H_\cot,H_\cot)$-bicomodule algebra ${}_\cot G_\cot$
with product
\be\label{gG}
g \,{{}_\gamma \bullet_\gamma}\, g'= \co{\pi(\one{g})}{\pi( \one{g'})} ~\two{g}\two{g'} ~\coin{\pi(\three{g})}{\pi(\three{g'})} ~,
\ee 
for all $g,g' \in {}_\gamma G_\gamma$.
Theorem \ref{Th:diagr-can3} then implies that ${}_\cot B := {}_\cot G_\gamma^{co H_\gamma}\subseteq
{}_\cot G_\cot$ is an $H_\cot$-Hopf-Galois extension if and only if 
$B = G^{coH}\subseteq G$ is a $H$-Hopf-Galois extension.

\begin{prop}
The algebra $_\gamma G_\gamma $ is isomorphic to the algebra  underlying the Hopf algebra
$G_{\widetilde{\gamma}}$ and hence inherits from it a Hopf algebra structure.  
The subalgebra of coinvariants $_\gamma B$ is isomorphic to the quantum homogeneous space $B_{\widetilde{\gamma}}$.
\end{prop}
\begin{proof}
By comparing (\ref{gG}) with \eqref{Gtilde} we have that the algebras  $_\gamma G_\gamma$ and $G_{\widetilde{\gamma}}$ 
are isomorphic via the identity map.  For $b,b' \in B_{\widetilde{\gamma}}$
we have 
$$
b\cdot_{\widetilde{\gamma}} b'= \co{\pi(\one{b})}{\pi( \one{b'})} \two{b}\two{b'} \coin{\pi(\three{b})}{\pi(\three{b'})} =\co{\pi(\one{b})}{\pi( \one{b'})} \two{b}\two{b'} ~,
$$
because $B=B_{\widetilde{\gamma}}$ as $\bbK$-modules and hence $b,b'$ are right $H$-coinvariant.
Hence the result  ${}_\gamma B \simeq B_{\widetilde{\gamma}}$.
\end{proof}

As a direct consequence of Theorem \ref{Th:diagr-can3} we then obtain
that quantum principal bundles over quantum homogeneous spaces deform
into quantum principal bundles over quantum homogeneous spaces:

\begin{cor}\label{cor:qhs}
The extension $B_{\tilde{\gamma}} \subseteq G_{\tilde{\gamma}}$ of the 
quantum homogeneous space $B_{\tilde{\gamma}}$ is $H_\cot$-Hopf-Galois if and only 
if the extension $B\subseteq G$ of the quantum homogeneous space $B$ is $H$-Hopf-Galois.
\end{cor}

\subsubsection{The quantum homogeneous spaces $S^{2n}_\theta$ and their associated quantum principal bundles}\label{S2n} 
The $\theta$-spheres $S^{2n}_\theta$ were introduced in \cite{cl} as
noncommutative manifolds with the property that the Hochschild dimension equals the commutative dimension. They were shown to be
 homogeneous spaces of twisted deformations of
$SO(2n+1,\mathbb{R})$ in \cite{var}. Their geometry was further
studied in \cite{cdv}, see also \cite{ab}.
We here revisit their explicit construction and as a corollary of the
previous section conclude that the
Hopf algebra of noncommutative coordinate functions
$\mathcal{O}(SO_\theta(2n+1,\mathbb{R}))$ is a quantum principal bundle over
the quantum homogeneous space $\mathcal{O}(S^{2n}_\theta)$ of
noncommutative coordinate functions on the sphere.  We then immediately conclude that the
Hopf-Galois extension $\mathcal{O}(S^{2n}_\theta)\subset
\mathcal{O}(SO_\theta(2n+1,\mathbb{R}))$ is a principal 
comodule algebra. 
\\

We begin by introducing the algebra of coordinate functions on $SO(2n,
\IR)$, on $SO(2n+1, \IR)$ and on their quotient $S^{2n}$.
Let $\mathcal{O}(M(2n, \IR))$, $n \in \mathbb{N}$  be the  commutative
$*$-algebra over $\mathbb{C}$ with generators 
$a_{ij}, b_{ij}, a_{ij}^* = *(a_{ij}), b_{ij}^*=*(b_{ij})$, $i,j=1, \dots n$.     It is a bialgebra with coproduct and counit  given in matrix notation as
\begin{equation}
\Delta(M)=M \overset{.}{\otimes} M \quad , \quad \varepsilon(M)=\mathbbm{1} , \mbox{ for  }\quad M=(M_{IJ}):=\begin{pmatrix}
(a_{ij}) & (  b_{ij})
\\
(b_{ij}^*) & (a_{ij}^*)
\end{pmatrix} \quad,
\end{equation}
where $\overset{.}{\otimes}$ denotes the combination of tensor product and matrix multiplication, $\mathbbm{1}$ is the identity matrix and capital indices $I,J$ run from $1$ to $2n$.
The Hopf algebra of
coordinate functions on $SO(2n, \IR)$ is the quotient 
$\mathcal{O}(SO(2n, \IR))= \mathcal{O}(M(2n, \IR)){/{I_Q}} $
where $I_Q$ is the bialgebra ideal defined by   
\begin{equation}\label{idealQ}
I_Q= \langle\, M^t Q M -Q \; ; \; M Q M^t-Q\; ; \;
\det(M)-1
\,\rangle \; 
, \quad Q:=\begin{pmatrix}
0 &\mathbbm{1}_n 
\\
 \mathbbm{1}_n & 0
\end{pmatrix} =Q^t=Q^{-1}~.
\end{equation}
In matrix notation the $*$-structure in $\mathcal{O}(M(2n, \IR))$ is given by
$*(M)=QMQ$ so that  $I_Q$ is easily seen to be a $*$-ideal.  The  $*$-bialgebra $\mathcal{O}(SO(2n, \IR))$ is a $*$-Hopf algebra with antipode 
$
S(M):= Q M^tQ^{-1}$. Notice that in  $\mathcal{O}(SO(2n, \IR))$ we have $M^\dag M= \mathbbm{1} =M M^\dag$, where ${}^\dag$ indicates the composition of matrix transposition ${}^t$  and $*$-conjugation.\\

Similarly,  for the odd case let $\mathcal{O}(M(2n+1, \IR))$, $n \in \mathbb{N}$, be the  commutative $*$-bialgebra with generators 
$a_{ij}, b_{ij},  a_{ij}^*=*(a_{ij}), b_{ij}^*=*(b_{ij}), u_{i}, v_{i}, u_i^*=*(u_{i}), v_i^*=*(v_{i})$, $i,j=1, \dots n$, and $x=*(x)$. The coproduct and counit are given as 
\begin{equation}
\Delta(N)=N \overset{.}{\otimes} N \quad , \quad \varepsilon(N)=\mathbbm{1} , \mbox{ where  }
N:=\begin{pmatrix}
(a_{ij}) & (  b_{ij})  & (u_i)
\\
(b_{ij}^*) & (a_{ij}^*) & (u_i^*)
\\
(v_i)&  (v_i^*) & x 
\end{pmatrix}.
\end{equation}
The algebra of coordinate functions on $SO(2n+1, \IR)$  is the quotient $\mathcal{O}(SO(2n+1, \IR))= \mathcal{O}(M(2n+1,
\IR)){/{J_Q}} $  where $J_Q$ is the bialgebra $*$-ideal
\begin{equation}
J_Q= \langle N^t Q N -Q \; ; \; N Q N^t-Q \; ; \; 
\det(N)-1\rangle \; , \quad Q:=\begin{pmatrix}
0 & \mathbbm{1}_n &0
\\
\mathbbm{1}_n & 0 &0
\\
0 & 0 & 1
\end{pmatrix} .
\end{equation}
The $*$-structure can be written in terms of $Q$ as $*(N)=QNQ^{-1}$.
 The  $*$-bialgebra $\mathcal{O}(SO(2n+1, \IR))$ is a $*$-Hopf algebra with
 antipode $S(N)= Q N^tQ=N^\dag$. 
 
The (commutative) Hopf algebra $\mathcal{O}(SO(2n, \IR))$ is a quantum subgroup of $\mathcal{O}(SO(2n+1, \IR))$ with surjective Hopf algebra morphism
\be\label{pi}
\pi:\mathcal{O}(SO(2n+1, \IR)) \longrightarrow \mathcal{O}(SO(2n, \IR)) \; , \quad 
\begin{pmatrix}
(a_{ij}) & (  b_{ij})  & (u_i)
\\
(b_{ij}^*) & (a_{ij}^*) & (u_i^*)
\\
(v_i)&  (v_i^*) & x 
\end{pmatrix}
\longmapsto
\begin{pmatrix}
(a_{ij}) & (  b_{ij}) & 0
\\
(b_{ij}^*) & (a_{ij}^*) &0 
\\
0 &0 & 1
\end{pmatrix} .
\ee
Hence there is a natural right coaction of $\mathcal{O}(SO(2n, \IR))$ on
$\mathcal{O}(SO(2n+1, \IR))$, given by (cf.\  \eqref{rrc})
\begin{eqnarray*}
\delta:= (\id \ot \pi)\Delta: \mathcal{O}(SO(2n+1, \IR)) &\longrightarrow& \mathcal{O}(SO(2n+1, \IR)) \ot \mathcal{O}(SO(2n, \IR))~,
\\
N &\longmapsto& N \overset{.}{\otimes} \pi(N)~.
\end{eqnarray*}
The subalgebra $B\subset  \mathcal{O}(SO(2n+1, \IR))$ of coinvariants is generated by  the elements in the last column of the defining matrix $N$: $u_i, u_i^*$ and $x$. 
It  is isomorphic to the algebra of coordinate functions $\mathcal{O}(S^{2n})$
on the even sphere $S^{2n}\subset \IR^{2n+1}$, indeed from $N^\dag N=\mathbbm{1}$ we have that the generators of $B$ (rescaling the $u_i$'s by $1/\sqrt{2}$) satisfy the sphere equation
$\sum_{i=1}^{n} u_i^* u_i +x^2=1.$  

Finally, in this affine variety setting we can identify 
$\mathcal{O}(
SO(2n+1, \IR)\times SO(2n, \IR))$ with $\mathcal{O}(
SO(2n+1, \IR))\otimes \mathcal{O}(SO(2n, \IR))$, and 
$\mathcal{O}(
SO(2n+1, \IR)\times_{S^{2n}} SO(2n+1, \IR))$ with $\mathcal{O}(
SO(2n+1, \IR))\otimes_{\mathcal{O}(S^{2n})} \mathcal{O}(
SO(2n+1, \IR))$, hence principality of the $SO(2n,\IR)$-bundle 
$SO(2n+1, \IR)\to S^{2n}$ implies that the algebra extension $\mathcal{O}(S^{2n}) \subset \mathcal{O}(SO(2n+1, \IR))$ is Hopf Galois with $H=\mathcal{O}(SO(2n, \IR))$.
\\

Next we consider a 2-cocycle $\gamma$ on the quantum subgroup $\mathcal{O}(SO(2n, \IR))$, or rather on its maximal torus $\mathbb{T}^n$, and use it
to deform the quantum homogeneous space $\mathcal{O}(S^{2n})$ and the principal
fibration on it.
Let $\mathcal{O} (\mathbb{T}^n)$ be  the  commutative $*$-algebra of functions on the $n$-torus  with generators $t_j,~{t_j}^*=*(t_j)$ satisfying $t_j {t_j}^*=1= {t_j}^* t_j$ (no sum on $j$) for  $j=1,\dots n$. It is a  Hopf algebra with 
$$
\Delta(T)=T \overset{.}{\otimes} T \; , \; \varepsilon(T)=\mathbbm{1}\; , \; S(T)=T^* \; , \quad T:= \mathrm{diag}(t_1, \dots t_n,
t_1^* , \dots t_n^*).
$$
We consider the exponential  2-cocycle $\gamma$  on $\mathcal{O} (\mathbb{T}^n)$
defined on the generators $t_i$ by
\be\label{cocycle-exp}
\co{t_j}{t_k}= \exp\big(i \pi \theta_{jk}\big)  \quad ;\quad \theta_{jk}=- \theta_{kj} \in \mathbb{R}
\ee
and  extended to the whole algebra by requiring 
$\co{ab}{c}=\co{a}{\one{c}}\co{b}{\two{c}}$ and 
 $\co{a}{bc}=\co{\one{a}}{c}\co{\two{a}}{b}$, for all $a,b,c, \in \mathcal{O} (\mathbb{T}^n)$,
(cf.\ \ref{cocycleT2}). The Hopf algebra $\mathcal{O} (\mathbb{T}^n) $ is a quantum subgroup of $\mathcal{O}(SO(2n, \IR))$
with projection
\be
  M \mapsto T, \quad  \mbox{  i.e.,} ~~~
a_{ij}\mapsto \delta^i_j t_i \quad ; \quad {a}^*_{ij}\mapsto
\delta^i_j t^*_i \quad ; \quad b_{ij}\mapsto 0\quad ; \quad
b^*_{ij}\mapsto 0
\ee
and hence the 2-cocycle $\gamma$ lifts by pullback to a 2-cocycle on
$\mathcal{O}(SO(2n, \IR)) $ (see Lemma \ref{lem:coc-by-pb}), that we
still denote by $\gamma$. Now  to deform with $\gamma$ the Hopf
algebra $\mathcal{O}(SO(2n, \IR))$ into the noncommutative Hopf algebra
$\mathcal{O}(SO_\theta(2n, \IR))$. The twisted algebra product is given by
(cf.\ \eqref{hopf-twist})   
$$
M_{IJ} \mt M_{KL}= \co{T_I}{T_K} M_{IJ} M_{KL}\coin{T_J}{T_L} , \quad
I,J,K,L=1, \dots 2n.
$$
Since $\co{T_I}{T_K} = (\co{T_K}{T_I})^{-1}$ and similarly for  $\bar\gamma$, it follows  that
the generators in $\mathcal{O}(SO_\theta(2n, \IR))$  satisfy the commutation relations
$$
M_{IJ} \mt M_{KL}= \big(\co{T_I}{T_K}\big)^2 \left(\coin{T_J}{T_L} \right)^2 M_{KL} \mt M_{IJ},\quad I,J,K,L=1, \dots 2n .
$$
Explicitly, setting $\lambda_{IJ}:=(\gamma(T_I\otimes T_J))^2$, so that
$\lambda_{ij}= \exp(2i\pi \theta_{ij})$, and since
$\bar\gamma(T_J\otimes T_L)=\gamma(T_L\otimes T_J)$, they
read 
\begin{eqnarray}\label{thetaCR}
a_{ij} \mt a_{kl} = \lambda_{ik}\lambda_{lj} ~a_{kl}\mt a_{ij} &,  & 
a_{ij} \mt b^*_{kl} = \lambda_{ki}\lambda_{lj} ~b^*_{kl}\mt a_{ij} \nn
\\
a_{ij} \mt b_{kl} = \lambda_{ik}\lambda_{jl} ~b_{kl}\mt a_{ij} &,& 
a_{ij} \mt a^*_{kl} = \lambda_{ki}\lambda_{jl} ~a^*_{kl}\mt a_{ij} \nn
\\
b_{ij} \mt b_{kl} = \lambda_{ik}\lambda_{lj} ~b_{kl}\mt b_{ij} &,& 
b_{ij} \mt b^*_{kl} = \lambda_{ki}\lambda_{jl} ~b^*_{kl}\mt b_{ij} 
\end{eqnarray}
together with their $*$-conjugated.  
It is also not difficult to show the equivalence of the quotient
conditions (\ref{idealQ}) with the relations 
\begin{eqnarray}
\label{qidealQ}
M^t \mt Q \mt M =Q~,~~
 M \mt Q \mt M^t=Q~,~~
{\det}_{\theta}(M)=1
\end{eqnarray}
where the quantum determinant is defined by
\begin{eqnarray}\label{qdet}{\det}_{\theta}(M)=\sum_{\sigma \in {\cal P}_{2n}}(-1)^{|\sigma|\,}\big(\!\prod_{\mbox{${}^{~\;I<J}_{\,\sigma_I>\sigma_J}$}} \!\lambda_{\sigma_I\sigma_J}\big)\,
M_{1\sigma_1}\!\cdot_\gamma\ldots 
M_{2n\,\sigma_{2n}}~.
\end{eqnarray}
A quick way to prove the  orthogonality relations  is to observe
that the new  antipode, obtained according to Proposition \ref{prop:co},
remains undeformed (sum on $L,K,R,P$  indices understood)
\begin{eqnarray}
S_\gamma (M_{IJ})&=& u_\gamma(M_{IL})  S(M_{LK})  \bar{u}_\gamma(M_{KJ})
=\co{M_{IR}}{S(M_{RL})} S(M_{LK}) \coin{S(M_{KP})}{M_{PJ}} \nonumber\\&=&
\co{T_{I}}{T_{I}} S(M_{IJ}) \coin{T_{J}}{T_{J}} =S(M_{IJ}) ~,
\end{eqnarray}
so that the orthogonality relations are the Hopf algebra relations $m_\gamma\circ (S_\gamma
\otimes\id)\Delta(M)=\epsilon(M)$ and  
$m_\gamma\circ (\id
\otimes S_\gamma)\Delta(M)=\epsilon(M)$.
In order to obtain the quantum determinant
relation first use $\gamma(T_{\sigma_I}\otimes
T_{\sigma_J})=\gamma(T_{\sigma_J}\otimes
T_{\sigma_I})\lambda_{\sigma_I\sigma_J}$
to show that for each permutation
$\sigma$ we have the equality $\prod_{I<J, \sigma_I>\sigma_J} \lambda_{\sigma_I\sigma_J}=\prod_{I<J}\bar\gamma(T_I\otimes T_J)  \gamma(T_{\sigma_I}\otimes T_{\sigma_J})$. 
Next  expand the twisted products in (\ref{qdet}) in terms of the
commutative products using $\co{ab}{c}=\co{a}{\one{c}}\co{b}{\two{c}}$
as well as the equivalent relation $\bar\gamma(ab\otimes
c)=\bar\gamma(a\otimes \two{c})\bar\gamma(b\otimes\one{c})$  for all
$a,b,c\in  \mathcal{O} (SO(2n,\mathbb{R})) $ and notice that (\ref{qdet}) becomes the usual determinant of $M$.

From (\ref{thetaCR}) and (\ref{qidealQ}) we see that the twisted Hopf algebra $\mathcal{O}(SO_\theta(2n, \IR))$
can be described algebraically as the algebra over $\mathbb{C}$ freely generated by the matrix
entries $M_{IJ}$ modulo the ideal implementing the relations (\ref{qidealQ}).
The twisted  Hopf algebras $\mathcal{O}(SO_\theta(n, \IR))$ were studied in
\cite{pl96} (see also \cite{res})  and in \cite{cdv} as symmetries of $\theta$-planes and spheres.\\

We can lift the 2-cocycle from the quantum subgroup $\mathcal{O}(SO(2n, \IR))$
to the Hopf algebra $\mathcal{O}(SO(2n+1, \IR))$ by using the projection $\pi$
in \eqref{pi}  (or equivalently we can consider the torus
$\mathbb{T}^n$ embedded in $SO(2n+1)$). The resulting Hopf algebra is
denoted by $\mathcal{O}(SO_\theta(2n+1, \IR))$. It is the Hopf algebra over
$\mathbb{C}$ freely
generated by the matrix entries $N_{IJ}$ modulo the relations 
\begin{equation}\label{comN}
N_{IJ} \mt N_{KL}= \big(\co{T_I}{T_K}\big)^2 \left(\coin{T_J}{T_L}
\right)^2 N_{KL} \mt N_{IJ},\quad I,J,K,L=1, \dots 2n+1 ,
\end{equation}
where now $T:= \mathrm{diag}(t_1, \dots t_n,
t_1^* , \dots t_n^*,1),$ and
\begin{equation}\label{Nort}
N^t \mt Q \mt N =Q~,~
 N \mt Q \mt N^t=Q~,~
{\det}_{\theta}(N)=1~,
\end{equation}
where ${\det}_{\theta}(N)$ is defined as in (\ref{qdet}), just consider the
 permutation group ${\cal P}_{2n+1}$. 

As from Lemma \ref{lem:qhs}, 
the quantum homogeneous space $B=\mathcal{O}(S^{2n})$ is deformed into the
quantum homogeneous space of coinvariants of $\mathcal{O}(SO_\theta(2n+1, \IR))$ 
under the $\mathcal{O}(SO_\theta(2n, \IR))$-coaction. This is the subalgebra
$B_\theta=:\mathcal{O}(S_\theta^{2n})\subset \mathcal{O}(SO_\theta(2n+1, \IR))$ which is
generated by the elements $u_i, u^*_i $ and $x$ entering the last
column of the matrix $N$. Their commutation relations follow from \eqref{comN}
$$
u_i \mt u_j= \lambda_{ij}~ u_j \mt u_i \quad ; \quad
u_i^* \mt u^*_j= \lambda_{ij}~ u^*_j \mt u^*_i \quad;
 \quad u_i \mt u_j^*= \lambda_{ji} ~u_j^* \mt u_i  ~~,
$$
while the orthogonality conditions (\ref{Nort}) imply the sphere
relation $\sum_{i=1}^{n} u_i^* \mt u_i +x^2=1.$    
By Corollary \ref{cor:qhs} we conclude
\begin{prop}
The algebra extension $\mathcal{O}(S_\theta^{2n})\subset \mathcal{O}(SO_\theta(2n+1, \IR))$ of the quantum homogeneous space $\mathcal{O}(S_\theta^{2n})= \mathcal{O}(SO_\theta(2n+1, \IR))^{co\mathcal{O}(SO_\theta(2n, \IR))}$ is  Hopf-Galois.  
\end{prop}
Invertibility of the antipode and injectivity of
$\mathcal{O}(SO_\theta(2n+1, \IR))$ as an $\mathcal{O}(SO_\theta(2n,
\IR))$-comodule imply that $\mathcal{O}(S_\theta^{2n})\subset
\mathcal{O}(SO_\theta(2n+1, \IR))$ is a principal comodule algebra.

\subsection{Twisting sheaves of Hopf-Galois extensions}\label{sec:sheaves} 
In classical geometry a  principal bundle over a topological space  $X$ can be given in 
terms of the local data of trivial product bundles over the open sets of a  covering of $X$ 
and a set of transition functions which specify how to glue the local trivial pieces into a 
(possibly non trivial) global one.  
A local-type approach to noncommutative  principal bundles was  given in \cite{pflaum} 
by using sheaf theoretical methods. A quantum principle bundle consists in the data of two sheaves    
of $\C$-algebras over a (classical) topological space together with a quantum group, 
playing the role of the structure group, and  a family of sheaf morphisms, 
satisfying some suitable conditions, as local trivializations.
The two sheaves of algebras have to be regarded as the quantum analogues of the 
sheaves of functions over the base and total space of a classical  fibration. The basic 
idea behind is that of considering a quantum space as a `quantum ringed space' 
$(M,\mathcal{O}_M)$, i.e.\ a topological space $M$ whose structure sheaf $\mathcal{O}_M$
is  a sheaf of (not necessarily commutative) algebras rather than of commutative rings.
\sk

A refinement of this sheaf theoretical approach to noncommutative bundles 
was proposed  in \cite{cp} in terms of sheaves of Hopf-Galois extensions. For 
simplicity let us here assume all algebras are over a  field.
\begin{defi}
\label{HGsh}
Let $X$ be a topological space and  $H$ a Hopf algebra. A sheaf $\sft$ of 
(not necessarily commutative) algebras over $X$ is said to be a sheaf of $H$-Hopf-Galois extensions  if: 
\begin{enumerate}[(i)]
\item $\sft$ is a sheaf of (say) right $H$-comodules algebras: for each open $U\subseteq X$, $\sft(U)$ is a right $H$-comodule algebra and 
for each  open $W \subseteq U$ the restriction map $r_{UW}:\sft(U) \ra \sft(W)$ 
is a morphism of $H$-comodule algebras;
\item for each open $U\subseteq X$, $\sft(U)^{coH}\subseteq \sft(U)$ is a $H$-Hopf-Galois extension. 
\end{enumerate}
A sheaf $\sft$ of $H$-Hopf-Galois extensions over a topological space $X$ is 
called locally cleft if there exists an open covering $\{U_i\}_{i\in I}$ of $X$ such 
that $\sft(U_i)$ is cleft, $\forall i\in I$.
\end{defi}
The sheaf $\sft$ and its subsheaf $\sft^{coH}: U \mapsto \sft(U)^{coH}$ play the
role of noncommutative analogues of the sheaf of functions on the total space, 
respectively base space, of the bundle.
Notice that condition (ii) is equivalent to requiring just the  algebra $\sft(X)$ 
to be an $H$-Hopf-Galois extension, indeed it was observed in \cite{cp} that 
the property of being  Hopf-Galois restricts locally: if on an open set $U$, the 
algebra extension $\sft(U)^{coH} \subseteq \sft(U)$ is Hopf-Galois, then 
$\sft(W)^{coH} \subseteq \sft (W)$ is  a Hopf-Galois extension for any  $W\subseteq U$. 
(This is the algebraic counterpart of the well-known classical fact that the restriction of a principal action  is still principal).

The notions of quantum principal bundle introduced in \cite{pflaum} and that 
of locally cleft sheaf of Hopf-Galois extensions are closely related: every locally 
cleft sheaf of Hopf-Galois extensions is a quantum principal bundle in the sense of 
\cite{pflaum}.  On the other hand,  a sufficient condition for a quantum principal 
bundle in the sense of \cite{pflaum} to be a sheaf of Hopf-Galois extensions 
(in fact, locally cleft) is that the restriction maps are surjective  (see \cite[\S 4]{cp}).
\sk

Since a (locally cleft) sheaf $\sft$ of  $H$-Hopf-Galois extensions is in particular 
a sheaf of $H$-comodule algebras, given a 2-cocycle 
in $H$ we can  apply the functor $\Gamma$ in \eqref{functGamma} 
and obtain a new sheaf $\sft_\cot$ over the same topological space 
$X$. The sheaf  $\sft_\cot$ is a sheaf of $H_\cot$-comodule algebras
and is defined by $\sft_\cot (U):=
\Gamma(\sft(U))=(\sft(U))_\gamma$, with  restriction 
maps given by  morphisms of $H_\cot$-comodule algebras 
$\Gamma(r_{UW})=r_{UW}:\sft_\cot(U) \ra \sft_\cot(W)$, for all $W\subset U$ open sets.
\sk

By Corollary \ref{cor-can}, we can conclude that  $\sft_\cot$ is a (locally cleft) sheaf of
$H_\cot$-Hopf-Galois extensions if and only if $\sft$ is  a sheaf of (locally cleft) $H$-Hopf-Galois 
extensions. The subsheaves  $\sft_\cot^{coH_\cot}$ and
$\sft^{coH}$ over $X$ coincide (i.e.\ they are isomorphic via the identity maps). 
\sk

Let now $K$ be another Hopf algebra;
we may assume the additional (restrictive) condition for the sheaf $\sft$ to be 
valued in  the category of $(K,H)$-bicomodule algebras, i.e.\ $\sft(U) \in {}^K\mathcal{A}^H$ for each open $U$ and the restriction maps are  morphisms of $(K,H)$-bicomodule algebras. In this case we can deform $\sft$ also  by using a 
2-cocycle $\sigma$ on the external Hopf algebra $K$, or even by using  both $\sigma$ 
on $K$ and $\gamma$ on $H$.  With the same reasoning as above, by using the  results 
obtained in \S \ref{sec:def_es} and \S \ref{sec:combidef}, the 
two sheaves  ${}_\sigma\sft$ and ${}_\sigma\sft_\cot$ obtained in this way are 
sheaves of Hopf-Galois extensions if and only if the original sheaf $\sft$ is. In 
general, the subsheaves ${}_\sigma\sft^{coH}$ and ${}_\sigma \sft_\cot^{coH_\cot}$ 
of coinvariants will not coincide with $\sft^{coH}$.
In the following
subsection we provide an example.


\subsubsection{\label{S4astwistedsheaf}The Hopf bundle over
  $S^4_\theta$ as a twisted sheaf}
We describe the Hopf bundle over $S^4_\theta$ of  Example \ref{exCL}  
as a twist deformation of a sheaf of $(K,H)$-bicomodules algebras over the classical 4-sphere 
$S^4$, where $H={\mathcal{O}}(SU(2))$ and $K={\mathcal O}(\mathbb{T}^2)$. The  algebras ${}_\sigma A=\mathcal{O}(S^7_\theta)$ and ${}_\sigma A^{coH}=\mathcal{O}(S^4_\theta)$ will be
replaced by a locally cleft sheaf  ${}_\sigma\sft$ of  $H$-Hopf-Galois extensions over $S^4$ and
its subsheaf ${}_\sigma\sft^{coH}$ of
coinvariant elements. 

We here outline a bottom up approach based 
on local transition functions on opens of $S^4$, a
complementary top down approach starting from the total space $S^7$  description of the
Hopf-Galois extension $\mathcal{O}(S^4)\subseteq\mathcal{O}(S^7)$ is presented in Appendix \ref{appB}.
\sk
As a first step we define the trivial Hopf-Galois extensions  
\begin{eqnarray}\label{HGtrivialsheaf} &&
\ocn
\subseteq 
 \ocn
\otimes H=:\sft(U_{\scriptscriptstyle{N}})\nn\\[.4em] 
&&\ocs
\subseteq
\ocs
\otimes H=:\sft(U_{\scriptscriptstyle{S}})\nn\\[.4em]
 &&\ocns
\subseteq 
 \ocns
\otimes H=:\sft(U_{\scriptscriptstyle{NS}})
\end{eqnarray} where $\ocn$ denotes the $*$-algebra generated by the $S^4$
coordinates $\alpha,
\beta, x$ and by $c_{\scriptscriptstyle{N}}^{\pm 1}$, with
$c_{\scriptscriptstyle{N}}^{2}=\frac{1}{2}(\1-x)$ (thus  the generator $x$  becomes redundant). Since $x\not=\1$ (i.e., $c_{\scriptscriptstyle{N}}\not=0$)
except in the north pole $N$, these coordinates generate the algebra
of coordinate functions on the open
$U_{\scriptscriptstyle{N}}:=S^4\backslash \{N\}\simeq \mathbb{R}^4$. 
 Similarly, the other
coordinate algebras are over  $U_{\scriptscriptstyle{S}}:=S^4 \backslash \{S\}$, with
$c_{\scriptscriptstyle{S}}^{2}=\frac{1}{2}(\1+x)$, and
$U_{\scriptscriptstyle{NS}}:=U_{\scriptscriptstyle{N}}\backslash \{S\}$
(cf.\ Appendix \ref{appsheafB}).
\sk
 Next we introduce the  restriction maps defining the sheaf $\sft$ of locally trivial Hopf-Galois
 extensions, and precisely the trivial restriction map 
\be\label{restrFN}
r^{\sft}_{\scriptscriptstyle{N,NS}}: 
\ocn
 \ot H
~\xhookrightarrow{i_{\scriptscriptstyle{N}}\ot \id}~
\ocns
\ot H~,
\ee
(where $i_{\scriptscriptstyle{N}}$ denotes the canonical injection)
and the nontrivial one (defined on the generators and extended as
$*$-algebra map),\footnote{This restriction map encodes the
  information on the transition function $g_{\scriptscriptstyle{NS}}$ characterizing the  two
  charts
  $U_{\scriptscriptstyle{N}}, U_{\scriptscriptstyle{S}}$  description
  of the  Hopf bundle $S^7\to S^4$. Indeed we have
   $g_{\scriptscriptstyle{NS}}: U_{\scriptscriptstyle{NS}}\to SU(2)~,~~ (\alpha,\beta,x)\mapsto 
\frac{1}{2} c_{\scriptscriptstyle{N}}^{-1}c_{\scriptscriptstyle{S}}^{-1}\begin{pmatrix}
 \alpha & -  \beta^*
\\[.4em]
\beta &  \alpha^* 
\end{pmatrix}
\,$, (we use the same notation for the coordinate functions and the point coordinates).}
\begin{align}\label{restrFS}
r^{\sft}_{\scriptscriptstyle{S,NS}}: 
\ocs
 \ot H &
~\longrightarrow ~
\ocns
\ot H~,\\
\1\otimes \begin{pmatrix}
 w_1 & -  w_2^*
\\[.4em]
w_2 &  w_1^* 
\end{pmatrix}&~\longmapsto~  \frac{1}{2} c_{\scriptscriptstyle{N}}^{-1}c_{\scriptscriptstyle{S}}^{-1}\begin{pmatrix}
 \alpha & -  \beta^*
\\[.4em]
\beta &  \alpha^* 
\end{pmatrix}\overset{.}{\otimes}\begin{pmatrix}
 w_1 & -  w_2^*
\\[.4em]
w_2 &  w_1^* 
\end{pmatrix} \nn \\
f\otimes \1 & ~\longmapsto~ i_{\scriptscriptstyle{S}}(f)\otimes \1 \nn
\end{align}
where $ i_{\scriptscriptstyle{S}}$ is the canonical injection 
$\ocs
\xhookrightarrow{i_{\scriptscriptstyle{S}}}
\ocns$.
It is straightforward to check that these restriction maps are
morphisms of $H$-comodule algebras. Since $\{\emptyset, U_{\scriptscriptstyle{N}},
U_{\scriptscriptstyle{S}}, U_{\scriptscriptstyle{NS}}\}$ is a basis
of the topology $\{\emptyset, U_{\scriptscriptstyle{N}},
U_{\scriptscriptstyle{S}}, U_{\scriptscriptstyle{NS}}, S^4\}$
the Hopf-Galois extensions in (\ref{HGtrivialsheaf}) and the
restriction maps (\ref{restrFN}), (\ref{restrFS}) uniquely define
the locally cleft sheaf $\sft$ on the  topology $\{\emptyset, U_{\scriptscriptstyle{N}},
U_{\scriptscriptstyle{S}}, U_{\scriptscriptstyle{NS}}, S^4\}$  (to
$\emptyset$ we assign the one element algebra, terminal object in the
category of algebras). 

In particular the
Hopf-Galois extension on the sphere $S^4$ is obtained as the
pull-back (in the category of $*$-algebras) 
\begin{equation}
\sft(S^4):=\{(a_{\scriptscriptstyle{N}},a_{\scriptscriptstyle{S}}) \in
  \sft(U_{\scriptscriptstyle{N}}) \times \sft(U_{\scriptscriptstyle{S}})~|~ r^{\sft}_{\scriptscriptstyle{N,NS}}(a_{\scriptscriptstyle{N}})=
r^{\sft}_{\scriptscriptstyle{S,NS}}(a_{\scriptscriptstyle{S}})\}~.
\end{equation} 
From Lemma \ref{lem:triv} in Appendix \ref{appB}  (for the notation used
see (\ref{sheafglobal}) and (\ref{coiAcoH})) and the $H$-comodule
algebra isomorphism (\ref{restrA})
we immediately conclude that the pull-back $\sft(S^4)$ is isomorphic to
$\mathcal{O}(S^7)$ as an $H$-comodule algebra. Then the subalgebra of coinvariants 
is $\mathcal{O}(S^4)$ and the Hopf-Galois extension $\sft(S^4)^{coH}\subseteq \sft(S^4)$
describes the instanton bundle $S^7\to S^4$. 

\sk
 Finally  the sheaf $\sft$ is a sheaf of
 $(K,H)$-bicomodule algebras, where the $K$-coactions are given by 
\begin{align}
\rho^{\sft}_{\scriptscriptstyle{N}}: 
\ocn
 \ot H &
~\longrightarrow ~
K\otimes \ocn
\ot H~,\\
\1\otimes \begin{pmatrix}
 w_1 & -  w_2^*
\\[.4em]
w_2 &  w_1^* 
\end{pmatrix}&~\longmapsto~  \begin{pmatrix}
 t_2 & 0
\nn\\[.4em]
0 &  t_2^*
\end{pmatrix}\overset{.}{\otimes}\1\overset{.}{\otimes}\begin{pmatrix}
 w_1 & -  w_2^*
\\[.4em]
w_2 &  w_1^* 
\end{pmatrix}\nn\\
 \alpha\ot \1 &~ \longmapsto~ t_1 t_2^* \ot \alpha\ot \1\nn\\
\beta\ot \1 &~\longmapsto ~t_1^* t_2^*  \ot \beta\ot \1 \nn\\[.8em]
\rho^{\sft}_{\scriptscriptstyle{S}}: 
\ocs
 \ot H &
~\longrightarrow ~
K\otimes \ocs
\ot H~,\\
\1\otimes \begin{pmatrix}
 w_1 & -  w_2^*
\\[.4em]
w_2 &  w_1^* 
\end{pmatrix}&~\longmapsto~  \begin{pmatrix}
 t_1 & 0
\\[.4em]
0 &  t_1^*
\end{pmatrix}\overset{.}{\otimes}1\overset{.}{\otimes}\begin{pmatrix}
 w_1 & -  w_2^*
\\[.4em]
w_2 &  w_1^* 
\end{pmatrix}\nn\\
\alpha \ot \1&~ \longmapsto~ t_1 t_2^* \ot \alpha\ot \1\nn\\
\beta\ot \1 &~\longmapsto ~t_1^* t_2^* \ot \beta \ot \1\nn
 \end{align}
with $ x\ot \1$,  $c_{\scriptscriptstyle{N}}^{\pm 1}\ot 1$ and
$c_{\scriptscriptstyle{S}}^{\pm 1}\ot \1$ coinvariant;
likewise
$\rho^{\sft}_{\scriptscriptstyle{NS}}: 
\ocns
 \ot H 
\rightarrow
K\otimes \ocns\ot H$ is the extension of $\rho^{\sft}_{\scriptscriptstyle{N}}$ obtained by defining $c_{\scriptscriptstyle{S}}^{\pm 1}\ot \1$ to be coinvariant.  
Observe that the $K$-coactions $\rho^{\sft}_{\scriptscriptstyle{N}}$
and $\rho^{\sft}_{\scriptscriptstyle{S}}$ differ from each other on $1 \ot H$, henceforth   the nontrivial
restriction map $r^{\sft}_{\scriptscriptstyle{S,NS}}$ is a morphism of
$K$-comodule algebras. 
\sk

We can now consider the 2-cocycle $\sigma$ in \eqref{cocycleT2} on $K$
and use it to deform $\sft(S^4)^{coH}\subseteq \sft(S^4)$ to
${}_\sigma\sft(S^4)^{coH}\subseteq {}_\sigma\sft(S^4)$, and the commutative and trivial Hopf-Galois extensions in
(\ref{HGtrivialsheaf}) into the noncommutative and trivial
 Hopf-Galois extensions 
\begin{eqnarray}\label{NCHGtrivialsheaf} &&{}_\sigma\ocn
\subseteq 
 {}_\sigma\ocn
\otimes H=: {}_\sigma\sft(U_{\scriptscriptstyle{N}})\nn\\[.4em] 
&& {}_\sigma\ocs
\subseteq
  {}_\sigma\ocs
\otimes H=: {}_\sigma\sft(U_{\scriptscriptstyle{S}})\nn\\[.4em]
 && {}_\sigma\ocns
\subseteq 
  {}_\sigma\ocns
\otimes H=: {}_\sigma\sft(U_{\scriptscriptstyle{NS}})~.
\end{eqnarray} 
The corresponding restriction maps are 
$\Sigma(r^{\sft}_{\scriptscriptstyle{N,NS}})\!=\!r^{\sft}_{\scriptscriptstyle{N,NS}}\!:
{}_\sigma\sft(U_{\scriptscriptstyle{N}})\rightarrow
{}_\sigma\sft(U_{\scriptscriptstyle{NS}})$, 
$\Sigma(r^{\sft}_{\scriptscriptstyle{S,NS}})\!=\!r^{\sft}_{\scriptscriptstyle{S,NS}}\!:{}_\sigma\sft(U_{\scriptscriptstyle{S}})\to{}_\sigma\sft(U_{\scriptscriptstyle{NS}})$ and 
$\Sigma(pr^\sft_{\scriptscriptstyle{1}})=pr^\sft_{\scriptscriptstyle{1}}: {}_\sigma\sft(S^4)\to {}_\sigma\sft(U_{\scriptscriptstyle{N}})$,
$\Sigma(pr^\sft_{\scriptscriptstyle{2}})=pr^\sft_{\scriptscriptstyle{2}}: {}_\sigma\sft(S^4)\to {}_\sigma\sft(U_{\scriptscriptstyle{S}})$;
they are $(K _\sigma,H)$-comodule maps and define the
sheaf $ {}_\sigma\sft$, that by construction is locally cleft.
Since the Hopf-Galois extension $\sft(S^4)^{coH}\subseteq \sft(S^4)$ is
isomorphic to $\mathcal{O}(S^4)\subseteq
\mathcal{O}(S^7)$ then ${}_\sigma\sft(S^4)^{coH}\subseteq
{}_\sigma\sft(S^4)$ is
isomorphic to $\mathcal{O}(S^4_\theta)\subseteq
\mathcal{O}(S^7_\theta)$, and the sheaf of Hopf Galois extensions
${}_\sigma\sft$
gives a sheaf description of the Hopf bundle over $S^4_\theta$ addressed in
 Example \ref{exCL}.


\appendix

\section{Twists, 2-cocycles and untwisting}\label{app:twists}
We briefly outline the duality between 
the notions of Drinfeld twists  \cite{Dri83, Dri85} and 2-cocycles
that was mentioned in Section \ref{sec:twists} and illustrate the `untwisting procedure'. 

\subsection{Drinfeld twists}
\begin{defi}
Let $\U$ be a bialgebra (or Hopf algebra). An invertible counital twist  on 
$\U$, or
simply a {\bf twist},  is an invertible  element $\F \in \U \ot \U$ such
that   $(\varepsilon \ot \id) (\F)= \1= (\id \ot \varepsilon)(\F)$ and
\be\label{twist}
(\F \ot \1)[(\Delta \ot \id)(\F)]= (\1 \ot \F)[(\id \ot \Delta) (\F)] ~.
\ee
We use the notations $\F=\f^{ \alpha} \ot \f_\alpha\in \U
\otimes \U$ and  $\F^{-1}=: \ofu{\alpha} \ot
\ofd{\alpha}\in \U\otimes\U$ (with summations understood).
\end{defi}

Given a twist $\F\in \U \otimes\U$ we can deform the bialgebra (or
Hopf algebra) $\U$ according to the following
\begin{prop}\label{prop:twist}
Let $\F=\f^{ \alpha} \ot \f_\alpha$  be a twist on a bialgebra  $\U$. Then
the algebra $\U$ with coproduct 
\be
\Delta_\F(\xi):= \F \Delta(\xi) \F^{-1}= \f^{~\alpha} \one{\xi} \ofu{\beta} \ot \f_\alpha \two{\xi}\ofd{\beta}~,
\ee
for all $\xi\in \U$, and counit unchanged is a bialgebra, denoted $\U_\F$. If moreover $\U$ is a Hopf algebra, then the
twisted bialgebra $\U_\F$
is a Hopf algebra with antipode $S_\F(\xi):=u_\F S(\xi) u_\F^{-1}$, where
$u_\F:=\f^{~\alpha} S(\f_\alpha)$, with inverse $u_\F^{-1}= S(\ofu{\alpha})\ofd{\alpha}$.
\end{prop}

Furthermore, if  $A $ is a left $\U$-module algebra via 
$\rhd: \U \ot A \ra A$, then 
the $\bbK$-module $A$ with unchanged unit and twisted product
\be\label{lmod-F}
a \bullet_\F a':= (\ofu{\alpha} \rhd a)\, (\ofd{\alpha} \rhd a') ~, 
\ee
for all  $a,a' \in A$,
is a left $\U_\F$-module algebra with respect to the same action. 
We denote the twisted algebra by $A_\F$, with $\U_\F$-module structure
given by $\rhd$, now thought of as 
a map  $\U_\F \ot A_\F \ra A_\F$.

\subsection{Duality between twists and $2$-cocycles}\label{dualconst}
We here clarify how the two constructions of deforming by 2-cocycles
and twists are dual to each other.
 Suppose $H$ and $\U$ are dually paired bialgebras (or Hopf algebras)
 with pairing  $\langle~,~ \rangle : \U \times H \ra \bbK$, i.e., for
 all $\xi,\zeta\in \U$ and $h,k\in H$ we have 
$\langle\xi\zeta , h\rangle=\langle \xi, \one{h}\rangle\langle \zeta,\two{h}\rangle$, $\langle\xi,hk\rangle=\langle\one{\xi},h\rangle\langle\two{\xi}
   ,k\rangle$, $\langle \xi, \1_H\rangle=\varepsilon_\U(\xi)$, $\langle \1_\U, h\rangle=\varepsilon_H(h)$.
 Then to each invertible and counital  twist $\F=\f^\alpha \ot \f_\alpha \in \U \ot \U$ there corresponds a convolution invertible 
 and unital 2-cocycle $\cot_\F : H \otimes H \ra \bbK$ on $H$ defined by 
\be\label{tw2coc}
\cot_\F(h \ot k) := \langle \f^\alpha ,h \rangle \, \langle \f_\alpha , k\rangle ~,
\ee
for all  $h,k\in H$,
with convolution inverse $\bar\gamma_\F(h \ot k) = \langle \ofu{\alpha} ,h \rangle  \, \langle \ofd{\alpha} , k\rangle 
$, for all $h,k\in H$.
The 2-cocycle condition  for $\cot_\F$ follows from the  twist
condition for $\F$; indeed condition  \eqref{twist} in the
$\F=\f^\alpha\ot\f_\alpha$ notation reads as
\be\label{twist-exp} 
 \f^\alpha\, \one{\f^\beta}\ot \f_\alpha\, \two{\f^\beta} \ot {\f_\beta}
=
 \f^\alpha\ot \f^\beta\, \one{\f_\alpha} 
\ot \f_\beta \,\two{\f_\alpha}~,
\ee 
so that
\begin{eqnarray}
\nn \cot_\F ({\one{g}} \ot{\one{h}}) \,\cot_\F({\two{g}\two{h}}\ot{k}) &=&
\langle \f^\alpha, \one{g}  \rangle \, \langle \f_\alpha,\one{h}  \rangle\,
\langle \one{\f^\beta},\two{g}  \rangle\, \langle \two{\f^\beta},\two{h}  \rangle\,
\langle {\f_\beta},k  \rangle \\
\nn&=&
\langle \f^\alpha \one{\f^\beta},g \rangle \, \langle \f_\alpha \two{\f^\beta},h \rangle\,
\langle {\f_\beta},k  \rangle\\
\nn&=&
\langle \f^\alpha ,g \rangle \, \langle \f^\beta \one{\f_\alpha},h \rangle\,
\langle \f_\beta \two{\f_\alpha},k \rangle \\
\nn&=&
\langle \f^\alpha ,g \rangle\, \langle \f^\beta ,\one{h} \rangle\,
\langle  \one{\f_\alpha},\two{h} \rangle\,
\langle \f_\beta ,\one{k} \rangle \,
\langle  \two{\f_\alpha}, \two{k} \rangle 
\\
\nn&=&
\langle \f^\alpha ,g \rangle  \,\langle \f^\beta , \one{h} \rangle \,
\langle \f_\beta , \one{k} \rangle \,
\langle \f_\alpha , \two{h} \two{k}  \rangle \\
&=& 
 \cot_\F( {\one{h}}\ot{\one{k}}) \,\cot_\F ({g}\ot{\two{h}\two{k}}) ~ .
\end{eqnarray}
If we use $\F$ to twist the coproduct in $\U$ according to Proposition
 \ref{prop:twist} and $\cot_\F$ to deform the product in $H$ as in Proposition \ref{prop:co}, then  
the deformed  bialgebras (or Hopf algebras) $\U_\F$ and $H_{\cot_{\F}}$ are dually paired via the same pairing
 $\langle ~,~ \rangle$; indeed, it is easy to prove that $\langle \Delta_\F (\xi) ,h \ot k \rangle = \langle \xi , h \cdot_{\cot_\F} k \rangle$
 for all $\xi\in\U$ and $h,k\in H$.
 \sk

Notice that if $A$ is a right $H$-comodule algebra via $\delta^A : A\to A\otimes H\,,~ a \mapsto \zero{a} \ot \one{a}$, then $A$ is 
a left $\U$-module algebra with left $\U$-action $\rhd: \U\ot A\to A$,
$(\xi,a)\mapsto \xi \rhd a:= \zero{a} \langle \xi, \one{a} \rangle$. Hence,
 we can twist the product in $A$ by using $\F$  as in \eqref{lmod-F}   or
by using  $\cot_\F$ as in \eqref{rmod-twist}. The two constructions give the same algebra $A_\F = A_{\cot_\F}$; indeed,
\begin{eqnarray}
\nn a \bullet_{\cot_\F} a' &=& \zero{a} \zero{a'} \, \bar\gamma_\F ({\one{a}} \ot{\one{a'}})
=
\zero{a} \zero{a'} \, \langle \ofu{\alpha}, \one{a} \rangle\,  \langle \ofd{\alpha}, \one{a'} \rangle \\
&=&
(\ofu{\alpha} \rhd a)\, (\ofd{\alpha} \rhd a') = a \bullet_\F a'~,
\end{eqnarray}
for all $a,a^\prime\in A$.
Finally we  observe that for a  2-cocycle $\cot_\F$ associated
with a twist 
$\F=\f^\alpha \ot \f_\alpha \in \U \ot \U$, the map $\varphi_{A,A}$ introduced in Theorem
 \ref{thm:funct} reads
\be
\varphi_{A,A}(a \ot a')= (\ofu{\alpha} \rhd a) \ot (\ofd{\alpha} \rhd a')=: \F^{-1}  \rhd(a \ot a') ~,
\ee
for all $a,a' \in A_{\cot_\F}$.

\subsection{Untwisting with $2$-cocycles}
We show that if  we  twist a bialgebra (or Hopf algebra)  $H$ to $H_\cot$ via a 2-cocycle $\cot$ on $H$  we
can untwist $H_\cot$ to $H$ via the 2-cocycle $\bar\cot$ on
$H_\cot$. More in general we have,

\begin{prop}Let $\cot$ be a 2-cocycle on a bialgebra (or Hopf algebra) $H$, and
let $\hg$ be the corresponding twisted bialgebra (or Hopf algebra).  
Then $\tau$ is a 2-cocycle on $H_\gamma$ if and only if $\tau *\cot$ is a 2-cocycle
 on $H$.
Furthermore, the twisted bialgebras (or Hopf algebras)
${(\hg)}_{\!\tau}$ and $H_{\tau *\cot}$ coincide. 
 \end{prop}
\begin{proof}
By definition $\tau *\cot$ is a 2-cocycle on $H$ if and only if for all
$g,h,k\in H$,
$$
(\tau *\cot)(\one{g} \ot\one{h}) \, (\tau *\cot)(\two{g}\two{h} \ot{k}) =(\tau *\cot)(\one{h} \ot{\one{k}}) \,(\tau *\cot)({g} \ot{\two{h}\two{k}}) ~,
$$
this equality equivalently reads
\begin{multline*}
\taug{\one{g}}{\one{h}}  \co{\two{g}}{\two{h}} 
 \taug{\three{g}\three{h}}{\one{k}} \co{\four{g}\four{h}}{ \two{k}}\\
=\taug{\one{h}}{\one{k}}  \co{\two{h}}{\two{k}} 
 \taug{\one{g}}{\three{h}\three{k}} \co{\two{g}}{\four{h}\four{k}} ~,
\end{multline*}
and since
$\cot(\one{g}\ot\one{h})\two{g}\two{h}=\one{g}\!\mt\one{h}\,\cot(\two{g}\ot\two{h})$
for all $g,h$, the equality holds  if and only if 
\begin{multline*}
\taug{\one{g}}{\one{h}} 
 \taug{\two{g}\mt\two{h}}{k}  \co{\three{g}}{\three{h}} \co{\four{g}\four{h}}{ \two{k}}\\
 =  
\taug{{\one{h}}}{\one{k}}   \taug{g}{\two{h}\mt\two{k}} \co{\three{h}}{\three{k}}  \co{\two{g}}{\four{h}\four{k}} ~,
\end{multline*}
i.e., since $\cot$ is a twist on $H$, if and only if $\tau$ is a twist
on $H_\cot$:
$$
\taug{\one{g}}{\one{h}} \taug{\two{g} \mt \two{h}}{k} 
=  \taug{\one{h}}{\one{k}} \taug{g}{\two{h}\mt \two{k}}~.
$$
It is straightforward to show that  the twisted product $\cdot_{\tau *\cot}$ in $H_{\tau * \cot}$
equals the twisted product ${\cdot_\cot}_\tau$ in ${(\hg)}_{\!\tau\!\:}$; indeed,
$$h\cdot_{\tau *\cot} g=  (\tau *\cot)(\one{h}\ot\one{k})  \, \two{h}\two{k} \, ( \,\overline{\!\tau\! *\!\cot\!}\,)(\three{h} \ot\three{k}) =  \tau(\one{h}\ot\one{k})  \, \two{h}\mt\two{k} \, \tau(\three{h} \ot\three{k}) =h {\,\cdot_\cot}_{\tau\,} k ~.$$ 
Since the antipode if it exists is unique we immediately have the
statement for Hopf algebras.
\end{proof}

Setting $\tau=\bar\cot$, since $\bar\cot *\cot=\varepsilon\ot
\varepsilon$ is trivially a 2-cocycle, we conclude
that the twisted bialgebra (or Hopf algebra)
 $\hg$ can be `untwisted' via the convolution inverse  $\bar\cot$:
\begin{cor}\label{cor:untwist}
 If $\cot$ is a 2-cocycle on the bialgebra (or Hopf algebra)  $H$, then its
 convolution inverse $\bar\cot$ is a  2-cocycle on the bialgebra (or
 Hopf algebra) $\hg$,
 and
 ${(\hg)}_{\bar\cot}$
is isomorphic to $H$ (through the identity map). 
\end{cor}

\section{Equivalence of closed monoidal categories and the $\Q$-map}\label{appC}
In this section we show how the $\Q$-map of Theorem \ref{prop:mapQ}
is related (by duality) to the natural transformation
which establishes that twisting may be regarded as an 
equivalence of closed monoidal categories.
\sk

Recall from Theorem \ref{prop:mapQ} that $\Q : \underline{H_\cot} \to \underline{H}_\cot$
is a right $H_\cot$-comodule isomorphism, where $\underline{H_\cot}$ carries the $\Ad_\cot$-coaction
and  $\underline{H}_\cot$ the $\Ad$-coaction (regarded as an $H_\cot$-coaction).
Assume that $H'$ and $H$ are dually paired Hopf algebras 
with pairing $\langle~,~\rangle: H'\otimes H\to \bbK$, and let $\gamma$ be a
2-cocycle on $H$ with corresponding dual twist $\F\in H'\otimes H'$.  Then
the Hopf algebras $H_\gamma$ and $H'_\F$ are dually paired and the
right $H_\cot$-comodules $\underline{H_\gamma}$ and
$\underline{H}_\gamma$ are dually paired to the right $H'_\F$-modules 
$\underline{H'_\F}$ and ${\underline H}'_\F$. These coincide with $H^\prime$
as $\bbK$-modules and by definition have right
$H'_\F$-adjoint actions respectively given by  
\begin{flalign}
\blacktriangleleft_{\F}\,  : \underline{H'_\F}  \otimes H_\F^\prime\longrightarrow  \underline{H'_\F} ~,~~\zeta\otimes \xi 
\longmapsto S_\F(\xi_{\scriptscriptstyle{(1)_\F}})\, \zeta\, \xi_{\scriptscriptstyle{(2)_\F}}~
\end{flalign}
and
\begin{flalign}
\blacktriangleleft \,: {\underline H}'_\F \otimes   H_\F^\prime \longrightarrow  {\underline H}'_\F ~,~~\zeta\otimes \xi 
\longmapsto S(\one{\xi})\, \zeta\, \two{\xi} ~.
\end{flalign}
(The dual pairing extends also to a dual pairing between right $H_\cot$-comodule
coalgebras and right $H_\F^\prime$-module algebras).
\sk

The isomorphism $\Q: \underline{H_\gamma}\to \underline{H}_\gamma$ of
Theorem \ref{prop:mapQ} can be dualized to an isomorphism 
\begin{flalign}
\Q': {\underline H}'_\F \longrightarrow  \underline{H'_\F} 
\end{flalign}
by setting $\langle \Q'(\xi),h\rangle=\langle\xi,\Q(h)\rangle$, for all $\xi\in
{\underline H}'_\F$ and $h\in \underline{H_\gamma}$. 
Explicitly, we have that
\begin{flalign}
\Q'(\xi)=\f^\beta \,(\xi\blacktriangleleft \f_\beta) = \f^{ \beta} \, S(\f_{\beta_{(1)}})\, \xi\, \f_{\beta_{(2)}}~,
\end{flalign}
for all $\xi\in {\underline H}'_\F$. 
Recall that right $H'$-modules are equivalently left ${{H'}^{op}}^{cop}$-modules and right 
$H'_\F$-modules are left ${({H'}_\F)^{op}}^{cop}=({{{H'}^{op}}^{cop}})_{{\F^{op}}^{cop}}$-modules, 
where ${\F^{op}}^{cop}={\F_{21}^{-1}}=\ofd{\alpha}\otimes \ofu{\alpha}$.
The map $\Q':  {\underline H}'_\F \to  \underline{H'_\F} $ is therefore
equivalently an isomorphism of left $({{{H'}^{op}}^{cop}})_{{\F^{op}}^{cop}}$-modules, 
and using the left $({{{H'}^{op}}^{cop}})_{{\F^{op}}^{cop}}$-action it reads as
\begin{flalign}
\Q'(\xi) = \f_{\beta_{(1)^{cop}}} \,\cdot^{op}\,\xi \,\cdot^{op}\, {S^{op}}^{cop} (\f_{\beta_{(2)^{cop}}})\,\cdot^{op}\,\f^{\beta}
= (\f_{\beta}{{\blacktriangleright}^{op}}^{cop} \xi)\,\cdot^{op}\,\f^\beta~.
\end{flalign}
Referring to \cite[Section 3.2]{AS}, it follows that $\Q'$ is precisely the isomorphism
$D_{{\F^{op}}^{cop}}$ for the Hopf algebra ${{{H'}^{op}}^{cop}}$ twisted by
the twist ${\F^{op}}^{cop}$.
It has been shown in \cite{BSS} that such $D$-maps have a categorical interpretation in terms of the
natural isomorphism which establishes that twisting is an equivalence of closed monoidal categories.
Hence, in conclusion, the dual of our $\Q$-map can be given a categorical interpretation.

\section{The twisted sheaf of the Hopf bundle over $S^4_\theta$: top down approach}\label{appB} 
We complement the example of the twisted sheaf in \S \ref{S4astwistedsheaf}
by presenting a top down approach: we first describe $S^7$ as a ringed
space, then on these algebras (rings)  of coordinate functions on 
opens of $S^7$ we induce the $H$-coaction 
leading to a sheaf $\ft$ of $H$-comodule algebras (with $H={\mathcal{O}}(SU(2))$).
Next we show that this is a locally cleft sheaf of $H$-Hopf Galois
extensions, and as a corollary that it is
naturally isomorphic to the sheaf $\sft$ of \S \ref{S4astwistedsheaf}, $\ft\simeq\sft$.
 Finally, in the last of the paragraphs titled in italics, the torus action on $\pi:S^7 \ra
S^4$ is pulled back to this sheaf description and the corresponding
twist deformation is obtained. In Section \ref{appsheafB} we study the subsheaf of
$H$-coinvariants, it is generated by two copies (of the exponential
version) of the
Moyal-Weyl
algebra on $\mathbb{R}^4_\theta$ that describe $S^4_\theta$ as a ringed space.
\bigskip

{\noindent {\it The sheaf $\ft$ over $S^4$ of coordinate functions on opens of $S^7$}}\\
As in \S \ref{S4astwistedsheaf}  we consider the sphere $S^4$ with topology $\{\emptyset,
U_{\scriptscriptstyle{N}}, U_{\scriptscriptstyle{S}},
U_{\scriptscriptstyle{NS}}, S^4 \}$; it is generated by the
basis with (open) sets $\emptyset$,  
 $U_{\scriptscriptstyle{N}}$, $U_{\scriptscriptstyle{S}}$ and their
 intersection $U_{\scriptscriptstyle{NS}}$. The topology on $S^4$
 induces a topology on $S^7$ given by the opens $\pi^{-1}(U)$, with $U$ open in $S^4$.

We define a sheaf $\ft$ of algebras on $S^4$ by
assigning an algebra to each open of the basis for the topology on $S^4$. To the
empty set $\emptyset$ we assign the algebra $\ft(\emptyset)$ that is the one-element
algebra (where unit and zero elements coincide), while to the remaining
open sets of the basis we define
$\ft(U)$  as quotients of central real extensions of
$A=\mathcal{O}(S^7)$,
the coordinate algebra on $S^7 $ generated by the commuting elements
 $z_i, z^*_i$ ($i=1,...4$) satisfying the sphere condition $\sum
z_i z_i^*=1$. Explicitly we define
\begin{eqnarray}\label{sheafglobal}
\ft(U_{\scriptscriptstyle{N}})&\!\!\!:=&\!\!\! \mathcal{O}(S^7)[c_{\scriptscriptstyle{N}},c_{\scriptscriptstyle{N}}^{-1}] ~\big/ \langle z_3z_3^*+z_4z_4^*-c_{\scriptscriptstyle{N}}^{2} ~,~c_{\scriptscriptstyle{N}}c_{\scriptscriptstyle{N}}^{-1}-1\rangle~,\\[.6em]
\ft(U_{\scriptscriptstyle{S}})&\!\!\!:=& \!\!\!\mathcal{O}(S^7)[c_{\scriptscriptstyle{S}},c_{\scriptscriptstyle{S}}^{-1}] ~\big/ \langle z_1z_1^*+z_2z_2^*-c_{\scriptscriptstyle{S}}^{2} 
~,~c_{\scriptscriptstyle{S}}c_{\scriptscriptstyle{S}}^{-1}-1\rangle~,\nn\\[.6em]
\ft(U_{\scriptscriptstyle{NS}})&\!\!\!:=&\!\!\! \mathcal{O}(S^7)[c_{\scriptscriptstyle{N}}^{\pm 1},c_{\scriptscriptstyle{S}}^{\pm 1} ] 
~\big/ \langle z_3z_3^*+z_4z_4^*-c_{\scriptscriptstyle{N}}^{2} , ~~ z_1z_1^*+z_2z_2^*-c_{\scriptscriptstyle{S}}^{2}
~,~c_{\scriptscriptstyle{N}}c_{\scriptscriptstyle{N}}^{-1}-1
~,~c_{\scriptscriptstyle{S}}c_{\scriptscriptstyle{S}}^{-1}-1
\rangle~.\nn
\end{eqnarray}
For each open set $U \subset S^4$,  $\ft(U)$ is the algebra
of coordinate functions on $\pi^{-1} (U)\subset S^7$. Indeed 
extending the algebra $
\mathcal{O}(S^7)$ by the generator $c_{\scriptscriptstyle{N}}^{-1}$
corresponds, geometrically,  to  restricting to the subspace of $S^7$
of those points with $z_3z_3^*+z_4z_4^*$ never vanishing. Now recalling
relation (\ref{4sphere-coinv}) between coordinates on $S^7$ and on
$S^4$, we see that these are the points with $x\not= 1$, i.e., they are
the points of $\pi^{-1} (U_{\scriptscriptstyle{N}})$ (we use the same notation for the coordinate functions and the point coordinates). Conversely, enlarging the algebra with $c_{\scriptscriptstyle{N}}$ 
does not have a geometrical significance, but it is a pure algebraic
operation designed to add the square root of the positive real element
$z_3z_3^*+z_4z_4^*= c_{\scriptscriptstyle{N}}^{2}$ already belonging
to the algebra $\mathcal{O}(S^7)$. The same discussion is valid for
the elements $c_{\scriptscriptstyle{S}}^{ \pm 1}$, so that
$\ft(U_{\scriptscriptstyle{S}})$ and
$\ft(U_{\scriptscriptstyle{NS}})$ are coordinate algebras on $\pi^{-1}
(U_{\scriptscriptstyle{S}})$ and $\pi^{-1} (U_{\scriptscriptstyle{NS}})$ respectively.
\\

The assignment $\ft: U \mapsto \ft(U)$, with restriction morphisms 
given by the canonical inclusions $i_{\scriptscriptstyle{N}}: \ft(U_{\scriptscriptstyle{N}}) \hookrightarrow  \ft(U_{\scriptscriptstyle{NS}})$ and $
i_{\scriptscriptstyle{S}}: \ft(U_{\scriptscriptstyle{S}}) \hookrightarrow \ft(U_{\scriptscriptstyle{NS}}) $
defines\footnote{We also have the restriction morphisms 
 $\ft(U)\to\ft(\emptyset)$ that are canonical
 (and characterize the one-element algebra as the terminal object in the category of algebras).}
a sheaf of algebras over $S^4$.  
The algebra of global sections $\ft(S^4)$ is the pull-back 
\begin{equation}\label{restrA}
\ft(S^4) :=
\big\{(a_{\scriptscriptstyle{N}},a_{\scriptscriptstyle{S}}) \in \ft(U_{\scriptscriptstyle{N}}) \times \ft(U_{\scriptscriptstyle{S}}) ~|~  i_{\scriptscriptstyle{N}}(a_{\scriptscriptstyle{N}})=
i_{\scriptscriptstyle{S}}(a_{\scriptscriptstyle{S}})
\big \}\simeq \mathcal{O}(S^7)~.
\end{equation}
where $\times$ is the (categorical) product of $\ast$-algebras.
In the last equality we have observed that $\ft(S^4)$ is isomorphic to the
coordinate algebra $ \mathcal{O}(S^7)$ of $S^7$, indeed
$\ft(S^4)$ is the diagonal of
$\mathcal{O}(S^7)\times\mathcal{O}(S^7)\hookrightarrow\ft(U_{\scriptscriptstyle{N}}) \times \ft(U_{\scriptscriptstyle{S}}) $. \\[-.5em]
\bigskip

\noindent {\it The $H={\mathcal O}(SU(2))$-comodule structure on $\ft$ and the subsheaf  $\mathcal{B}=\mathcal{A}^{coH}$}\\
We recall from Example \ref{exCL} that $\ft(S^4)$ is a right $H={\mathcal O}(SU(2))$-comodule
algebra with right 
coaction given by (cf.\ \eqref{princ-coactSU2}):
\be\label{coazSU2}
\, u
 \longmapsto
u
\overset{.}{\otimes}
\begin{pmatrix}
 w_1 & -w_2^*
\vspace{2pt}
\\
w_2 & w_1^*\end{pmatrix} \quad , \qquad 
u:=
\begin{pmatrix}
z_1& z_2 & z_3& z_4
\vspace{2pt}
\\
-z_2^*  &
  z_1^* &
  -z_4^*
& z_3^*
\end{pmatrix}^t 
 \ee
(where  $\overset{.}{\otimes}$ denotes the composition of $\ot$ with
the matrix multiplication) and that the $H$-coinvariant subalgebra
$B=A^{coH}$ 
is generated by the elements
\be\label{4sphere-coinv2}
\alpha:= 2(z_1 z_3^* + z^*_2 z_4)~, \quad 
\beta:= 2(z_2 z_3^* - z^*_1 z_4)~, \quad 
x:= z_1 z_1^* + z_2 z_2^* - z_3 z_3^* -z_4 z_4^* ~, \quad 
\ee 
(and their $*$-conjugated $\alpha^*, \beta^*$,  with $x^*=x$) that satisfy
$
\alpha^* \alpha + \beta^* \beta + x^2=1 . 
$
Thus the subalgebra $B=A^{coH}$  of coinvariants is isomorphic to the algebra 
$\mathcal{O}(S^4)$ of coordinate  functions  on  $S^4$.
\sk
Since in particular the elements
\begin{equation}\label{cNcSx}
 c_{\scriptscriptstyle{N}}^{2}=z_3z_3^*+z_4z_4^*=\frac{1}{2}(1-x)~~,~~~~
c_{\scriptscriptstyle{S}}^{2}=z_1z_1^*+z_2z_2^*=\frac{1}{2}(1+x)
\end{equation}
 are
$H$-coinvariant, we then define,  for each open set $U \subset S^4$,   the $H$-comodule
structure on $\ft(U)$ via (\ref{coazSU2}) and by requiring $c_{\scriptscriptstyle{N}}$ and
$c_{\scriptscriptstyle{S}}$ to be $H$-coinvariant. In this way the
canonical inclusions
$\ft(S^4) \hookrightarrow  \ft(U_{\scriptscriptstyle{N}})$, $
\ft(S^4) \hookrightarrow \ft(U_{\scriptscriptstyle{S}}) $, $\ft(U_{\scriptscriptstyle{N}}) \stackrel{i_{\scriptscriptstyle{N}}}{\hookrightarrow}  \ft(U_{\scriptscriptstyle{NS}})$ and $
\ft(U_{\scriptscriptstyle{S}})
\stackrel{i_{\scriptscriptstyle{S}}}{\hookrightarrow} \ft(U_{\scriptscriptstyle{NS}}) $ are trivially right
$H$-comodule algebra inclusions. We have thus shown that $\ft$
is a sheaf of $H=\mathcal{O}(SU(2))$-comodule algebras.
The $H$-comodule structure on $\ft(S^4)$ is obtained from the pull-back
(\ref{restrA}), thought now as pull-back of $H$-comodule algebras; 
the isomorphism $\ft(S^4)\simeq
\mathcal{O}(S^7)$ then becomes an $H$-comodule algebra isomorphism.

\sk
The subalgebras of $H$-coinvariants are given by 
\be\label{coiAcoH} 
\ft(U_{\scriptscriptstyle{N}})^{coH}=
\ocn
\quad , \quad 
\ft(U_{\scriptscriptstyle{S}})^{coH}= \ocs
\quad , \quad
\ft(U_{\scriptscriptstyle{NS}})^{coH}=
\ocns
\ee
where $\ocn$
denotes the $*$-subalgebra of $\ft(U_{\scriptscriptstyle{N}})$
generated by the elements $\alpha,
\beta, x, c_{\scriptscriptstyle{N}}^{\pm 1}$, and similarly for the other basic open sets. Notice that
$\ft^{coH}(U)=\sft^{coH}(U)$ as defined in (\ref{HGtrivialsheaf}).
The algebra of global coinvariant sections is $\ft(S^4)^{coH}\simeq{\mathcal{O}}(S^4)$, and,
similarly to (\ref{restrA}), it is isomorphic  to the pull-back
\begin{equation}\label{Bsette}
\{(b_{\scriptscriptstyle{N}},b_{\scriptscriptstyle{S}}) \in \ft(U_{\scriptscriptstyle{N}})^{coH} \times \ft(U_{\scriptscriptstyle{S}})^{coH} ~|~  i_{\scriptscriptstyle{N}}(b_{\scriptscriptstyle{N}})=
i_{\scriptscriptstyle{S}}(b_{\scriptscriptstyle{S}}) \}~.
\end{equation}

In \S \ref{appsheafB}  we explicitly show that the subsheaf $\ft^{coH}$ of 
coinvariant elements (complemented by
$\ft({\emptyset})^{coH}=\ft(\emptyset)$) is that of coordinate functions
on the opens
$\emptyset, U_{\scriptscriptstyle{N}}, U_{\scriptscriptstyle{S}}, U_{\scriptscriptstyle{NS}}, S^4$. 
\bigskip

\noindent {\it The sheaf $\ft$ is a locally cleft sheaf of $H$-Hopf-Galois extensions}\\
The $H$-comodule algebra isomorphism $\ft(S^4)\simeq
\mathcal{O}(S^7)$ shows that the global sections $\ft(S^7)$ are an
 $H=\mathcal{O}(SU(2))$-Hopf-Galois extensions of the global
 coinvariant sections  $\ft(S^4)^{coH}\simeq \mathcal{O}(S^4)$.
Recalling the general theory, cf.\ \S \ref{sec:sheaves}, this shows that the
sheaf $\ft$  of $H$-comodule algebras
is a sheaf of $H$-Hopf-Galois extensions.

In order to prove that $\ft$ is  locally cleft  we consider the
open covering $\{U_{\scriptscriptstyle{N}},
U_{\scriptscriptstyle{S}}\}$ of $S^4$ and show that $\ft(U_{\scriptscriptstyle{N}})^{coH}\subseteq \ft (U_{\scriptscriptstyle{N}})$ and $\ft
(U_{\scriptscriptstyle{S}})^{coH}\subseteq \ft (U_{\scriptscriptstyle{S}})$ are
cleft extensions.

We first observe that the  matrix elements
\begin{align}
\begin{pmatrix}
w^{\scriptscriptstyle{N}}_1 & -(w^{\scriptscriptstyle{N}}_2)^*
\\[.4em]
w^{\scriptscriptstyle{N}}_2 & (w^{\scriptscriptstyle{N}}_1)^*
\end{pmatrix}:=c_{\scriptscriptstyle{N}}^{-1} \begin{pmatrix}
 z_3 & -  z_4^*
\\[.4em]
z_4&  z_3^* 
\end{pmatrix}~
\end{align}
generate a $*$-subalgebra of  $\ft (U_{\scriptscriptstyle{N}})$
isomorphic to  $H=\mathcal{O}(SU(2))$;
indeed this matrix has unit determinant since
$c_{\scriptscriptstyle{N}}^{-2} (z_3z_3^*+z_4z_4^*)  =1$.
Similarly  the matrix elements
\begin{align}
\begin{pmatrix}
w^{\scriptscriptstyle{S}}_1 & -(w^{\scriptscriptstyle{S}}_2)^*
\\[.4em]
w^{\scriptscriptstyle{S}}_2 & (w^{\scriptscriptstyle{S}}_1)^*
\end{pmatrix}:=c_{\scriptscriptstyle{S}}^{-1}\begin{pmatrix}
  z_1 &- z_2^*
\\[.4em]
 z_2& z_1^* 
\end{pmatrix}~,
\end{align}
generate an $H=\mathcal{O}(SU(2))$ $*$-subalgebra of  $\ft (U_{\scriptscriptstyle{N}})$.

By using the matrix $u$ in \eqref{coazSU2} and the matrix projector
$\textsf{P}=uu^*$ (whose entries are the generators of $S^4$, see \eqref{proj-inst}) we introduce  ``local trivialization maps'':$\;$\footnote{\label{f5}
Recall that 
 the Hopf bundle $\pi:S^7 \rightarrow S^4$ trivializes 
on the two charts $U_{\scriptscriptstyle{N}},
U_{\scriptscriptstyle{S}}$ with trivializations (we use the same notation for the coordinate functions and the point coordinates)
$$
\pi^{-1}(U_{\scriptscriptstyle{N}}) \longrightarrow U_{\scriptscriptstyle{N}} \times SU(2)~, ~~ 
{z}=(z_1,z_2,z_3,z_4) \longmapsto \left( \pi({z}), \frac{1}{(|z_3|^2+|z_4|^2)^{\frac{1}{2}}}\begin{pmatrix}
 z_3 & -z_4^*
\vspace{2pt}
\\
z_4 & z_3^*\end{pmatrix}
 \right)$$
and
$$
\pi^{-1}(U_{\scriptscriptstyle{S}}) \longrightarrow U_{\scriptscriptstyle{S}} \times SU(2)~, ~~ 
{z}=(z_1,z_2,z_3,z_4) \longmapsto \left( \pi({z}), \frac{1}{(|z_1|^2+|z_2|^2)^{\frac{1}{2}}}\begin{pmatrix}
 z_1 & -z_2^*
\vspace{2pt}
\\
z_2 & z_1^*\end{pmatrix}
 \right) ~~.$$
The datum of the transition functions characterizing the bundle is
contained in the trivialization maps. In the present case we have  just one intersection
$U_{\scriptscriptstyle{NS}},$ defining the transition function
$$g_{\scriptscriptstyle{NS}}: U_{\scriptscriptstyle{NS}}\longrightarrow SU(2)~,~~~\pi(z)\longmapsto 
{\frac{(|z_3|^2+|z_4|^2) ^{\frac{1}{2}}}{(|z_1|^2+|z_2|^2)^{\frac{1}{2}}}}\begin{pmatrix}
 z_1 & -z_2^*
\vspace{2pt}
\\
z_2 & z_1^*\end{pmatrix}
\begin{pmatrix}
 z_3 & -z_4^*
\vspace{2pt}
\\
z_4 & z_3^*\end{pmatrix}
^{-1}~,
$$
i.e.\ 
$$
(\alpha,\beta,x)\longmapsto 
\frac{1}{2} c_{\scriptscriptstyle{N}}^{-1}c_{\scriptscriptstyle{S}}^{-1}\begin{pmatrix}
 \alpha & -  \beta^*
\\[.4em]
\beta &  \alpha^* 
\end{pmatrix}~.
$$
} 
\begin{align}\label{loc-triv-P}
\!\!\!\!\Psi_{\scriptscriptstyle{N}}: \ft(U_{\scriptscriptstyle{N}})  &\longrightarrow ~\ft(U_{\scriptscriptstyle{N}})^{coH} \ot H
\qquad 
&\!\!\!\!\!\!\Psi_{\scriptscriptstyle{S}}: \ft(U_{\scriptscriptstyle{S}}) &\longrightarrow  ~\ft(U_{\scriptscriptstyle{S}})^{coH} \ot H
\\ \nn
 u &\longmapsto  c_{\scriptscriptstyle{N}}^{-1} ~\textsf{P}
\overset{.}{\otimes}
c_{\scriptscriptstyle{N}}^{-1}
\begin{pmatrix}
1 &0 
\\
0& 1 
\\
z_3 & -z_4^*
\\
z_4& z_3^* 
\end{pmatrix} 
~,
\qquad
&u& \longmapsto c_{\scriptscriptstyle{S}}^{-1} ~ \textsf{P}
\overset{.}{\otimes}
c_{\scriptscriptstyle{S}}^{-1} 
\begin{pmatrix}
z_1 & -z_2^*
\\
z_2& z_1^* 
\\
1 &0 
\\
0& 1 
\end{pmatrix}~,
\end{align}
(extended as $*$-algebra maps).

\begin{lem} \label{lem:triv}
The maps $\Psi_{\scriptscriptstyle{N}}$ and 
$\Psi_{\scriptscriptstyle{S}}$ in (\ref{loc-triv-P}) 
are well defined algebra morphisms and are isomorphisms of left
 $\ft(U_{\scriptscriptstyle{N}})^{coH}$-modules (respectively
 $\ft(U_{\scriptscriptstyle{S}})^{coH}$-modules) and  also of right
 $H$-comodule algebras, where 
$\ft(U_{\scriptscriptstyle{N}})^{coH} \ot H_{\scriptscriptstyle{N}}
$
and $\ft(U_{\scriptscriptstyle{S}})^{coH} \ot H_{\scriptscriptstyle{S}}  $
have $H$-coaction given by the coproduct, $\id\otimes\Delta$
(cf.\ \eqref{deltaVW}). 
Hence they structure $\ft(U_{\scriptscriptstyle{N}})$ and $\ft(U_{\scriptscriptstyle{S}})$
as cleft  Hopf-Galois extension.
\end{lem}
\begin{proof}
Explicitly, these local trivialization maps are given by
\begin{eqnarray}\label{zN}
\Psi_{\scriptscriptstyle{N}}: \ft(U_{\scriptscriptstyle{N}})  &\longrightarrow& \ft(U_{\scriptscriptstyle{N}})^{coH} \ot H
\\
z_1 &\longmapsto &z_1^{\scriptscriptstyle{N}}:= \frac{c_{\scriptscriptstyle{N}}^{-1}}{2}  (\alpha \ot w_1^{\scriptscriptstyle{N}} - \beta^* \ot w_2^{\scriptscriptstyle{N}}  ) \nn
\\ \nn
z_2 &\longmapsto&z_2^{\scriptscriptstyle{N}}:= \frac{c_{\scriptscriptstyle{N}}^{-1}}{2} (\beta \ot w_1^{\scriptscriptstyle{N}} + \alpha^* \ot w_2^{\scriptscriptstyle{N}})
\\ \nn
z_3 &\longmapsto&z_3^{\scriptscriptstyle{N}}:={c_{\scriptscriptstyle{N}}} \ot w_1^{\scriptscriptstyle{N}}
\\ \nn
z_ 4&\longmapsto&z_4^{\scriptscriptstyle{N}}:= {c_{\scriptscriptstyle{N}}}\ot w_2^{\scriptscriptstyle{N}}
\\
c_{\scriptscriptstyle{N}}^{\pm 1}  &\longmapsto&  c_{\scriptscriptstyle{N}}^{\pm 1}   \ot |w^N|^2 = c_{\scriptscriptstyle{N}}^{\pm 1}   \ot 1 \nn~,
\end{eqnarray}
and 
\begin{eqnarray}\label{zS}
\Psi_{\scriptscriptstyle{S}}: \ft(U_{\scriptscriptstyle{S}}) &\longrightarrow & \ft(U_{\scriptscriptstyle{S}})^{coH} \ot H
\\ \nn
z_ 1&\longmapsto&z_1^{\scriptscriptstyle{S}}:= {c_{\scriptscriptstyle{S}}} \ot w_1^{\scriptscriptstyle{S}}
\\ \nn
z_ 2&\longmapsto&z_2^{\scriptscriptstyle{S}}:= {c_{\scriptscriptstyle{S}}} \ot w_2^{\scriptscriptstyle{S}}
\\ \nn
z_ 3&\longmapsto&z_3^{\scriptscriptstyle{S}}:=\frac{c_{\scriptscriptstyle{S}}^{-1}}{2}  (  \alpha^* \ot w_1^{\scriptscriptstyle{S}} + \beta^* \ot w_2^{\scriptscriptstyle{S}})
\\ \nn
z_ 4&\longmapsto&z_4^{\scriptscriptstyle{S}}:= \frac{c_{\scriptscriptstyle{S}}^{-1}}{2} (- \beta \ot w_1^{\scriptscriptstyle{S}} + \alpha \ot w_2^{\scriptscriptstyle{S}})
\\
c_{\scriptscriptstyle{S}}^{\pm 1} &\longmapsto& c_{\scriptscriptstyle{S}}^{\pm 1}   \ot |w^{\scriptscriptstyle{S}}|^2 
= c_{\scriptscriptstyle{S}}^{\pm 1}   \ot 1\nn~.
\end{eqnarray}
Since
\begin{align}\label{loc-sect}
 z_1^{\scriptscriptstyle{N}} (z_3^{\scriptscriptstyle{N}})^* + (z_2^{\scriptscriptstyle{N}})^* z_4^{\scriptscriptstyle{N}}= \frac{1}{2}\alpha \ot 1 \quad , \qquad
z_2^{\scriptscriptstyle{N}} (z_3^{\scriptscriptstyle{N}})^* -
  (z_1^{\scriptscriptstyle{N}})^* z_4^{\scriptscriptstyle{N}}=
\frac{1}{2}\beta \ot 1  ~,\nn
\\[.3em]
z_1^{\scriptscriptstyle{N}} (z_1^{\scriptscriptstyle{N}})^* + z_2^{\scriptscriptstyle{N}} (z_2^{\scriptscriptstyle{N}})^* - z_3^{\scriptscriptstyle{N}} (z_3^{\scriptscriptstyle{N}})^* -z_4^{\scriptscriptstyle{N}} (z_4^{\scriptscriptstyle{N}})^* = (1-2 c_{\scriptscriptstyle{N}}^{2})=x \ot 1~,~~~~~~~~~
\end{align}
the elements $z_i^{\scriptscriptstyle{N}},
c_{\scriptscriptstyle{N}}^\pm$ generate
  $\ft(U_{\scriptscriptstyle{N}})^{coH} \ot H$.
Use of \eqref{cNcSx}  shows that 
$\Psi_{\scriptscriptstyle{N}}(\sum z_iz_i^*)=\sum
z_i^{\scriptscriptstyle{N}} (z_i^{\scriptscriptstyle{N}})^*=1$,
and
$\Psi_{\scriptscriptstyle{N}}(z_3z_3^*+z_4z_4^*)= 
z^{\scriptscriptstyle{N}}_3(z^{\scriptscriptstyle{N}}_3)^*+z^{\scriptscriptstyle{N}}_4(z_4^{\scriptscriptstyle{N}})^*
=c_{\scriptscriptstyle{N}}^{2}$ 
so that 
$\Psi_{\scriptscriptstyle{N}}$ is a well defined algebra map. It is a
one to one correspondence between generators and relations defining
$\ft(U_{\scriptscriptstyle{N}})$ and generators and relations defining
$\ft(U_{\scriptscriptstyle{N}})^{coH} \ot H$, hence it is a $*$-algebra isomorphism.
Identical expressions hold for the elements
$z_i^{\scriptscriptstyle{S}},c_{\scriptscriptstyle{S}}^{\pm 1}$, so that also
$\Psi_{\scriptscriptstyle{S}}$ is a $*$-algebra isomorphism.
It is evident from \eqref{coazSU2} that
$\Psi_{\scriptscriptstyle{N}}$ and $\Psi_{\scriptscriptstyle{S}}$ are
$H$-comodule maps and left $\ft(U_{\scriptscriptstyle{N}})^{coH}$,
respectively $\ft(U_{\scriptscriptstyle{S}})^{coH}$-module maps.
\end{proof}
Notice that as a corollary of this lemma  the sheaf $\ft$
and the sheaf $\sft$ of \S \ref{S4astwistedsheaf} are naturally isomorphic
 sheaves of  $H$-Hopf-Galois extensions on $S^4$.

\bigskip

\sk
\noindent {\it The $K$-comodule structure on $\ft$ and the locally cleft sheaf of $H$-Hopf-Galois extensions ${}_\sigma\ft$}\\
Recall from Example {\ref{exCL}} that $K=\mathcal{O} (\mathbb{T}^2) $ denotes the commutative $*$-Hopf algebra of coordinates on the torus $\mathbb{T}^2$.
For each open set 
$U$, the algebra $\ft(U)$ carries a $K$-comodule
structure where the left $K$-coaction is the $*$-algebra map defined on generators by (cf.\  \eqref{coazioneT-S7})
\be\label{coazT}
u \longmapsto \mathrm{diag}(t_1,t_1^*,t_2,t^*_2) \overset{.}{\otimes} u 
\quad, \quad  c_{\scriptscriptstyle{N}}^{\pm 1} \longmapsto 1 \ot c_{\scriptscriptstyle{N}}^{\pm 1} 
\quad, \quad  c_{\scriptscriptstyle{S}}^{\pm 1}  \longmapsto 1 \ot c_{\scriptscriptstyle{S}}^{\pm 1} 
\ee
(here  $\overset{.}{\otimes}$ denotes the composition of $\ot$ with the matrix multiplication).
 For each open
$U$ the $K$-coaction and the $H$-coaction  given in \eqref{coazSU2}
satisfy the compatibility condition \eqref{compatib} and 
thus $\ft(U)$ is a  $(K, H)$-bicomodule algebra.
Furthermore the restriction morphisms of $\ft$ are morphisms of
$(K,H)$-bicomodule algebras, so that  $\ft$ is a sheaf of $(K, H)$-bicomodule algebras.

We can now consider the 2-cocycle $\sigma$ in \eqref{cocycleT2} on $K$ and 
use it to deform the sheaf $\ft$. According to the general theory in
\S \ref{sec:sheaves},  for each open set $U \subseteq S^4$, 
 $\ft(U)$  is deformed into a $(K_\sigma,H)$-bicomodule algebra
 ${}_\sigma\ft(U)$ that is also an $H$-Hopf-Galois extension. The resulting
 sheaf of $H$-Hopf-Galois extensions ${}_\sigma\ft$ gives a
 sheaf-description of the Hopf bundle over $S^4_\theta$ addressed in
 Example \ref{exCL} because, since $\ft(S^4)\simeq\mathcal{O}(S^7)$ and 
$\ft(S^4)^{coH}\simeq\mathcal{O}(S^4)$ (cf.\ discussion after
(\ref{restrA}) and before (\ref{Bsette})), then
 ${}_\sigma\ft(S^4)\simeq\mathcal{O}(S^7_\theta)$ and  ${}_\sigma\ft(S^4)^{coH}\simeq\mathcal{O}(S^4_\theta)$.

\subsection{The spheres $S^4$ and $S^4_\theta$ as  ringed spaces}\label{appsheafB}

We show that the subsheaf of coinvariant elements $\sft^{coH}=\ft^{coH}$
introduced in \S  \ref{S4astwistedsheaf} (see also
(\ref{coiAcoH})) is that of the algebras (rings)
of coordinate functions on the opens $U_{\scriptscriptstyle{N}},
U_{\scriptscriptstyle{S}}, U_{\scriptscriptstyle{NS}}, S^4$, and correspondingly,
that the subsheaf of coinvariant elements ${}_\sigma \sft^{coH}$ arises
from the Moyal-Weyl algebra on $\mathbb{R}^4_\theta$.

\begin{lem}\label{lem:phiN}
Let $\mathscr{B}(U_{\scriptscriptstyle{N}})$ denote the commutative $*$-algebra 
generated by elements ${ x_1, x_2, x_1^*, x^*_2}$ together with $\rho_{\scriptscriptstyle{N}}^{\pm 1}$ satisfying
$\rho_{\scriptscriptstyle{N}}^{-2} (1+ x_1 x_1^*+x_2 x_2^*)=1$ and 
$\rho_{\scriptscriptstyle{N}}^{-1} \rho_{\scriptscriptstyle{N}}=1$. Let 
$\mathscr{B}(U_{\scriptscriptstyle{S}})$  denote
the commutative $*$-algebra generated by elements $ y_1, y_2, y_1^*, y^*_2$ together with  $\rho_{\scriptscriptstyle{S}}^{\pm1}$ satisfying $\rho_{\scriptscriptstyle{S}}^{-2} (1+y_1 y_1^*+y_2 y_2^*)=1$ and
$\rho_{\scriptscriptstyle{S}}^{-1}\rho_{\scriptscriptstyle{S}}=1$.
The maps
\be\label{piNS}
\begin{array}{rclrcl}
\phi_{\scriptscriptstyle{N}}: \mathscr{B}(U_{\scriptscriptstyle{N}}) &\longrightarrow&  
\sft(U_{\scriptscriptstyle{N}})^{coH}
\qquad 
&\phi_{\scriptscriptstyle{S}}: \mathscr{B}(U_{\scriptscriptstyle{S}}) &\longrightarrow&  \sft(U_{\scriptscriptstyle{S}})^{coH}
\\
 x_1 &\longmapsto&  \frac{1}{2}\alpha c_{\scriptscriptstyle{N}}^{-2} \; \quad \vspace{3pt}
& y_1 &\longmapsto&  \frac{1}{2} \alpha c_{\scriptscriptstyle{S}}^{-2} \; \quad 
\\
 x_2  &\longmapsto&  \frac{1}{2} \beta c_{\scriptscriptstyle{N}}^{-2} 
& y_2  &\longmapsto& \frac{1}{2} \beta c_{\scriptscriptstyle{S}}^{-2}
\\[.4em]
\rho_{\scriptscriptstyle{N}}^{\pm 1}&\longmapsto& c_{\scriptscriptstyle{N}}^{\mp 1} 
&\rho_{\scriptscriptstyle{S}}^{\pm 1} &\longmapsto& c_{\scriptscriptstyle{S}}^{\mp 1} 
\end{array}
\ee
(extended as $*$-algebra maps) define $*$-algebra isomorphisms.  
\end{lem}
\begin{proof}
The inverse maps are given by
\begin{eqnarray*}
\phi_{\scriptscriptstyle{N}}^{-1}: \quad
c_{\scriptscriptstyle{N}}^{\mp1} \longmapsto \rho_{\scriptscriptstyle{N}}^{\pm1}
\; ; \quad
\alpha \longmapsto 2 x_1 \rho_{\scriptscriptstyle{N}}^{-2} 
\; ; \quad 
\beta \longmapsto 2 x_2 \rho_{\scriptscriptstyle{N}}^{-2} 
\; ; \quad
\\
\phi_{\scriptscriptstyle{S}}^{-1}: \quad
c_{\scriptscriptstyle{S}}^{\mp1} \longmapsto \rho_{\scriptscriptstyle{S}}^{\pm1}
\; ; \quad
\alpha \longmapsto 2 y_1 \rho_{\scriptscriptstyle{S}}^{-2} 
~ ; \quad
\beta \longmapsto 2 y_2 \rho_{\scriptscriptstyle{S}}^{-2} 
~ . \quad
\end{eqnarray*}
By using  $(1-x)=2c_{\scriptscriptstyle{N}}^{2} $, valid in $\sft(U_{\scriptscriptstyle{N}})^{coH}$, it is  easy to show that the map $\phi_{\scriptscriptstyle{N}}$ is an algebra map, i.e.\
preserves the identity $\rho_{\scriptscriptstyle{N}}^{-2} (1+ x_1 x_1^*+x_2 x_2^*)=1$.
An analogous computation,  using again \eqref{cNcSx},   shows that $\phi_{\scriptscriptstyle{S}}$ is an algebra map.
\end{proof}
 
Because of this lemma the  algebras  $\sft(U_{\scriptscriptstyle{N}})^{coH}$  and 
$\sft(U_{\scriptscriptstyle{S}})^{coH}$ are interpreted as two (isomorphic) copies of the algebra of coordinate functions  on $\IR^4$. (Adding to the algebra generated by 
$ x_1, x_2, x_1^*, x^*_2$ the 
 generators 
$\rho_{\scriptscriptstyle{N}}^{\pm1}$ is geometrically ineffective,
similarly for 
$\rho_{\scriptscriptstyle{S}}^{\pm1}$).
Specifically, they  describe the algebras of coordinate functions on the open sets 
$U_{\scriptscriptstyle{N}},U_{\scriptscriptstyle{S}}$, 
obtained via stereographic projections
from the North and South poles of the $4$-sphere.\footnote{
Using the same notation for the coordinate functions and the point coordinates,
a point  $(\alpha, \alpha^*,\beta, \beta^*, x) \in  S^4$ maps via the stereographic projection 
from the North pole to the point  $(x_1,x_1^*,x_2,x_2^*) \in \mathbb{R}^4$ with 
coordinates $x_1= \alpha (1-x)^{-1}$, $x_2= \beta (1-x)^{-1}$.
While, via stereographic projection from the South pole, it projects to the point 
with coordinates  $y_1= \alpha (1+x)^{-1}$,  $y_2= \beta (1+x)^{-1}$.
The coordinate function $\rho_{\scriptscriptstyle{N}}^{-1}$  in $\mathscr{B}(U_{\scriptscriptstyle{N}})$, as $\rho_{\scriptscriptstyle{S}}^{-1}$ in $\mathscr{B}(U_{\scriptscriptstyle{S}})$,
has no geometrical significance since $1+ x_1 x_1^*+x_2 x_2^*$
has always a well defined and invertible square root (being $1+ x_1 x_1^*+x_2 x_2^*\geq 1$). Conversely,
 from $\alpha \alpha^* + \beta \beta^* +x^2=1$, 
 it follows that $r^2=x_1 x_1^*+x_2 x_2^*=(1+x)(1-x)^{-1} $ is defined and is non-zero 
when the point $(\alpha, \alpha^*,\beta, \beta^*, x)$ we started from 
belongs to $S^4 \backslash \{N, S\}$. 
The algebra extension of  $\mathscr{B}(U_{\scriptscriptstyle{N}})$  by $r^{-2}$ considered in Lemma \ref{lem:int} geometrically 
corresponds indeed to the restriction
to the points in the intersection of the charts
$U_{\scriptscriptstyle{N}}$ and $U_{\scriptscriptstyle{S}}$  of
$S^4$. 
} 

Similarly, the following lemma shows that the algebra  $\sft(U_{\scriptscriptstyle{NS}})^{coH}$
is that of coordinate functions on 
$\mathbb{R}^4$ minus the origin.
\begin{lem}\label{lem:int} 
We denote by $\mathscr{B}(U_{\scriptscriptstyle{NS}})$ 
the algebra extension  of
$\mathscr{B}(U_{\scriptscriptstyle{N}})$ by  central real elements $r^{\pm 1}$, satisfying $r r^{-1}=1$, $r^2:=x_1 x_1^*+x_2 x_2^*$.
The map $\phi_{\scriptscriptstyle{N}}$ in \eqref{piNS}
extends to an algebra isomorphism
$\phi_{\scriptscriptstyle{NS}}: \mathscr{B}(U_{\scriptscriptstyle{NS}}) \stackrel{\simeq}{\ra}  
\sft(U_{\scriptscriptstyle{NS}})^{coH}
$ by setting
$$ r^{-1} \longmapsto   c_{\scriptscriptstyle{N}}  c_{\scriptscriptstyle{S}}^{-1} 
\quad , \quad 
r \longmapsto   c_{\scriptscriptstyle{S}}  c_{\scriptscriptstyle{N}}^{-1} 
~.$$
\end{lem}
\begin{proof}
Observe that 
$\phi_{\scriptscriptstyle{N}}(r^2)=\frac{1}{2} (1 + x) c_{\scriptscriptstyle{N}}^{-2}$.  Then,
since
 $2 c_{\scriptscriptstyle{S}}^{2}=  (1+x)$ (see \eqref{cNcSx}),  we immediately conclude that  
$\phi_{\scriptscriptstyle{N}}(r^2)= \phi_{\scriptscriptstyle{NS}}(r^2)$.
For the inverse map set $\phi_{\scriptscriptstyle{NS}}^{-1}: ~
c_{\scriptscriptstyle{S}}^{\mp1} \mapsto \rho_{\scriptscriptstyle{N}}^{\pm1} r^{\mp1}$. 
\end{proof}
\medskip

The restriction maps characterizing the subsheaf of coinvariants $\mathscr{B}\simeq \sft^{coH}$ 
are  the composition of the canonical inclusions $i_{\scriptscriptstyle{N}}: \sft(U_{\scriptscriptstyle{N}})^{coH} \hookrightarrow  \sft(U_{\scriptscriptstyle{NS}})^{coH} $ and $
i_{\scriptscriptstyle{S}}: \sft(U_{\scriptscriptstyle{S}})^{coH}
\hookrightarrow \sft(U_{\scriptscriptstyle{NS}})^{coH} $ with the
 isomorphisms $\phi_{\scriptscriptstyle{N}}$,
 $\phi_{\scriptscriptstyle{S}}$ and $\phi_{\scriptscriptstyle{NS}}^{-1}$. 
These restriction maps are $*$-algebra homomorphisms and their
explicit expression  on the
generators reads
\begin{flalign}\label{restr}
r^{\mathscr{B}}_{\scriptscriptstyle{N,NS}}: \mathscr{B}(U_{\scriptscriptstyle{N}}) &\longrightarrow \mathscr{B}(U_{\scriptscriptstyle{NS}}) \, , \quad x_i \longmapsto x_i ~,\quad \rho_{\scriptscriptstyle{N}}^{\pm 1}\longmapsto  \rho_{\scriptscriptstyle{N}}^{\pm 1} ~,\\
\nn r^{\mathscr{B}}_{\scriptscriptstyle{S,NS}}: \mathscr{B}(U_{\scriptscriptstyle{S}}) &\longrightarrow \mathscr{B}(U_{\scriptscriptstyle{NS}}) \, , \quad y_i \longmapsto x_i r^{-2}~,
\quad
\rho_{\scriptscriptstyle{S}}^{\pm1} \longmapsto \rho_{\scriptscriptstyle{N}}^{\pm1} r^{\mp1}~.
\end{flalign}
The algebra of global sections $\mathscr{B}(S^4)$ is the pull-back 
\begin{equation}\label{restrB}
\mathscr{B}(S^4) =
\big\{(b_{\scriptscriptstyle{N}},b_{\scriptscriptstyle{S}}) \in \mathscr{B}(U_{\scriptscriptstyle{N}}) \times \mathscr{B}(U_{\scriptscriptstyle{S}}) ~|~  r^{\mathscr{B}}_{\scriptscriptstyle{N,NS}}(b_{\scriptscriptstyle{N}})=
r^{\mathscr{B}}_{\scriptscriptstyle{S,NS}}(b_{\scriptscriptstyle{S}})
\big \}~.
\end{equation}
The algebra $\mathscr{B}(S^4)$ is generated by the elements $(\rho_{\scriptscriptstyle{N}}^{-2} x_i, \rho_{\scriptscriptstyle{S}}^{-2} y_i)$, $(\rho_{\scriptscriptstyle{N}}^{-2} x_i^*, \rho_{\scriptscriptstyle{S}}^{-2} y_i^*)$, $i=1,2$ and
$(1-2 \rho_{\scriptscriptstyle{N}}^{-2}, 2 \rho_{\scriptscriptstyle{S}}^{-2} -1)$ and is a copy of the
coordinate algebra $\mathcal{O}(S^4)=\mathcal{O}(S^7)^{coH}.$ 
\sk

Using the isomorphisms of Lemma \ref{lem:phiN} and Lemma
\ref{lem:int}  it is immediate to induce from $\sft(U)^{coH}$ the
$K={\mathcal{O}}(\mathbb{T}^2)$-comodule structure on
$\mathscr{B}(U)$, ($U=U_{\scriptscriptstyle{N}},
U_{\scriptscriptstyle{S}}, U_{\scriptscriptstyle{NS}}$)
 and to see that the
restriction maps (\ref{restr}) are $K$-comodule maps.  Considering the twist
(\ref{cocycleT2}) on $K$ we then obtain
the noncommutative algebra
${}_\sigma\mathscr{B}(U_{\scriptscriptstyle{N}})$ that is the
(geometrically trivial
central extension via the real elements $\rho_{\scriptscriptstyle{N}}^\pm$
of the) coordinate algebra on $\mathbb{R}^4_\theta$; i.e.\ the
(exponential version of the Moyal-Weyl) algebra defined by the
commutation relations $x_1 {\mtcols}\ x_2=e^{-2\pi i
  \theta}x_2{\mtcols}\ x_1$. Similarly for
${}_\sigma\mathscr{B}(U_{\scriptscriptstyle{S}})$, and for
${}_\sigma\mathscr{B}(U_{\scriptscriptstyle{NS}})$ that is the
geometrically nontrivial central extension of
${}_\sigma\mathscr{B}(U_{\scriptscriptstyle{N}})$ via the real elements $r^{\pm1}$. 
These algebras and the restriction maps
$\Sigma(r^{\mathscr{B}}_{\scriptscriptstyle{N,NS}})=r^{\mathscr{B}}_{\scriptscriptstyle{N,NS}}$,
$\Sigma(r^{\mathscr{B}}_{\scriptscriptstyle{S,NS}})=r^{\mathscr{B}}_{\scriptscriptstyle{S,NS}}$,
define the sheaf ${}_\sigma\mathscr{B} $ of noncommutative coordinates
algebras over $S^4$, i.e.\ define $S^4_\theta$ as a ringed space.

\end{document}